\documentclass[10pt,reqno]{amsart}%,draft
\usepackage{amsmath,amsthm,mathrsfs,amsfonts,amssymb,amscd,euscript,overpic,color}

\newtheorem{theorem}{Theorem}[section]

\newtheorem{lemma}[theorem]{Lemma}

\newtheorem{corollary}[theorem]{Corollary}
\newtheorem{claim}[theorem]{Claim}
\newtheorem{proposition}[theorem]{Proposition}

\newtheorem{maintheorem}{Theorem}

\theoremstyle{definition}
\newtheorem{definition}[theorem]{Definition}
\newtheorem{example}[theorem]{Example}
\newtheorem*{example*}{Example}
\newtheorem{remark}[theorem]{Remark}

\newtheorem*{remark*}{Remark}

\newtheorem*{standinghypotheses*}{Standing hypotheses}
\newtheorem*{pinchinghypothesis*}{Pinching hypothesis}
\numberwithin{equation}{section}

\newcommand{\eqdef}{\stackrel{\scriptscriptstyle\rm def}{=}}

\let\OldMarginpar\marginpar
\renewcommand{\marginpar}[1]{\OldMarginpar{\tiny#1}}

\def\ext{{\rm ext}}
\def\w{{\rm w}}

\def\h{{\rm h}}
\def\s{{\rm s}}
\def\u{{\rm u}}
\def\uu{{\rm uu}}
\def\cu{{\rm cu}}
\def\cs{{\rm cs}}
\DeclareMathOperator{\Diff}{Diff}
\DeclareMathOperator{\diam}{diam}

\DeclareMathOperator{\interior}{int}

\DeclareMathOperator{\card}{card}
\DeclareMathOperator{\var}{var}

\def\Cuu{\mathcal{C}^{\rm uu}}

\def\va{\mathbf{a}}

\def\bi{\mathbf{i}}

\def\bN{\mathbb{N}}
\def\bT{\mathbb{T}}

\def\bZ{\mathbb{Z}}
\def\bR{\mathbb{R}}

\def\cC{\mathscr{C}}
\def\cD{\mathcal{D}}
\def\cF{\mathcal{F}}
\def\cT{\mathcal{T}}

\def\ch{\EuScript{H}}
\def\cS{\EuScript{S}}

\def\cs{\mathcal{S}}
\def\cU{\EuScript{U}}

\def\cW{\mathscr{W}}
\def\cM{\EuScript{M}}

\def\cL{\EuScript{L}}
\def\fB{\mathfrak{B}}
\def\fR{\EuScript{R}}

\DeclareMathSymbol{\varnothing}{\mathord}{AMSb}{"3F}

\author{}\address{}\email{}\urladdr{}

\begin{document}

\title[Hyperbolic graphs]{Hyperbolic graphs:\\ critical regularity and box dimension}

\author[L.~J.~D\'iaz]{L. J. D\'\i az}\address{Departamento de Matem\'atica PUC-Rio, Marqu\^es de S\~ao Vicente 225, G\'avea, Rio de Janeiro 22451-900, Brazil}\email{lodiaz@mat.puc-rio.br}

\author[K.~Gelfert]{K. Gelfert}\address{Instituto de Matem\'atica Universidade Federal do Rio de Janeiro, Av. Athos da Silveira Ramos 149, Cidade Universit\'aria - Ilha do Fund\~ao, Rio de Janeiro 21945-909,  Brazil}\email{gelfert@im.ufrj.br}

\author[M.~Gr\"oger]{M. Gr\"oger}\address{Friedrich-Schiller-University Jena, Institute of Mathematics, Ernst-Abbe-Platz 2, 07743 Jena, Germany}\email{maik.groeger@uni-jena.de}

\author[T.~J\"ager]{T. J\"ager}\address{Friedrich-Schiller-University Jena, Institute of Mathematics, Ernst-Abbe-Platz 2, 07743 Jena, Germany}\email{tobias.jaeger@uni-jena.de}

\begin{abstract}
	We study fractal properties of invariant graphs of hyperbolic and partially hyperbolic skew product diffeomorphisms in dimension three.  
	We describe the critical (either Lipschitz or at all scales H\"older continuous) regularity of such graphs. We provide a formula for their box dimension given in terms of appropriate pressure functions. We distinguish three scenarios according to the base dynamics: Anosov, one-dimensional attractor, or Cantor set.  A key ingredient for the dimension arguments in the latter case will be the presence of a so-called fibered blender.
\end{abstract}

\begin{thanks}{This research has been supported, in part, by CNE-FaperjE/26/202.977/2015 and CNPq research grants 302879/2015-3 and 302880/2015-1 and Universal 474406/2013-0 and 474211/2013-4 (Brazil) and EU Marie-Curie IRSES Brazilian-European partnership in Dynamical Systems FP7-PEOPLE-2012-IRSES 318999 BREUDS and DFG Emmy-Noether grant Ja 1721/2-1 and DFG Heisenberg grant Oe 538/6-1. This project is also part of the activities of the Scientific Network ``Skew product dynamics and multifractal analysis'' (DFG grant Oe 538/3-1). Further, LD and KG thank ICERM (USA) and CMUP (Portugal) for their hospitality and financial support.}\end{thanks}
\maketitle
%\tableofcontents

%----------------------------------------------------------------------------------------------------
\section{Introduction}\label{sec:introduction}
%----------------------------------------------------------------------------------------------------

We study regularity properties and box dimension of fractal graphs appearing
as attractors, repellers, or saddle-sets in skew product dynamics. 

Our motivation is two-fold. 
First, there is an intrinsic interest in the fractal
properties of such graphs, which is best exemplified by the well-known and
paradigmatic examples of Weierstrass functions. Based on dynamical methods,
recent advances have allowed to obtain a detailed understanding of
their fractal structure including their Hausdorff dimension (thus
solving a long-standing conjecture) \cite{BarBarRom:14,Kel:14,She:}.

Second, there is a general motivation for these endeavors.
The investigation of fractal attractors, repellers, horseshoes, and other
types of hyperbolic sets has been a major driving force for many important
developments in ergodic theory and its interfaces with mathematical physics and
fractal geometry (see, for instance, \cite{LedYou:85,Pes:97,Fal:03} for more information).
Thereby, the situation is fairly well understood for
two-dimensional hyperbolic systems (see~\cite{McCMan:83,Tak:88, PalVia:88} and Theorem~\ref{teo:seminal} below), which is essentially a conformal setting and comparable to the study of conformal repellers (see \cite{PrzUrb:10}).
However, extending the theory to higher-dimensional and genuinely
nonconformal situations is well known to be difficult, and there exist
only few and  specific results in this direction (see, for example,~\cite{KapMalYor:84,Fal:94,SimSol:99,HasSch:04} and more details in Remark~\ref{rem:dimensions}). 
Amongst the different phenomena that complicate matters are:
\begin{itemize}
\item The possible loss of equality between Hausdorff and box  dimensions, 
\item both dimensions may not vary continuously with the dynamics.
 %the lack of the continuous dependence of the dimensions of hyperbolic invariant sets on the dynamics.
\end{itemize}

A natural, and quite common, approach to proceed is to study gradually more complex (e.g.~higher-dimensional) systems. We proceed by studying  graphs in three-dimensional  skew product systems 
$$
T\colon\Xi\times\bR\to\Xi\times\bR, \quad T(\xi,x)=(\tau(\xi),T_\xi(x)),
$$
 with hyperbolic surface diffeomorphisms, or their restrictions to basic pieces $\tau\colon\Xi\to\Xi$, in the base,  building on previous results in \cite{KapMalYor:84,Bed:89,HadNicWal:02}. Summarizing our main results, except in a nongeneric case when the graph is Lipschitz, its box dimension is given by $d^\s+d$, where $d^\s$ is the dimension of stable slices of $\Xi$ and where $d$ is determined as the unique solution of the pressure equation
\[
	P_{\tau|_\Xi}(\varphi^\cu+(d-1)\varphi^\u)=0.
\]
Here $\varphi^\cu,\varphi^\u$ are appropriately defined geometric potentials taking into account the expansion rates in the fiber center unstable and the strong unstable directions, respectively.
The above formula will be established in three scenarios (Anosov in the base, one-dimensional attractors in the base, and fibered blenders).  
%, we taking a new class of fractal graphs into consideration.
These results can be viewed as a natural step to address the corresponding technical and conceptual problems in a nontrivial, but still accessible setting.
Thereby, we focus on the box dimension as the most accessible quantity in a first instance.
Although eventually our approach could be instrumental for describing finer fractal properties like the Hausdorff dimension or carrying out a multifractal analysis as well, which is beyond the purposes of this paper. 

%-----------------------------------------------------------------------------------------------------------------------
\subsection{Previous results on basic sets of surface diffeomorphisms}
%-----------------------------------------------------------------------------------------------------------------------
Before stating our first main result, let us provide more details on what is known in the two-dimensional case. 
Let $\tau\colon M\to M$ be a $C^{1+\alpha}$ surface diffeomorphism.
Recall that a set $\Xi\subset M$ is \emph{basic} if it is compact, invariant, \emph{locally maximal} in the sense that there is an open neighborhood $U$ of $\Xi$ such that $\Xi=\bigcap_{k\in\bZ}\tau^k(U)$, topologically mixing, and \emph{hyperbolic} in the sense that there exist a $d\tau$-invariant splitting $F^\s\oplus F^\u=T_\Xi M$ and numbers $0<\mu<1<\kappa$ such that for every $\xi\in\Xi$ 
\[%\begin{equation}\label{eq:hypest}
	\lVert d\tau|_{F^\s_\xi}\rVert\le \mu
	\quad\text{ and }\quad
	\kappa\le \lVert d\tau|_{F^\u_\xi}\rVert
\]%\end{equation}
(up to an equivalent change of metric),
where $d\tau|_{F^\s_\xi}$ and $d\tau|_{F^\u_\xi}$ denote the derivative of $\tau$ at $\xi$ in the stable and unstable directions, respectively.
%In dimension two, the structure of basic sets is fairly well understood. 
Further, recall that basic sets have a (local) product structure, that is, they can locally be described as products of representative stable and unstable slices, given by the intersection of $\Xi$ with the local stable and unstable manifolds, respectively (see~\cite{KatHas:95}).
In dimension two, their Hausdorff and box dimensions 
coincide and are given by the following classical Bowen-Ruelle type formula which is a compilation of results in~\cite{McCMan:83,Tak:88, PalVia:88}.
Consider for a basic set $\Xi\subset M$ the functions $\varphi^\s,\varphi^\u\colon\Xi\to\bR$ (also called \emph{potentials}) 
\begin{equation}\label{eq:defbasicpotentialvarphius}
	\varphi^\s(\xi)\eqdef \log\,\lVert d\tau|_{F_\xi^\s}\rVert,\quad
		\varphi^\u(\xi)\eqdef - \log\,\lVert d\tau|_{F_\xi^\u}\rVert.
\end{equation}	
%The potential functions $\varphi^\s,\varphi^\u\colon\Xi\to\bR$ measure the local contraction and expansion rates in the direction of local stable and unstable manifolds and are defined in~\eqref{eq:defbasicpotentialvarphius}. 
We denote by $P_{\tau|_\Xi}(\psi)$ the topological pressure of a potential $\psi\colon \Xi\to\bR$ (with respect to $\tau|_\Xi$) (see Section~\ref{sec:diment} for more details). 
Further, $\cW^\s_{\rm loc}(\xi,\tau)$ and $\cW^\u_{\rm loc}(\xi,\tau)$ denote the \emph{local stable} and the \emph{local unstable manifold} of $\xi$ (with respect to $\tau$), respectively (see Section~\ref{sec:2manif} for more details). 
Last, denote by $\dim_{\rm H}(E)$ the \emph{Hausdorff dimension} and by $\dim_{\rm B}(E)$ the \emph{box dimension} of a totally bounded subset $E$ in a metric space.
In general, we have $\dim_{\rm H}(E)\leq\dim_{\rm B}(E)$.
We recall the definition of box dimension and some properties in Section \ref{sec:defdim}; further information can be found in \cite{Fal:97}.

\begin{theorem}[\cite{McCMan:83,Tak:88, PalVia:88}]\label{teo:seminal}
	Consider a basic set $\Xi\subset M$ of a $C^{1+\alpha}$ surface diffeomorphism $\tau\colon M\to M$.
Let $d^\u$ and $d^\s$ be the unique real numbers for which 
\begin{equation}\label{eq:localdimension}
	P_{\tau|_\Xi}(d^\u\varphi^\u)=0 = P_{\tau|_\Xi}(d^\s\varphi^\s).
\end{equation}
Then for every $\xi\in\Xi$ we have
\begin{equation}\label{eq:localslicesMcCMan}
	\dim (\Xi\cap \cW^\u_{\rm loc}(\xi,\tau)) = d^\u\quad\textnormal{and}\quad
	\dim (\Xi\cap \cW^\s_{\rm loc}(\xi,\tau)) = d^\s,
\end{equation}
where $\dim$ stands either for $\dim_{\rm H}$ or $\dim_{\rm B}$.
Moreover, we have
\[
	\dim_{\rm H}( \Xi) = \dim_{\rm B}(\Xi) 
	= d^\s+d^\u.
\]
\end{theorem}

\begin{remark}%\label{rem:dimensions0}
Formulas~\eqref{eq:localslicesMcCMan} were derived for the Hausdorff dimension in \cite{McCMan:83}.
That Hausdorff and box dimension coincide was shown in \cite{Tak:88} for $C^2$ diffeomorphisms and in \cite{PalVia:88} as stated above (in fact, \cite{PalVia:88} assumes $C^1$ only). 
 %Under the assumptions of Theorem~\ref{teo:seminal} that $\tau$ is a \emph{surface} diffeomorphism the local product structure of $\Xi$ is a Lipschitz homeomorphism with Lipschitz continuous inverse and thus, also using the coincidence between Hausdorff and box dimension (see~\cite{Fal:97}), the dimension is given by adding the dimensions in the stable and unstable directions. 
To infer that the Hausdorff dimension of the (local) product is the sum of the dimensions of the
intersections in \eqref{eq:localslicesMcCMan} is conditioned to the fact that Hausdorff and box dimension coincide (see~\cite{Fal:97}). It requires the regularity of the stable/unstable holonomies, too. Yet, for
hyperbolic surface diffeomorphisms these holonomies are always bi-Lipschitz.
In \cite{PalVia:88}, the authors also establish the continuous dependence of the dimensions on the diffeomorphism.
\end{remark}

\begin{remark}\label{rem:dimensions}
In general, as already mentioned, in higher dimensions the above statements do not remain valid. 
For example, Hausdorff  and box dimension do not always coincide (confer the paradigmatic example in Remark~\ref{rem:coincidenceHausBox}, see also~\cite{PrzUrb:89,PolWei:94}).
Further, Hausdorff and box dimension may not vary continuously with the dynamics (see~\cite{BonDiaVia:95} and Remark~\ref{rem:disCantorbase}). 
Moreover, in general it is a difficult task to verify whether the dimensions of stable/unstable slices are constant (see~\cite{HasSch:04} for an investigation of the (three-dimensional and hyperbolic) solenoid). 
From a more technical point of view, in (non)conformal hyperbolic dynamics the study of dimensions is often based on a Markov partition and done by efficient coverings of cylinder sets. Notice that in a nonconformal setting, contrary to the conformal one, cylinder sets can be strongly distorted in directions of stronger contraction/expansion rates. This usually leads to a loss of distortion control of potential functions (see~\cite{Fal:94} for a rigorous treatment of nonconformal repellers assuming additionally a so-called bunching condition and~\cite{ManSim:07} for a discussion of counterexamples).
Last, in a higher-dimensional setting in general stable/unstable holonomies are not bi-Lipschitz but only H\"older continuous (see Section~\ref{sec:locpro} for further discussion), hence one cannot conclude about the dimensions of (local) products of slices.
\end{remark}

%-----------------------------------------------------------------------------------------------------------------------
\subsection{Setting}\label{ss:setttting}
%-----------------------------------------------------------------------------------------------------------------------

Unless stated otherwise, we will always assume that $\tau\colon M\to M$ is a $C^{1+\alpha}$ diffeomorphism on a Riemannian surface $M$ and that $T\colon M\times \bR\to M\times
\bR$ is a $C^{1+\alpha}$ diffeomorphism with skew product structure
\begin{equation} \label{e.skew-product-structure}
	T(\xi,x)  =  (\tau(\xi),T_\xi(x)) .
\end{equation}
Suppose that $\Xi\subset M$ is a basic set (with respect to $\tau$).
% which projects to $\Xi$.
Moreover, assume that $T$ is \emph{fiberwise expanding (over $\Xi$)}, that is,  
\[
	\inf_{(\xi,x)\in\Xi\times\bR} \lvert T'_\xi(x)\rvert>1.
\]	 
Then there exists a unique graph $\Phi\colon\Xi\to\bR$ that
is invariant under the dynamics in the sense that 
\begin{equation} \label{e.invariant-graph}
  T_\xi(\Phi(\xi)) \ = \ \Phi(\tau(\xi)) 
\end{equation}
holds for all $\xi\in \Xi$ (see \cite{HirPug:70}). In our setting, $\Phi$ is
the global repeller (over $\Xi$)%
\footnote{Note that we do not distinguish here between the function $\Phi\colon\Xi\to\bR$
	and the associated point set $\{(\xi,\Phi(\xi))\colon\xi\in\Xi\}$,
	that is, we identify the function with its graph. This is consistent with the
	formal definition of a function as a special type of a relation.
}%
\textsuperscript{,}\footnote{For technical reasons, we only consider expansion in
	the fibers. The case of contracting fibers would just amount to use the
	inverse of a fiberwise expanding system and would not affect the existence
	of a unique invariant graph $\Phi$ (which is then an attractor).} 
in the sense that all initial conditions $(\xi,x)\in\Xi\times\bR$ converge exponentially fast to $\Phi$ under iteration by {the inverse of $T$.

\begin{standinghypotheses*}
Assume that there are numbers
\begin{equation}\label{eq:constants}
%	0<\kappa_s<\kappa_w<\lambda_s<\lambda_w<1<\mu_w<\mu_s
	0<\mu_\s\le\mu_\w<1<\lambda_\w\le\lambda_\s<\kappa_\w\le\kappa_\s
\end{equation}
such that  for all $\xi\in \Xi$ we have
\begin{equation}\label{eq:basehyp}
	\mu_\s\le \lVert d\tau|_{F^\s_\xi}\rVert\le \mu_\w ,
	\quad 
	\lambda_\w\le \lvert T_\xi'\circ\Phi\rvert\le \lambda_\s ,
	\quad 
	\kappa_\w\le\lVert d\tau|_{F^\u_\xi}\rVert\le \kappa_\s .
\end{equation}
\end{standinghypotheses*}

\begin{remark}%\label{rem:partialhyp}
Conditions (\ref{eq:constants}) and (\ref{eq:basehyp}) imply that there exist three one-dim\-en\-sion\-al invariant bundles $E^\s,E^\cu,E^\uu$ (we refrain from giving their precise definitions). Using these bundles, we have that $\Phi$ (with respect to $T$) is at the same time hyperbolic (considering the   splitting into the two bundles $E^s$ and $E^\cu\oplus E^\uu$) and partially hyperbolic%
\footnote{This definition refers to what is  also known as \emph{absolute partial hyperbolicity} (see~\cite{HasPes:06}  or~\cite[Appendix B]{BonDiaVia:05}). There exist refined versions of partial hyperbolicity which require such type of norm separation satisfied only pointwise. } 
(considering the splitting into the three bundles $E^\s$, $E^\cu$, and $E^\uu$).  
This allows in particular to define the stable, unstable, center unstable, and strong unstable foliations of $T$ (see Section~\ref{sec:2manif}), which play a key role in all proofs. In our case, the center unstable foliation is naturally given by the fibers $\{\xi\}\times\bR$ of the skew product. 
\end{remark}
Similar to \eqref{eq:defbasicpotentialvarphius}, we consider the additional potential 
$\varphi^\cu\colon\Xi\to\bR$ defined by
\begin{equation}\label{e.potentials}
	\varphi^{\cu}(\xi)=-\log |T_\xi'(\Phi(\xi))| .
	%\quad \textrm{ and } \quad
	%\varphi^\u(\xi) = -\log \|d\tau|_{F^\u_\xi}\|.
\end{equation}
%\sout{ In the situation of Theorem~\ref{t.anosov}, the restriction of $\Phi_T$ to a basic piece $\Xi\subseteq M$ is itself a basic piece of $T$. Hence, the study of such sets can be seen as a first step towards dealing with these problems. }
Finally, we assume one additional technical hypothesis to simplify our exposition.

\begin{pinchinghypothesis*}
Suppose that $T$ is $C^2$ and satisfies
	\begin{equation}\label{eq:constants-holonomy}
		\kappa_\s\mu_\w\leq\lambda_\w.
	\end{equation}
\end{pinchinghypothesis*}

 \begin{remark}\label{r.pinchingholonomies}
The Pinching hypothesis is only required  to ensure that the holonomy map along the invariant manifolds of $T$ is bi-H\"older continuous with a H\"older constant arbitrarily close to $1$. See Section~\ref{sec:locpro}  for further details and discussion. Note that  we have  $\kappa_\s\mu_\w\leq \lambda_\w$ automatically when $\kappa_\s=\mu_\w^{-1}$, independently of $\lambda_\w$, as for example in the affine Anosov case in Section~\ref{sec:ex2}.
This allows us to compute the box dimension of
  $\Phi$ from the box dimensions of its restriction to local stable/unstable manifolds
  of the map $\tau$ in the base (which will be provided in Section~\ref{newsec:proofss}).
\end{remark}

%-----------------------------------------------------------------------------------------------------------------------
\subsection{Anosov maps in the base}
%-----------------------------------------------------------------------------------------------------------------------

Let us first consider the simplest case of $\tau$ being  an Anosov diffeomorphism and $\Xi=M$ the trivial basic piece.

\begin{maintheorem}\label{t.anosov}
	Let $T$ be a three-dimensional skew product diffeomorphism satisfying the Standing and Pinching hypotheses.
Assume that $\Xi=M$ and that $\tau\colon M\to M$ is an 
	Anosov diffeomorphism. Then
	\begin{itemize}
	\item
	either $\Phi$ is Lipschitz continuous and its box dimension is two, 
	\item
	or
	$\Phi$ is not $\gamma$-H\"older continuous at any point for any $\gamma>\log\lambda_\s/\log\kappa_\w$ and
	its box dimension is given by $\dim_{\rm B}(\Phi)= 1+d$, where $d$ is the  unique number such that	\begin{equation}\label{e.pressure_equality}
		P_{\tau}(\varphi^{\cu}+(d-1)\varphi^\u)=0 .
	\end{equation}
	\end{itemize}
\end{maintheorem}

We note that the particular case of skew product systems with affine fiber maps
and linear torus automorphisms in the base (see Section~\ref{sec:ex2}) is
already covered by the results of \cite{KapMalYor:84} using Fourier analysis. A
more general result that includes Theorem~\ref{t.anosov} has been announced in
\cite{Wal:07}. However, due to a serious flaw in the argument given in that
paper, a complete proof for the statement in \cite{Wal:07} does not exist so
far. We will discuss this issue in detail in Section~\ref{sub.hyperbolic_base} below.

\begin{remark}
%
%(a) We refer to Section~\ref{Examples} for some explicit examples that
%     demonstrate the application of the previous result as well as for the following ones.
 The fact that $\Phi$ is either Lipschitz or has a maximal H\"older
  exponent    is often referred to as {\em critical
    regularity} and has already been proven in our setting in
  \cite{HadNicWal:02}. We reproduce this result here both for the convenience
  of the reader and due to the fact that this will be a byproduct of the methods
  for computing the box dimension, and we have to introduce the respective
  concepts and estimates anyway.
\end{remark}

Theorem~\ref{t.anosov} treats the case of invariant graphs defined over the whole
mani\-fold $M$. In the broader context of the geometry of hyperbolic sets, it is
natural to consider also the restriction of such graphs to Cantor basic sets of $\tau$
in the base. However, before doing so, we consider  an intermediate case.

%-----------------------------------------------------------------------------------------------------------------------
\subsection{One-dimensional hyperbolic attractors in the base}%\label{ss.onedimatt}
%-----------------------------------------------------------------------------------------------------------------------

Following the terminology coined in the 70s, we say that a basic set
$\Xi$ 
 is a \emph{one-dimensional hyperbolic attractor}  of $\tau$ if it is a hyperbolic attractor
 (i.e., $\Xi=\bigcap_{k\in\bN}\tau^k(U)$ for some neighborhood  $U$ of $\Xi$) 
and locally  homeomorphic  to a direct product of a Cantor set  and an interval (and hence the ``intervals'' are contained in the unstable manifolds of the attractor). Important examples of these attractors are the
derived from Anosov (DA) and Plykin attractors%
\footnote{The construction of the \emph{derived from Anosov} (\emph{DA}) \emph{diffeomorphism} of $\bT^2$ by Smale in \cite{Sma:67} 
starts with a linear hyperbolic automorphism of $\bT^2$ and considers a local perturbation introducing a repeller in the dynamics in such a way that the resulting diffeomorphism is axiom A and has 
two basic sets:  the repelling fixed point and a one-dimensional attractor.
A \emph{DA attractor} is any hyperbolic attractor which is conjugate to the attractor of some DA diffeomorphism. The construction of   \emph{Plykin attractors} is more involved and a key fact is that
they are defined on a two-dimensional disk (which hence can be embedded into any surface), see for instance~\cite{Ply:74}.}.
 
\begin{maintheorem}\label{the:one-dimensional}
Let $T$ be a three-dimensional skew product diffeomorphism satisfying the Standing and Pinching hypotheses.
Assume that the set $\Xi\subset M$ is a one-dimensional attractor  of $\tau$. Then 
\begin{itemize}
\item
either $\Phi$ is Lipschitz continuous and its box dimension is given by $\dim_{\rm B}(\Phi)= d^\s+1$, where $d^\s$ is as in~\eqref{eq:localdimension}, 
\item
or $\Phi$ is not $\gamma$-H\"older continuous at any point for any $\gamma>\log\lambda_\s/\log\kappa_\w$ and the box dimension of $\Phi$ is given by $\dim_{\rm B}(\Phi)=d^\s+d$, where $d$ is the  unique number such that	
\[%begin{equation}%\label{e.pressure_hypolicattractor}
	P_{\tau|_\Xi}(\varphi^{\cu}+(d-1)\varphi^\u)=0 .
\]%end{equation}
\end{itemize}
\end{maintheorem}

%-----------------------------------------------------------------------------------------------------------------------
\subsection{Hyperbolic Cantor sets in the base} \label{sub.hyperbolic_base}
%-----------------------------------------------------------------------------------------------------------------------

For the next result we introduce an additional hypothesis that we call
\emph{fibered blender with the germ property} that we will discuss below.
First, let us observe that blenders appear in a very natural and ample form in
our setting and that they form an open class of examples (see
Proposition~\ref{pro:persistent}).  An informal discussion can be found
in~\cite{BonCroDiaWil:16}.  Naively, a \emph{blender} is a type of horseshoe
which ``geometrically'' behaves like something ``bigger'' than a usual
horseshoe. We provide details in Section~\ref{sec:Blenders} and give a
representative example in Section~\ref{sec:ex1}. In our particular setting, a
blender guarantees that the invariant graph appears as if it would have a
``two-dimensional stable set'' (instead of just a one-dimensional stable set by
assumption). In rough terms, when projecting onto fibers there are
superpositions at all levels in the sense that in any local unstable manifold,
the projection of the graph along strong unstable leaves onto a fiber always
results in a nontrivial interval.  Let us observe that a rather different
approach to the construction of ``blenders" is considered in \cite{MorSil:}
starting from hyperbolic sets (in dimension three or higher) whose fractal
dimension is sufficiently large. This construction relies on the notion of a
{\emph{compact recurrent set}} (see \cite{MorYoc:01}) which is a covering like
property with the same flavor as the germ property.

\begin{maintheorem}\label{the:1}
Let $T$ be a three-dimensional skew product diffeomorphism satisfying the Standing and Pinching hypotheses.
Assume that $\Xi\subset M$ is a Cantor set and that $\Phi$ is a fibered  blender with the germ property. 
%Then for every $X=(\xi,x)\in\Lambda$ we have
%\begin{equation}\label{e.stable_dimension}
%	\dim_{\rm H}(\Lambda \cap \cW^\s_{\rm loc}(X,T)) = \dim_{\rm
%          B}(\Lambda\cap \cW^\s_{\rm loc}(X,T)) = d^\s,
%\end{equation}
Then $\Phi$ is not $\gamma$-H\"older continuous at any point for any $\gamma>\log\lambda_\s/\log\kappa_\w$  and its box dimension is given by $\dim_{\rm B}(\Phi) =d^\s+d$, where $d^\s$ is as in~\eqref{eq:localdimension} and
%\begin{equation}\label{e.unstable_dimension}
%	\underline\dim_{\rm B} ( \Lambda\cap \cW^\u_{\rm loc}(X,T))
%	= \overline\dim_{\rm B} ( \Lambda\cap \cW^\u_{\rm loc}(X,T))
%	= d^u,
%\end{equation}
$d$ is the  unique number such that	
	\begin{equation}\label{e.pressure_basicpiece}
		P_{\tau|_{\Xi}}(\varphi^{\cu}+(d-1)\varphi^\u)=0.
	\end{equation}
%Moreover, we have that 
%\begin{equation}\label{e.dimension_formular_basicpiece}
%\dim_{\rm B}\Lambda = d^u+d^s.
%\end{equation}
\end{maintheorem}

In the Cantor case there is one simple fact that is important to understand, namely, that not all cases can be covered by a single dimension formula. Instead, at least
two different regimes have to be distinguished, depending on the box dimension
of $\Xi$ and the parameters in (\ref{eq:constants}). 
The reason for this is the following  observation about two elementary upper bounds for the box dimension (see~\cite{Fal:97}, compare also~\cite[Section 4]{PrzUrb:89}).
%\footnote{If $\gamma\in(0,1)$ and $\Xi$ is one-dimensional, then~\cite[Section 5]{PrzUrb:89} derives some elementary estimates for the Hausdorff and box dimension, in particular in this case $\dim_{\rm B}(\Phi)\ge\dim_{\rm H}(\Phi)>1$.} 
 Given a $\gamma$-H\"older continuous function $\Phi\colon\Xi\to\bR$ on a metric space $\Xi$, on the one hand H\"older continuity implies
\begin{equation}\label{eq:opt2a}
	\dim_{\rm B}(\Phi)\leq \frac{\dim_{\rm B}(\Xi)}{\gamma} 
	=: D_1(\gamma) .
\end{equation}
On the other hand a covering argument gives (compare also the proof of the first claim in Proposition~\ref{pro:localdimension-unstable})
\begin{equation}\label{eq:opt2}
	\dim_{\rm B}(\Phi) \leq \dim_{\rm B}(\Xi)+1-\gamma =: D_2(\gamma).
\end{equation}
If $\dim_{\rm B}(\Xi)\geq 1$, then $D_2(\gamma)\leq D_1(\gamma)$ for all
$\gamma\in[0,1]$, so that the first bound does not play any role. When
$\dim_{\rm B}(\Xi)<1$, however, then there is an interval  $(d,1)\subset[0,1]$ such
that $D_1(\gamma)<D_2(\gamma)$ for all $\gamma\in (d,1)$ (see Figure~\ref{fig.comparison}). In this
case, the box dimension of $\Phi$ can obviously not be equal to $D_2(\gamma)$. In the
context of Theorem~\ref{the:1}, this implies that the box dimension cannot be
determined by the analogue of \eqref{e.pressure_basicpiece} in all cases. An
explicit example for this will be discussed in  Section~\ref{sec:ex1} (see Remark~\ref{rem:differences}).

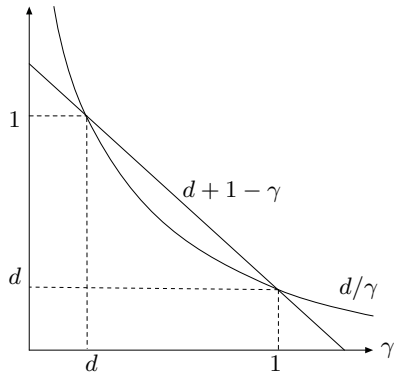
\begin{figure}
 \begin{overpic}[scale=.3]{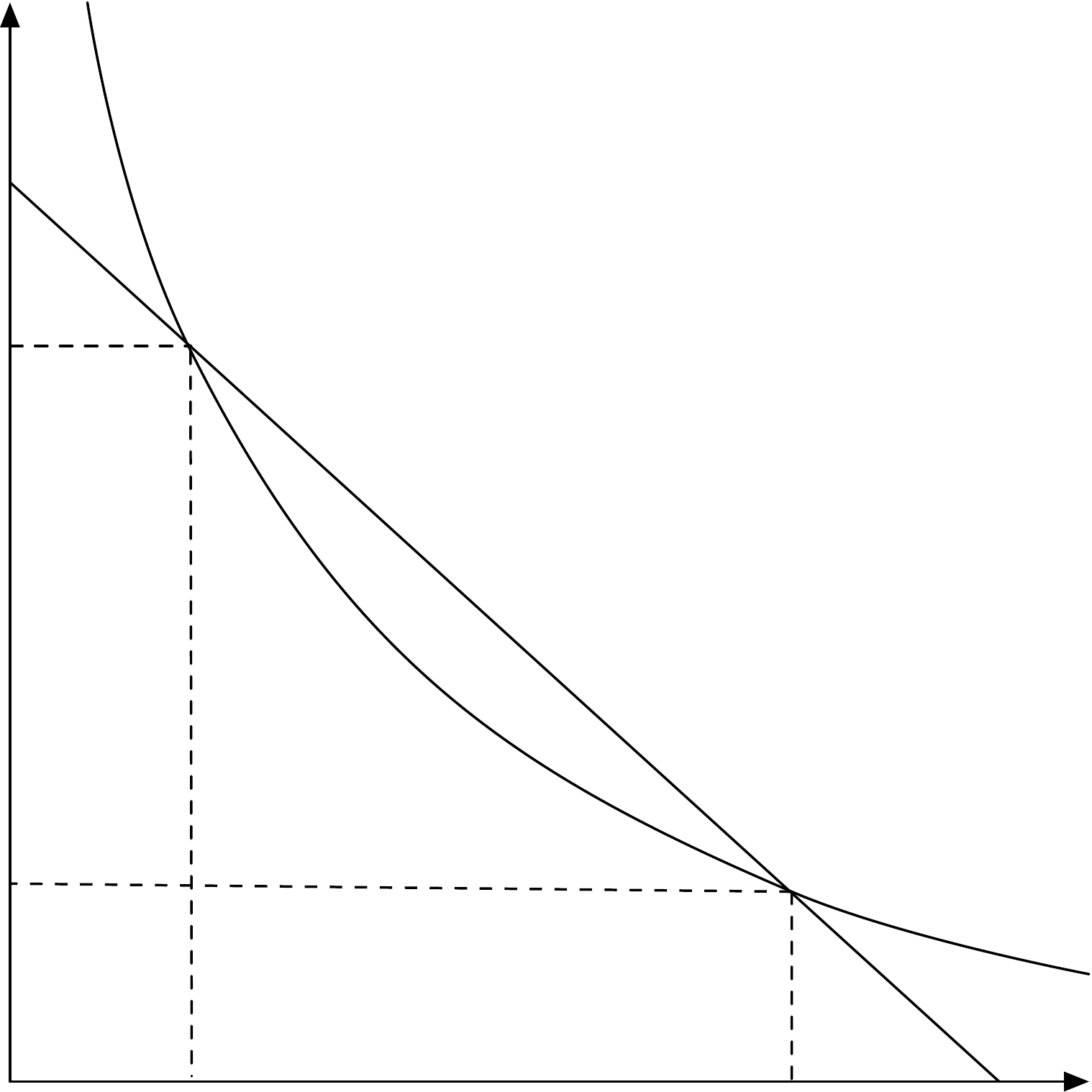} 
 	\put(102,0){\small $\gamma$}
	\put(-5,65){\small $1$}
	\put(17,-5){\small $d$}
	\put(70,-5){\small $1$}
	\put(90,17){\small $d/\gamma$}
	\put(45,45){\small $d+1-\gamma$}
	\put(-5,20){\small $d$}
 \end{overpic}
\caption{Comparison between the estimates $D_2(\gamma)$ (straight line) and
  $D_1(\gamma)$ (convex curve) for the case $d=\dim_{\rm B}(\Xi)<1$.} 
\label{fig.comparison}
\end{figure}

This also points out an error in \cite{Wal:07}, whose main result can be seen to
be false exactly because it asserts that the box dimension equals $D_2(\gamma)$
(and its generalized version corresponding to \eqref{e.pressure_basicpiece}) in
situations where $D_1(\gamma)<D_2(\gamma)$.  Specifically, compare
Section~\ref{sec:ex1}.  More precisely, one of the main problems in
\cite{Wal:07} is that the Intermediate Value Theorem (IVT) is applied to
continuous functions defined only on Cantor sets. Even for situations where the
formula for the box dimension in \cite{Wal:07} is expected to be the correct
one, we do not see how to fix this gap in a direct way.%
\footnote{It is interesting to compare this observation with the discussion and
  open problem in~\cite[Section 4, Remark 6]{PrzUrb:89}.}  On the contrary, this
rather leads to the concept of {\em fibered blenders}, that in certain
situations provides us with an analogue of the IVT.  The blender property is a
way to recover the one-dimensional structure which instead in the strong
unstable direction is now observed in the transverse center unstable (fiber)
direction.

Finally, the example in Section~\ref{sec:ex1} illustrates well that when such a
blender exists, then the strategy to use the upper estimate~\eqref{eq:opt2} is
essential and optimal.  In Section~\ref{ss.fiberedbh} we present a class of
examples of horseshoes which we call \emph{fibered blender-horseshoes} for which
we show that they are fibered blenders with a germ property.  We close this
section with two remarks about our main results.

\begin{remark}[Continuity of box dimension]%\label{rem:continuity}
As long as the skew product structure is maintained, all objects and quantities in the above three theorems depend continuously on the map. Hence, the box dimension of the invariant graph depends continuously on the map (when restricting to skew product maps).
\end{remark}

\begin{remark}[Upper bound for dimension and dimension of slices]
In Section~\ref{newsec:proofss} we study the dimensions of slices of the graph
by stable and unstable manifolds (see
Propositions~\ref{pro:localdimension-stable}
and~\ref{pro:localdimension-unstable}). Note that the dimension value $d$ in any
of the three theorems is in fact an upper bound for the upper box dimension of
unstable slices just assuming the Standing hypotheses (see first claim in
Proposition~\ref{pro:localdimension-unstable}).
\end{remark}

%-----------------------------------------------------------------------------------------------------------------------
\subsection{Main ingredients and organization}
%-----------------------------------------------------------------------------------------------------------------------

Let us briefly sketch the main idea for proving the theorems, at the same time giving an overview of the content of the paper.  In Section~\ref{sec:dimentnew} we state some basic facts about entropy, pressure and box dimension.
In Section~\ref{Examples} we provide some paradigmatic examples.
In Section~\ref{sec:perlim} we provide preliminary results about (partially) hyperbolic systems and the Markov structure of our sets.
The latter provides a natural (semi-)conjugation between the dynamics on the graph and the dynamics on a corresponding shift space
(this is essential since all dynamical quantifiers of $T$ such as Birkhoff averages and pressure have their corresponding quantifier in the symbolic setting).
In Section~\ref{sec:crireg} we discuss the critical regularity of the hyperbolic graphs.
Section~\ref{sec:Blenders} is devoted to the presentation and study of blenders.
%By means of a Markov partition the dynamics on $\Phi$ is represented by a symbolic dynamics on the space of bi-infinite sequences.
In Section~\ref{sec:shiftspace} we explain how to deal with the thermodynamic quantifiers when studying the dynamics on unstable manifolds only which amounts in studying the associate space of onesided infinite sequences. We also give a pedestrian approach to the multifractal analysis which is needed (see Remark~\ref{rem:pedest}). The results of this section then will be applied in Section~\ref{newsec:proofss} to determine the dimension of stable and unstable slices.
Here, we follow a strategy by Bedford~\cite{Bed:89} and perform a multifractal analysis of pairs of Lyapunov exponents of points in $\Phi$ by studying the weak expanding fiber direction (governed by Birkhoff averages of the potential $\varphi^\cu$) and the strong expanding direction tangent to the local strong unstable manifolds (governed by Birkhoff averages of the potential $\varphi^\u$). Then we choose the (uniquely determined) pair $(\alpha_1,\alpha_2$) such that the topological entropy of its level set $\cL(\alpha_1,\alpha_2)$ is maximal. 
%This entropy is the same in any intersection of the level set $B$ with a local unstable manifold  (see Theorem~\ref{the:Manning} below and discussion in Section~\ref{sec:diment}). 
The level set is close to being ``homogeneous'' in the sense that every point in it has the very same pair of exponents and hence one can  cover it by rectangles of approximately equal widths and heights. To argue that in a refining cover of those rectangles by squares indeed all squares are required, we distinguish two cases: 
\begin{itemize}
	\item  either $\tau$ is an Anosov surface (hence mixing) diffeomorphism on $\Xi=M$ or $\Xi$ is a one-dimensional hyperbolic attractor, 
	\item or we invoke the germ property of a fibered blender (see Section~\ref{sec:Blenders}). 
\end{itemize}
This will provide an estimate from below of the box dimension. The upper estimate is based on standard Moran cover arguments.
Finally, in Section~\ref{sec:locpro} we provide more details about stable/unstable holonomies and the dimension of (local) products.
The proofs of Theorems~\ref{t.anosov},~\ref{the:one-dimensional}, and~\ref{the:1} will be concluded at the end of Section~\ref{sec:locpro}.

%----------------------------------------------------------------------------------------------------------------------
\section{Preliminaries on entropy, pressure, and box dimension}\label{sec:dimentnew}
%----------------------------------------------------------------------------------------------------------------------

%----------------------------------------------------------------------------------------------------------------------
\subsection{Entropy and pressure}\label{sec:diment}
%----------------------------------------------------------------------------------------------------------------------

Consider a continuous map $\cT\colon X\to X$ of a compact metric space $(X,d)$.
Given $\varepsilon>0$ and $n\ge1$, a finite set of points $\{x_k\}\subset X$ is \emph{$(n,\varepsilon)$-separated} (with respect to $\cT$) if $\max_{m=0,\ldots,n-1}d(\cT^m(x_i),\cT^m(x_j))>\varepsilon$ for all $x_i,x_j$, $x_i\ne x_j$. 
Given  a continuous function $\psi\colon X\to\bR$ and $n\ge0$, the \emph{$n$th Birkhoff sum} of $\psi$ (with respect to $\cT$) is
\[
	S_n\psi
	\eqdef \psi+\psi\circ \cT+\cdots+\psi\circ \cT^{n-1}.
\]
The \emph{topological pressure} of $\psi$ (with respect to $\cT$) is defined by
\[
	P_{\cT}(\psi)
	\eqdef \lim_{\varepsilon\to0}\limsup_{n\to\infty}
	\frac1n\log\sup\sum_{k}e^{S_n\psi(x_k)},
\]
where the supremum is taken over all sets of points $\{x_k\}\subset X$ which are $(n,\varepsilon)$-separated (see~\cite{Wal:82} for properties of the pressure). Recall that $h(\cT)\eqdef P_\cT(0)$ is the \emph{topological entropy} of $\cT$.

Let $\cM(\cT)$ be the space of $\cT$-invariant probability measures. Given $\nu\in\cM(\cT)$, denote by $h(\nu)$ its \emph{(metric) entropy} (with respect to $\cT$), see for instance \cite{Wal:82}.
Recall that the topological pressure satisfies the following \emph{variational principle} 
\[
	P_{\cT}(\psi)
	=\max_{\nu\in\cM(\cT)}\Big(h(\nu)+\int\psi\,d\nu\Big).
\]

%----------------------------------------------------------------------------------------------------
\subsection{Box dimension. Definition and properties}\label{sec:defdim}
%----------------------------------------------------------------------------------------------------

We recall some standard definitions and facts (see~\cite{Fal:97}). Let $(X,d)$ be a  metric space and $A\subset X$ a totally bounded set.
Given $\delta>0$, denote by $N(\delta)\ge1$ the smallest number of $\delta$-balls which are needed to cover $A$. Define the \emph{lower} and \emph{upper box dimension} of $A$ by
\[
	\underline\dim_{\rm B}(A)\eqdef \liminf_{\delta\to0}\frac{\log N(\delta)}{-\log\delta}
	\quad\text{ and }\quad
	\overline\dim_{\rm B}(A)\eqdef \limsup_{\delta\to0}\frac{\log N(\delta)}{-\log\delta},
\]
respectively. If both limits coincide, then 
the \emph{box dimension} of $A$ is defined by 
$\dim_{\rm B}(A)\eqdef\underline\dim_{\rm B}(A)=\overline\dim_{\rm B}(A)$. 
Later on we will make us of the fact that in our setting we can also count by $N(\delta)$ the least number of squares of size $\delta$ needed to cover $A$, since after taking limits we obtain an equivalent definition of the corresponding dimensions (see~\cite{Fal:97}).

Note that $\dim_{\rm B}$ is \emph{stable} in the sense that for totally bounded sets $A,B\subset X$ satisfying $\underline\dim_{\rm B}(A)=\overline\dim_{\rm B}(A)$ and  $\underline\dim_{\rm B}(B)=\overline\dim_{\rm B}(B)$ the box  dimension of $A\cup B$ is well defined and 
\[
	\dim_{\rm B}(A\cup B)=\max\{\dim_{\rm B}(A),\dim_{\rm B}(B)\}.
\]

If $\pi\colon X\to Y$ with $(X,d_X)$, $(Y,d_Y)$ metric spaces is a H\"older continuous map with H\"older constant $\gamma$ and $A\subset X$ is totally bounded, then $\underline\dim_{\rm B}(\pi(A))\le\gamma^{-1}\underline\dim_{\rm B}(A)$ (the same holds for $\overline\dim_{\rm B}$ and $\dim_{\rm H}$, respectively). Hence,  box dimension (and  also Hausdorff dimension) is invariant under maps which are bi-H\"older continuous with H\"older exponents arbitrarily close to $1$  or which are bi-Lipschitz continuous. 

Finally, recall that if $A,B\subset X$ are two totally bounded sets for which the box dimension is well-defined, then the box dimension of the direct product $A\times B$  is as well  and
\begin{equation}\label{eq:product}
	 \dim_{\rm B}(A\times B)
	= \dim_{\rm B}(A)+\dim_{\rm B}(B).
\end{equation} 

%----------------------------------------------------------------------------------------------------
\section{Examples}\label{Examples}
%----------------------------------------------------------------------------------------------------

%----------------------------------------------------------------------------------------------------
\subsection{Anosov  map in the base}\label{sec:ex2}
%----------------------------------------------------------------------------------------------------

	Kaplan \emph{et al.}~\cite{KapMalYor:84} study the map 
\begin{equation}\label{eq:KapYor}
	\tilde T\colon
        \bT^2\times \bR\to\bT^2\times\bR,
	\quad
	(\xi,x)\mapsto(\tilde\tau(\xi) ,p(\xi)+\lambda^{-1} x),
\end{equation}
where $\bT^2$ is the two-dimensional torus and $\tilde\tau\colon\bT^2\to\bT^2$ is the linear hyperbolic (Anosov) torus automorphism induced by the matrix
\[
	A=\left(\begin{matrix}2&1\\1&1\end{matrix}\right)
\]
with eigenvalues $\kappa=(3+\sqrt5)/2$ and $\mu=\kappa^{-1}$, where $\lambda\in(1,\kappa)$, and $p\colon\bT^2\to\bR$ is a $C^3$ function of period $1$ in each coordinate. They show that for $\tilde\Phi\colon\bT^2\to\bR$ the invariant graph for~\eqref{eq:KapYor} either
\begin{itemize}
\item[(a)]  $\tilde\Phi$ is smooth (and hence $\dim_{\rm B}(\tilde\Phi)=2$), or
\item[(b)]  $\tilde\Phi$ is nowhere differentiable and
\begin{equation}\label{eq:ex2}
	\dim_{\rm B}(\tilde\Phi)
	%=1+1+1-\frac{\log\lambda}{\log\kappa}
	=3-\frac{\log\lambda}{\log((3+\sqrt 5)/2)}.
\end{equation}
\end{itemize}
In particular, the box dimension does not depend on the map $p$.

Theorem~\ref{t.anosov} applies to the inverse of this system $T=\tilde T^{-1}$ and $\tau=\tilde\tau^{-1}$. In this case the potentials defined in~\eqref{eq:defbasicpotentialvarphius} and~\eqref{e.potentials} are constant and given by $\varphi^\cu= -\log \lambda$ and $\varphi^\u=-\log\kappa$. Observe that  the Lebesgue measure $m$ is an SRB measure which is, at the same time, a measure of maximal dimension and maximal entropy. 
Hence,
%The latter fact is equivalent to 
$h(T)=h(m)=\log\kappa=\lvert\log\mu\rvert$.
This implies that for every $d\in \bR$ it holds
\[\begin{split}
	P_\tau(\varphi^\cu+(d-1)\varphi^\u)
	&= \max_{\nu \in \cM (\tau) }\left(h(\nu)+\int_{\mathbb{T}^2} (\varphi^\cu+(d-1)\varphi^\u)\, d\nu \right)\\
	&=\log\kappa -\log\lambda - (d-1)\log\kappa .
\end{split}\]	
Hence, $d$ satisfies  (\ref{e.pressure_equality})
if, and only if,  $2-d=\log\lambda / \log\kappa$. Hence, the box dimension of the invariant graph can be computed explicitly and equals%
% $3-\log\lambda/(\log(3+\sqrt5)-\log2)$ (and of course coincides with
~\eqref{eq:ex2}. 

For the particular case  $\lambda=3/2$ and $p(\xi_1,\xi_2) = \sin(2\pi\xi_1)\sin(2\pi\xi_2) + \cos(4\pi\xi_2)$, the global attractor of \eqref{eq:KapYor} is depicted in Figure~\ref{Fig.CatAttractor}. Two-dimensional slices through the $\xi_1$- and $\xi_2$-axis are shown in Figure~\ref{Fig.CatAttractor_slices}. 
\begin{figure}[h!]
 \begin{overpic}[scale=.25]{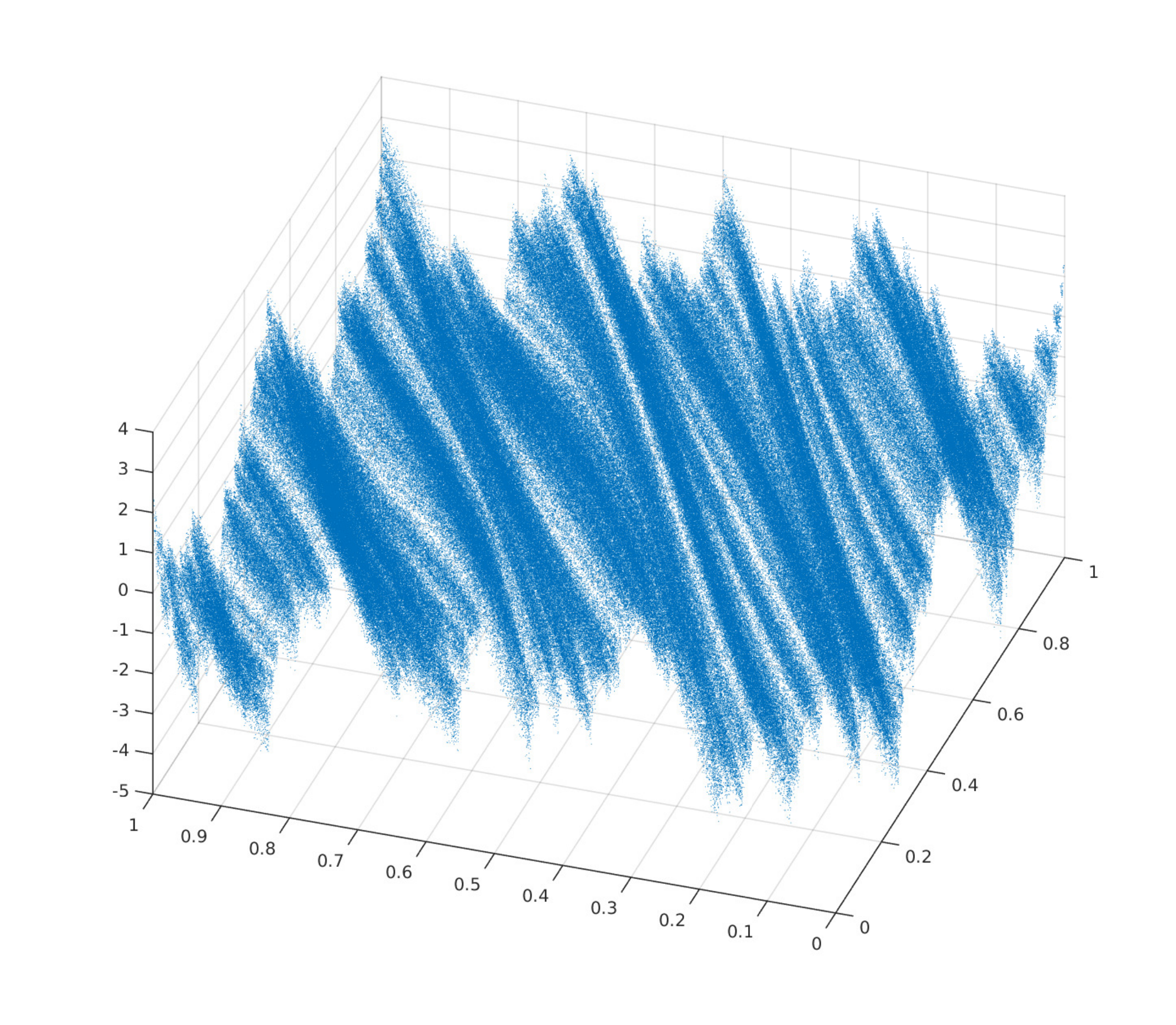}
 	\put(65,2){\small$\xi_1$}
 	\put(95,36){\small$\xi_2$}
 	\put(5,49){\small$x$}
\end{overpic}	 
 \begin{overpic}[scale=.25]{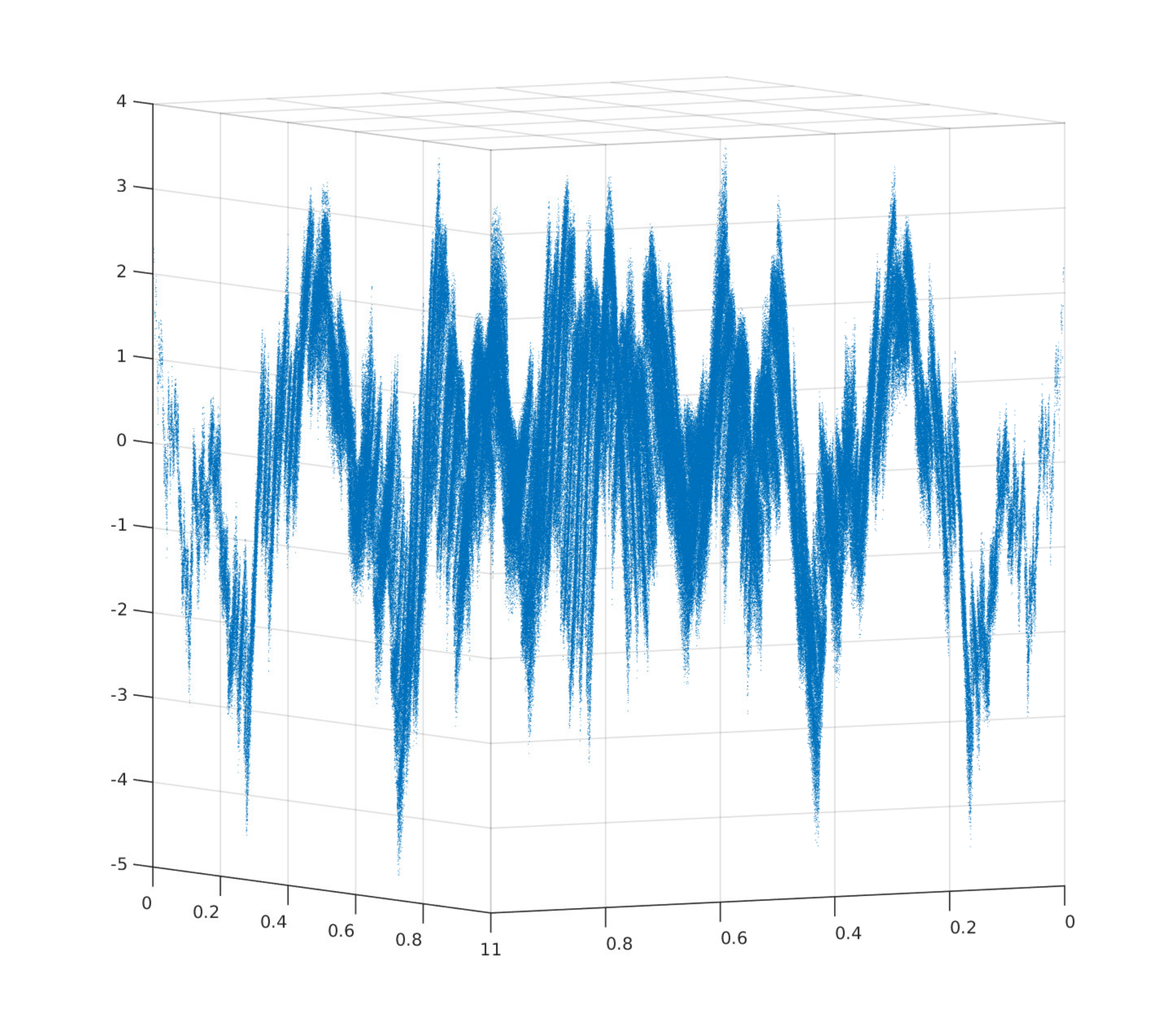}
 	\put(35,1){\small$\xi_1$}
 	\put(89,3){\small$\xi_2$}
 	\put(5,76){\small$x$}
\end{overpic}	
  \caption{The attractor of the skew product system \eqref{eq:KapYor} for the particular case considered, viewed from two different angles.}
  \label{Fig.CatAttractor}
\end{figure}
\begin{figure}[h!]
 \includegraphics[scale=.28]{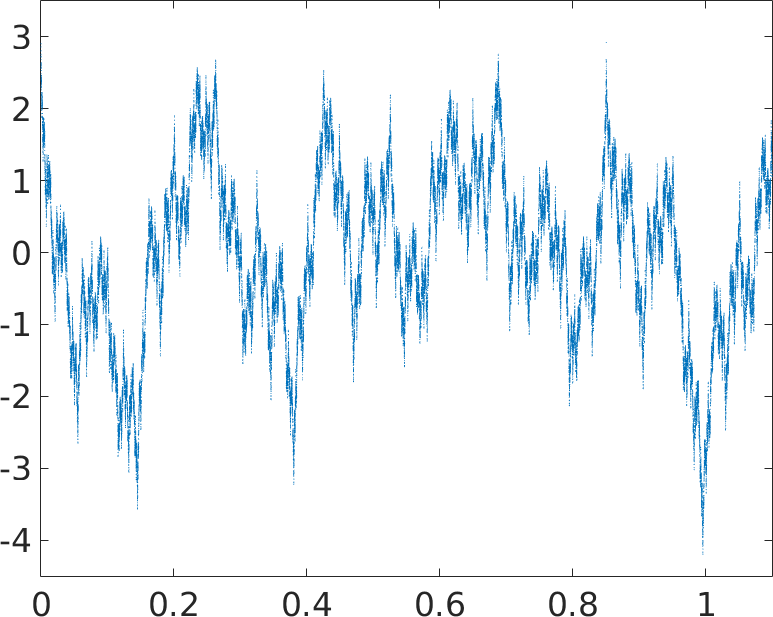}
% 	\put(85,0){\small$\xi_2$}%
% 	\put(5,60){\small$x$}
\quad
  \includegraphics[scale=.28]{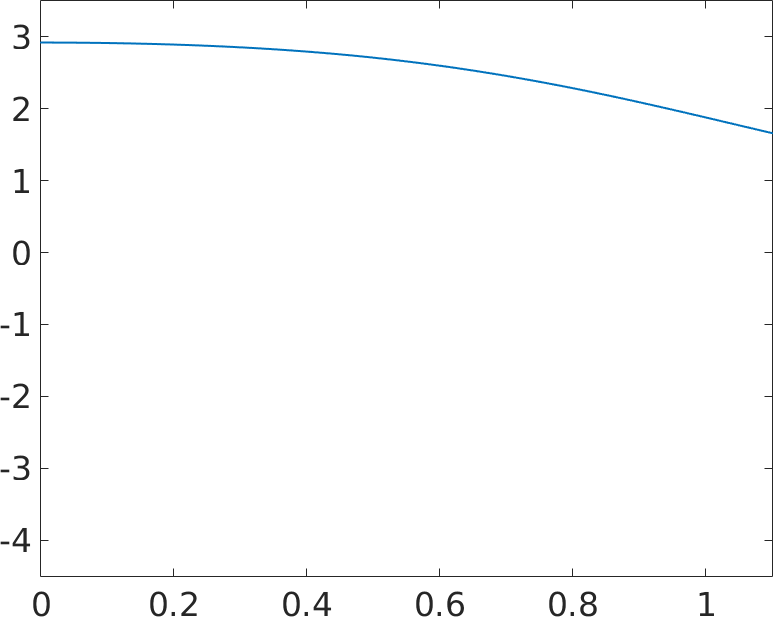}
  \caption{Slices through the attractor depicted in
    Figure~\ref{Fig.CatAttractor} along the unstable manifold (left)  and the stable manifold (right) through the point $\xi=(0,0)$, respectively.}
    \label{Fig.CatAttractor_slices}
\end{figure}
%
%%----------------------------------------------------------------------------------------------------
%\subsection{One-dimensional hyperbolic attractors in the base}\label{ss.DAP}
%%----------------------------------------------------------------------------------------------------
%
%%\begin{example}[One-dimensional hyperbolic attractors in the base] 
%%First, we define a DA attractor. To do so, 
%Let us first recall the construction of the \emph{derived from Anosov} (\emph{DA}) \emph{diffeomorphism} of $\bT^2$ by Smale in \cite{Sma:67}.
%Start with a linear hyperbolic automorphism of $\bT^2$ and consider a local perturbation introducing a repeller in the dynamics in such a way that the resulting diffeomorphisms is axiom A and has 
%two basic sets:  the repelling fixed point and a one-dimensional attractor (recall Section~\ref{ss.onedimatt}).
%A \emph{DA attractor} is any hyperbolic attractor which is conjugate to the attractor of some DA diffeomorphism.
%%
%
%Another example is a diffeomorphism on the two-dimensional disk (which hence can be embedded into any surface) with a \emph{Plykin attractor}, see for instance~\cite{Ply:74}. As in the case of DA attractor, the Plykin attractor is also one-dimensional attractor.
%%\end{example}

%----------------------------------------------------------------------------------------------------
\subsection{Cantor set in the base}\label{sec:ex1}
%----------------------------------------------------------------------------------------------------

In~\cite{BonDiaVia:95}  there are studied the following (very classical) models 
(see also~\cite{PolWei:94}) which are paradigmatic examples of blenders. For that reason we will provide
the detailed construction.
Fix numbers $0<\mu<1/2<1<\lambda<\kappa$, $\kappa>2$. 
Start with a surface diffeomorphism $\tau$ exhibiting a horseshoe $\Xi$. To simplify the exposition, let $\Xi=\bigcap_{n\in\bZ}\tau^n(R)\subset\bR^2$, $R=[-2,2]^2$,
 assume that $\tau^{-1}(R)\cap R$ consists of two connected components $\tilde D_1,\tilde D_2$, and that $\tau$ is affine in each of them and satisfies 
\[
	d\tau|_{\tilde D_1\cup\tilde D_2}
	=\left(\begin{matrix}\mu&0\\0&\kappa\end{matrix}\right).
\]
In particular, $\tilde D_i=[-2,2]\times D_i$ where $D_i$ is some interval in $[-2,2]$, $i=1,2$. Let $\xi=(\xi^\s,\xi^\u)$ be the usual coordinates in $R=[-2,2]^2$. Suppose that the fixed point of $\tau$ in $\tilde D_1$ is located at $(0,0)$ and that  the other fixed point of $\tau$ is located at $(1,1)\in\tilde D_2$. 
The set $\Xi$ is a direct product $\Xi=C^\s\times C^\u\subset[-2,2]^2$ of 
two 
%dynamically defined 
Cantor sets (for each of them Hausdorff dimension and box dimension coincide) which satisfy $\dim( C^\s)=\log2/\lvert\log \mu\rvert$ and $\dim( C^\u)=\log2/\log\kappa$ and hence $\dim(\Xi)= \log2/\lvert\log\mu\rvert+\log2/\log\kappa$ (see~\cite{Fal:97}).

Consider a family $\{T_t\}_{t\in(-\delta,\delta)}$, $\delta$ small, of diffeomorphisms of $[-2,2]^2\times \bR$ 
satisfying
\[
	T_t(\xi^\s,\xi^\u,x)
	=\begin{cases} 
		\big(\tau(\xi^\s,\xi^\u),\lambda x\big)&\text{ if }\xi^u\in D_1,\\
		\big(\tau(\xi^\s,\xi^\u),\lambda x-t\big)&\text{ if }\xi^u\in D_2.
	\end{cases}	
\]
Note that $T_t$ has two hyperbolic fixed points, $P_0^t=(0,0,0)$ and $P_1^t=(1,1,t/(\lambda-1))$.

Consider the basic set $\Phi_0=\Xi\times\{0\}$ (with respect to $T_0$)
which is an invariant graph. With the above we have 
\[
	\dim_{\rm H}(\Phi_0)
	= \dim_{\rm B}(\Phi_0)
	= \frac{\log2}{\lvert\log\mu\rvert}+\frac{\log2}{\log\kappa}.
\]
Denote by $\Phi_t$ the continuation for $T_t$ ($t$ small) of $\Phi_0$ which is 
also an invariant graph. %$\Phi_t\colon\Xi\to\bR$. 
Also note that it is a direct product $\Phi_t=C^\s\times F_t$, where $F_t\subset[-2,2]\times\bR$ is a self-affine limit set of a (contracting) iterated function system of affine maps which map the rectangle $[0,1]\times[0, t/(\lambda-1)]$ to the rectangles
\begin{equation}\label{eq:rectanglesS01}
	S_0^t=[0,\kappa^{-1}]\times\left[0,\frac{t}{\lambda (\lambda-1)} \right],\quad
	S_1^t=[1-\kappa^{-1},1]\times\left[\frac{t}{\lambda},\frac{t}{\lambda-1} \right],
\end{equation}
respectively (compare Figure~\ref{fig.first}, see also~\cite{PolWei:94}). 

Let us now discuss the dimension in the case $t\ne0$.
There is a dichotomy between the cases $\lambda\in(1,2)$ and $\lambda\ge2$ (recall that we always require $\lambda<\kappa$ and $\kappa>2$).

\subsubsection{Fibered blenders: $\lambda\in (1,2)$}\label{sss.fiberedblenders}
Note that for $t\ne 0$ and $\lambda\in 
(1,2)$ the projections of the rectangles $S^t_0$ and $S^t_1$ to the vertical axis (fiber) overlap
in the nontrivial interval $[t/\lambda, t/(\lambda (\lambda-1))]$.
Following {\emph{ipsis litteris}} the construction in \cite{BonDiaVia:95} one can verify that for every $t\ne 0$ the set $\Phi_t$ is a fibered blender with the germ property. 
Note that in this case $\mu_\s=\mu_\w=\mu$,
$\lambda_\s=\lambda_\w=\lambda$, and $\kappa_\s=\kappa_\w=\kappa$. Thus, for appropriate choices of the constants $\mu,\lambda,\kappa$ the Pinching hypothesis holds.
Thus, Theorem~\ref{the:1} can be applied. Let us observe that in the case where the graph is a direct product and the holonomies are trivial and hence bi-Lipschitz continuous, the pinching restriction 
to the parameters in \eqref{eq:constants-holonomy} is in fact not required, recall Remark~\ref{r.pinchingholonomies}. 
The potentials~\eqref{eq:defbasicpotentialvarphius} and~\eqref{e.potentials} are constant and given by $\varphi^\s=\log\mu$,  $\varphi^\cu= -\log \lambda$, and $\varphi^\u=-\log\kappa$. Thus, for every $d\in\bR$
\[
	P_{\tau|_{\Xi}}(\varphi^\cu+(d-1)\varphi^\u)=
	\log2 -\log\lambda -(d-1)\log\kappa.
\]
Hence, $d$ satisfies~\eqref{e.pressure_basicpiece} if, and only if,
\[
	d
	= \frac{\log2}{\log\kappa} + 1-\frac{\log\lambda}{\log\kappa}.
\] 
Moreover, as $C^\s$ is a dynamically defined Cantor set, we have
$
	d^\s
	= \log2/\lvert\log\mu\rvert.
$
Hence, for $t\ne0$ we have
\[
	\dim_{\rm B}(\Phi_t)
	= \frac{\log2}{\lvert\log\mu\rvert} + \frac{\log2}{\log\kappa} + 1-\frac{\log\lambda}{\log\kappa}.
\]
In fact, this formula is a consequence of~\cite[Case 5]{PolWei:94}. 
We conclude the study of this case with some general remarks.

\begin{remark}[Discontinuity of dimension]\label{rem:disCantorbase}
The example above illustrates that in general Hausdorff dimension and box dimension do not continuously depend on the dynamics. Suppose that we have chosen the parameters $\mu,\kappa$ such that $\dim_{\rm H} (\Xi)=\dim_{\rm B}(\Xi)<1$ and hence $\dim_{\rm H}(\Phi_0)=\dim_{\rm B}(\Phi_0)<1$. Note that for every $\xi^\s\in C^\s$ the projection of the set $F_t$ onto the fiber contains the interval $[t/\lambda, t/(\lambda (\lambda-1))]$ implying that $\dim_{\rm H}(\Phi_t)>1$ for every $t\ne0$, $t$ small. This immediately implies discontinuity of the dimensions (in fact, this is the point of~\cite{BonDiaVia:95}).  
\end{remark}

\begin{figure}
 \begin{overpic}[scale=.22]{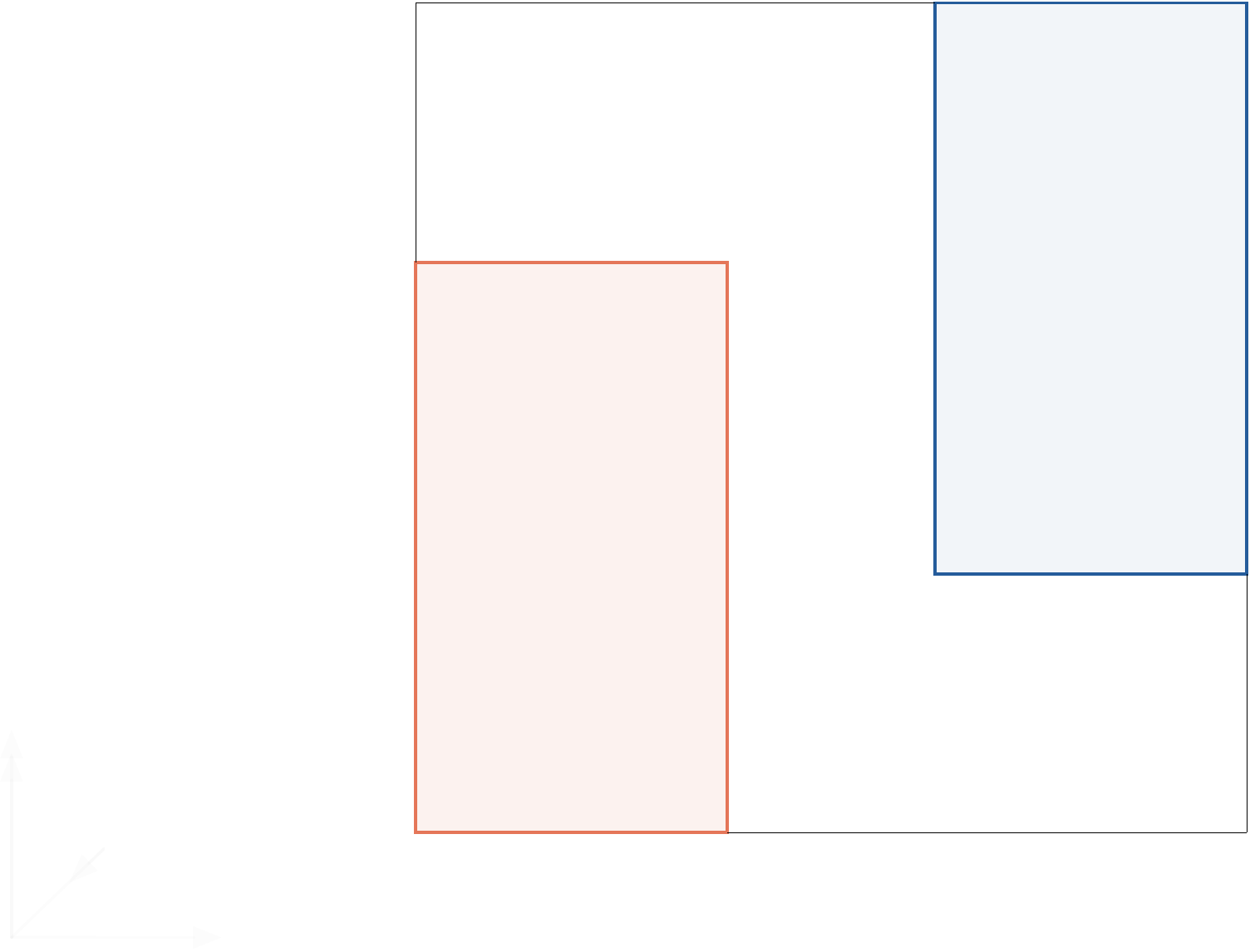} 
 	\put(36,25){\small $S_0^t$}
	\put(84,55){\small $S_1^t$}	
%	\put(10,52){\small $\lambda^{-1}$}
%	\put(52,-2){\small $\kappa^{-1}$}
 \end{overpic}
 \hspace{0.5cm}
\begin{overpic}[scale=.22]{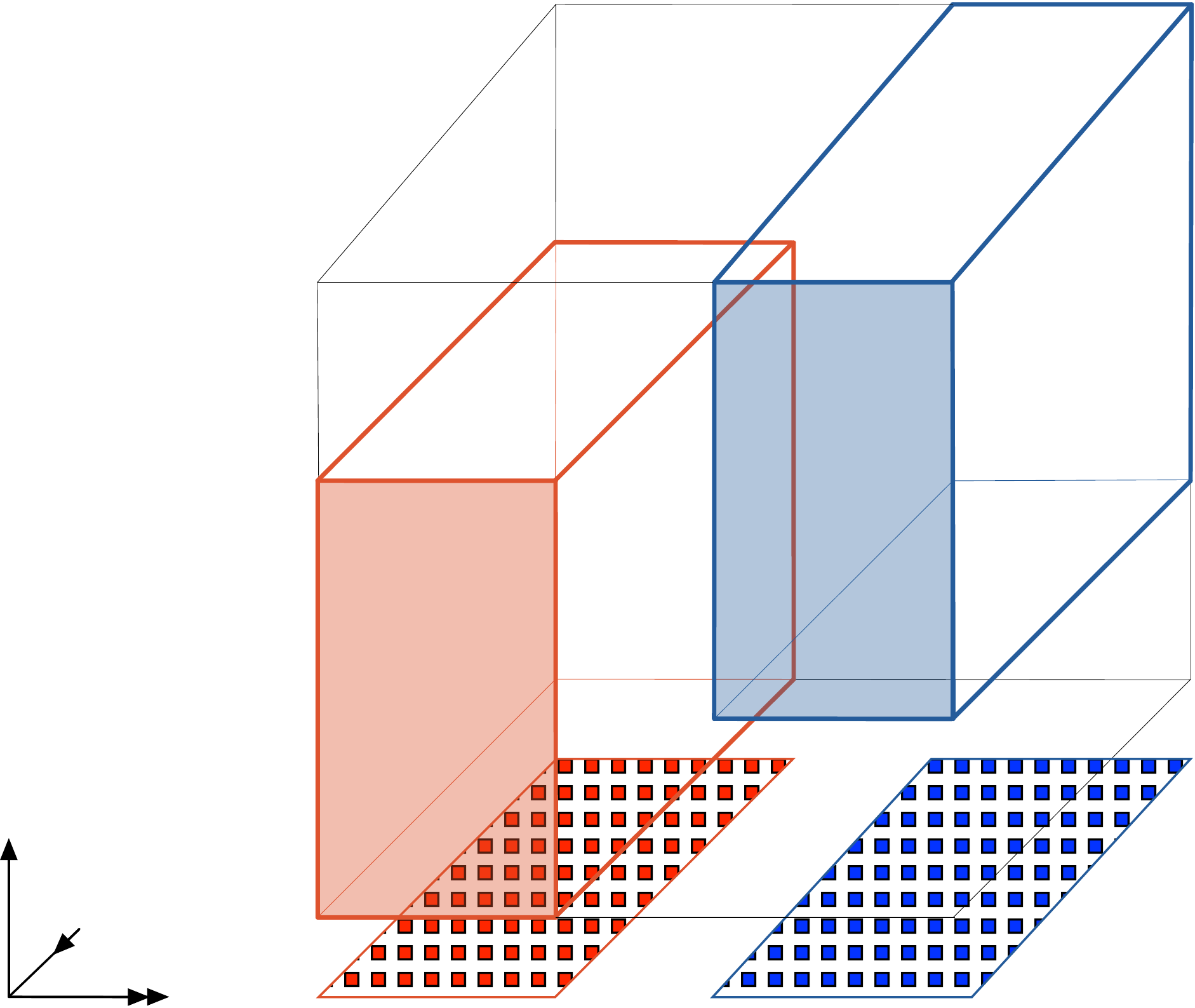}
        \put(17,0){\small $\xi^\u$}
         \put(-2,17){\small $x$}
         \put(8,10){\small $\xi^\s$}
         \put(32,-8){\small $\tilde D_1$}
         \put(65,-8){\small $\tilde D_2$}
%        \put(80,73){\small $\cW^{ss}(X)$}
%        \put(60,68){\small $X$}
%        \put(-6,40){\small $\lvert R^\s_n(X)\rvert_{\rm h}$}
%        \put(50,10){\small $\lvert R^\s_n(X)\rvert_\w $}
 \end{overpic}
 $\overset{T_t}{\longrightarrow}$
 \begin{overpic}[scale=.22]{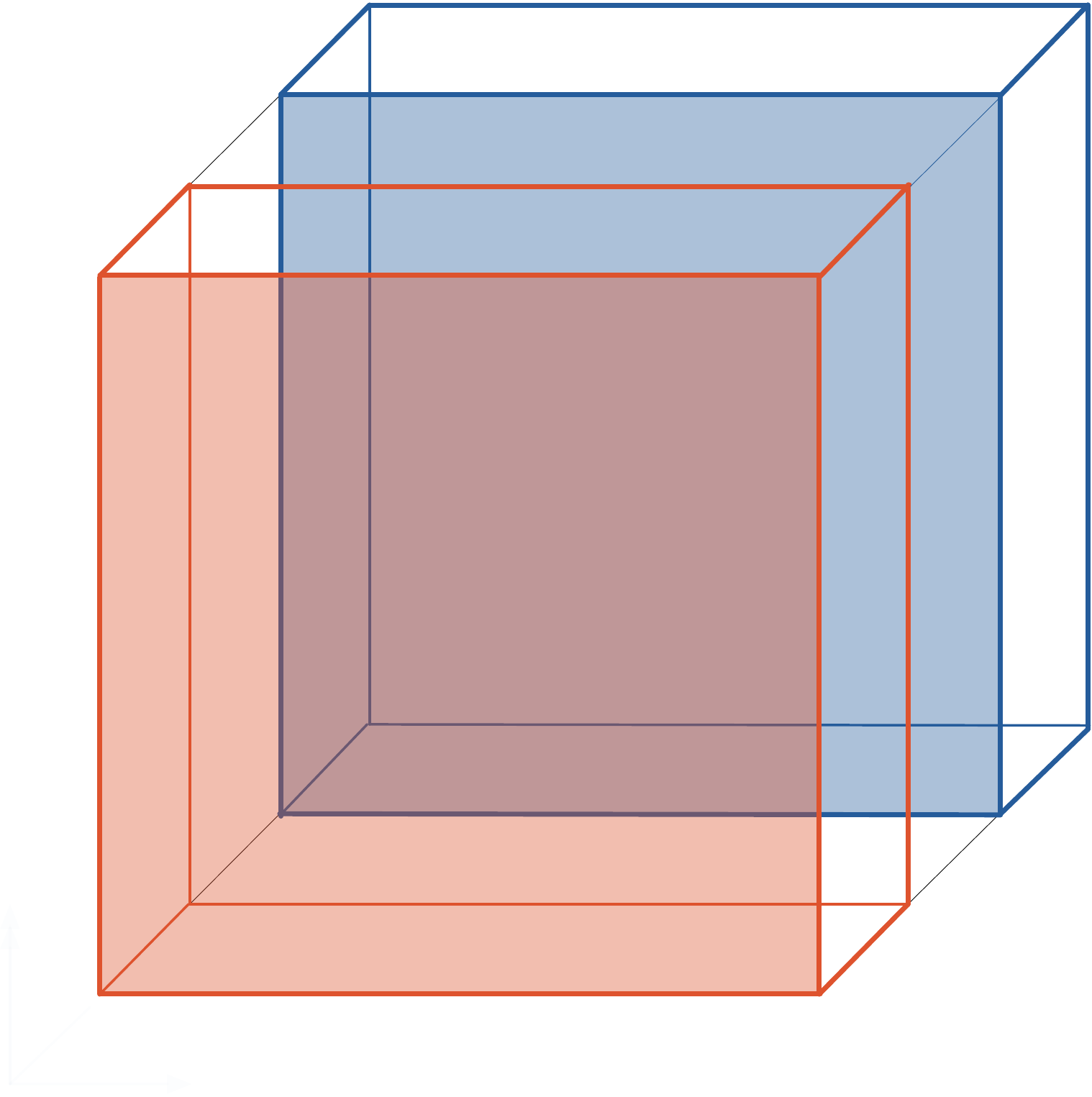} \end{overpic}
\caption{The iterated function system in \eqref{eq:rectanglesS01} (left figure)  and action of  $T_t$ (right figures).}
\label{fig.first}
\end{figure}
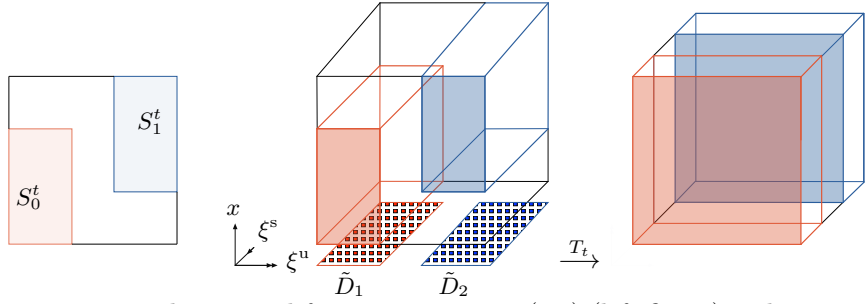
	
\begin{remark}[Lipschitz regularity]
	Since $\Phi_t$ consists of identical copies of $F_t$ over the Cantor set $C^\s$, to see the regularity of the graph  it suffices to study its restriction to any unstable leaf, that is, to study the structure of $F_t$. 
Note that this graph is a Lipschitz graph over the Cantor set $C^\u$ if, and only if, the unstable manifolds of the two fixed points of the iterated function system generating $F_t$ (i.e.~the maps $(\xi^u,x)\mapsto (\kappa\xi^\u,\lambda x)$ and  $(\xi^u,x)\mapsto (\kappa\xi^\u,\lambda x-t)$) coincide, that is, if and only if $t=0$ (compare~\cite[Proposition 2]{Bed:89}). 
\end{remark}

\begin{remark}[Coincidence of Hausdorff and box dimension]\label{rem:coincidenceHausBox}
The issue when Hausdorff dimension and box dimension coincide is in general a difficult task. There exist number-theoretic sufficient conditions on $\lambda$ to verify that both dimensions do not coincide. For example, if $\lambda\in(1,2)$ is the reciprocal of a Pisot-Vijayarghavan number, then the Hausdorff and box  dimension of $\Phi_t$, $t\ne0$ small, do not coincide (see~\cite{PrzUrb:89,PolWei:94}, for instance,  or~\cite{BarRamSim:16} for most recent results and further references).
\end{remark}

\begin{remark}[\emph{A priori} estimates of the dimension]\label{rem:differences}
To further explain the estimates by the numbers $D_1$ and $D_2$ defined in~\eqref{eq:opt2a} and~\eqref{eq:opt2}, note the graph $\Phi$ has a ``critical'' H\"older exponent in the sense of Lemma~\ref{lem:graLip} given by $\gamma=\log\lambda/\log\kappa$. The Cantor set $C^\u$ (and hence any unstable slice $\Xi\cap \cW^\u_{\rm loc}(\cdot,\tau)$) has box dimension $\log 2/\log \kappa$. Note that we have that $D_1(\gamma) = \log2/\log\lambda$ and $D_2(\gamma)=\log2/\log\kappa+1-\log\lambda/\log\kappa$ and hence $D_1(\gamma)=D_2(\gamma)$ if $\lambda=2$. Hence $\lambda\in(1,2)$ and $\lambda>2$ correspond to the cases $D_2(\gamma)<D_1(\gamma)$ and $D_1(\gamma)<D_2(\gamma)$, respectively.
\end{remark}

\subsubsection{The case  $\lambda\ge 2$}
Although this case is not covered by our methods, we remark  that when $t\ne0$ small and $\lambda\ge2$, by~\cite[Case 2]{PolWei:94} we have
\[
	\dim_{\rm H}(F_t)=\dim_{\rm B}(F_t)
	=\frac{\log2}{\log\lambda}
\]
and hence
\[%\label{eq:ex1star}
	\dim_{\rm H}(\Phi_t)=\dim_{\rm B}(\Phi_t)
	= \frac{\log2}{\lvert\log\mu\rvert}+\frac{\log2}{\log\lambda}.
\]

\begin{remark}[Failure of blender property] 
	The projection of $F_t$ to the fiber axis by the canonical projection $(\xi^\s,x)\mapsto x$ is a Cantor set. Hence, it follows that $\Phi_t$ is not a fibered blender with the germ property. Indeed, the germ property (see Definition~\ref{def:3}) is not satisfied.
\end{remark}
%
%whose local stable and local strong unstable manifolds satisfy \margem{used?}
%\[\begin{split}
%	\cW^\uu_{\rm loc}(P_0^t,T_t)
%	&\supset \{0\}\times[-2,2]\times\{0\}
%,\quad	\cW^\s_{\rm loc}(P_0^t,T_t)
%	\supset [-2,2]\times\{0\}\times\{0\},\\
%	\cW^\uu_{\rm loc}(P_1^t,T_t)
%	&\supset \{1%q^\s
%	\}\times[-2,2]\times\left\{\frac{t}{\lambda-1}\right\}
%,\quad	\cW^\s_{\rm loc}(P_1^t,T_t)
%	\supset [-2,2]\times\{1%q^\u
%	\}\times\left\{\frac{t}{\lambda-1}\right\}.
%\end{split}\]

%----------------------------------------------------------------------------------------------------
\subsection{Further examples of fractal graphs}
%----------------------------------------------------------------------------------------------------

Finally, we want to point out that there exist close analogies between the
methods we employ here and those used in recent advances on
Weierstrass graphs, whose Hausdorff dimensions have been determined in
\cite{BarBarRom:14,Kel:14,She:}. The following comments may be helpful  to compare the
different approaches. A Weierstrass function is
given by a converging Fourier series
\[
\varphi(t) = \sum_{n=1}^\infty \lambda^{-n}\cos(2\pi b^n t) ,
\]
where $\lambda>1$, $b\in\bN$ and $\lambda<b$. It is easily checked that $\varphi$
 defines an invariant graph of the skew product system
\begin{equation}\label{eq:basemapwei}
	\tilde T\colon\bT^1\times \bR\to\bT^1\times\bR,\quad
	(t,x) \mapsto (b t\bmod 1,\lambda(x-\cos(2\pi t))) .
\end{equation}
In contrast to our setting, here the base transformation is an expanding map of the circle and, in particular, is not invertible. 
Following \cite{KapMalYor:84,Bed:89}, using the baker's map
\[
	\tau\colon\mathbb{T}^2\to\mathbb{T}^2,\quad
	 \xi=(\xi_1,\xi_2)\mapsto (b\xi_1\bmod 1,(\xi_2+\lfloor b\xi_2\rfloor)/b)
\]
as a canonical extension of the base in~\eqref{eq:basemapwei} leads to the invertible system
\[
	T\colon\mathbb{T}^2\times\bR \to \mathbb{T}^2\times\bR, \quad  
	(\xi,x) \mapsto (\tau(\xi),\lambda(x-\cos(2\pi \xi_1)))  ,
\]
whose unique invariant graph is given by $\Phi(\xi)=\varphi(\xi_1)$. In this situation $\Phi$ is constant along the stable leaves of $\tau$ given by vertical fibers $\{\xi_1\}\times\mathbb{T}^1$. Therefore, it
suffices to determine the dimensions of $\Phi$ restricted to the unstable
leaves $\mathbb{T}^1\times\{\xi_2\}$. For this reason, the discontinuity of the
baker's map along the circles $\{k/b\}\times\mathbb{T}^1$, $k=0,\ldots, b-1$, does not
play any role on the technical level and the setting is
analogous to the one in Theorems~\ref{t.anosov} and \ref{the:one-dimensional} with
diffeomorphisms in the base.

A crucial step in
\cite{BarBarRom:14,Kel:14} is
to show the absolute continuity of the projection of the canonical invariant
measure on $\Phi\cap(\mathbb{T}^1\times\{\xi_2\}\times\bR)$ (which is the projection of
the Lebesgue measure in the base onto $\Phi$) to the section
$\{(0,\xi_2)\}\times\bR$ along the strong unstable manifolds of
$T$. Using either Ledrappier-Young theory
\cite{BarBarRom:14,LedYou:85,Led:92} or more
direct elementary arguments \cite{Kel:14}, this allows
to determine the pointwise dimension of the canonical measure almost surely,
which then entails the result on the Hausdorff dimension.

As this discussion indicates, the strong unstable foliation equally plays a
crucial role in these arguments, similar to the considerations on the box
dimension presented here. We thus hope that adapting and expanding arguments from
\cite{BarBarRom:14,Kel:14}
will eventually  allow to determine the Hausdorff dimension of broader
classes of fractal graphs such as those in this paper. 
%----------------------------------------------------------------------------------------------------
\section{Invariant manifolds and Markov structures}\label{sec:perlim}
%----------------------------------------------------------------------------------------------------

We recall some well-known facts and properties of basic sets (see~\cite{HirPugShu:77,Bow:08,KatHas:95} for details). 

%----------------------------------------------------------------------------------------------------
\subsection{Stable and unstable manifolds}\label{sec:2manif}
%----------------------------------------------------------------------------------------------------

Recall that we assume that $\Xi\subset M$ and $\Phi\subset \Xi\times \bR$ are both basic sets (with respect to $\tau$ and $T$, respectively). Let $d_1$ be the metric in $M$.
The \emph{stable manifold} of a point $\xi\in\Xi$  (with respect to $\tau$) is defined by
\[\begin{split}
	\cW^\s(\xi,\tau)
	&\eqdef \{\eta\in M\colon d_1(\tau^{n}(\eta),\tau^{n}(\xi))\to 0\text{ if }n\to\infty\}
%	&=\bigcup_{n\ge0}\tau^{-n}(\cW^\s_{\rm loc}(\tau^{n}(\xi),\tau)).
\end{split}\]
and is an injectively immersed $C^1$ manifold of dimension $\dim F^\s=1$ tangent to $F^\s$ on $\Xi$. The \emph{local stable manifold} of  $\xi\in\Xi$ (with respect to $\tau$ and a neighborhood $U$ of $\Xi$) is
\begin{equation}\label{def:unsmani}
	\cW^\s_{\rm loc}(\xi,\tau)
	\eqdef \big\{ \eta\in \cW^\s(\xi,\tau) 
		\colon \tau^{k}(\eta)\in U\text{ for every }k\ge 0\big\}.
\end{equation}
Note that there exists $\delta>0$ such that for every $\xi\in\Xi$ the local stable manifold of $\xi$ contains a $C^1$ disk centered at $\xi$ of radius $\delta$. The  \emph{local unstable manifold} at $\xi$, $\cW^\u_{\rm loc}(\xi,\tau)$,  is defined analogously considering $\tau^{-1}$ instead of $\tau$. For every $\xi\in\Xi$ we have
\[%\begin{equation}\label{eq:stabunstabmanifolds}
	\tau(\cW^\s_{\rm loc}(\xi,\tau))\subset \cW^\s_{\rm loc}(\tau(\xi),\tau)	
		\quad \text{ and } \quad
	\tau^{-1}(\cW^\u_{\rm loc}(\xi,\tau))\subset \cW^\u_{\rm loc}(\tau^{-1}(\xi),\tau).
\]%\end{equation}
The sets
\begin{equation}\label{dense}
	\bigcup_{k\ge1}\tau^{-k}\big(\cW^\s_{\rm loc}(\xi,\tau)\big)	
	\quad\text{ and }\quad
	\bigcup_{k\ge1}\tau^k\big(\cW^\u_{\rm loc}(\xi,\tau)\big)
\end{equation}
each are  dense in $\Xi$.

We equip $M\times \bR$ with the metric $d((\xi,x),(\eta,y))=\sup\{d_1(\xi,\eta),\lvert x-y\rvert\}$.

By the skew product structure~\eqref{e.skew-product-structure} and~\eqref{eq:constants}, the invariant bundle $E^\cu$ is tangent to the fiber direction. The \emph{strong unstable} subspace $E^\uu_X$ and the \emph{stable subspace} $E^\s_X$ vary H\"older continuously in $X\in\Phi$, that is, there exist $C>0$ and $\beta>0$ such that $\angle (E^i_X,E^i_Y) \le C d(X,Y)^\beta$ for all $X,Y\in\Phi$, $i=\uu,\s$. Indeed, the H\"older exponent $\beta$ can be controlled through the hyperbolicity estimates (see~\cite{HasPes:06}). 
At every point $X=(\xi,\Phi(\xi))$ 
the subspace $E^\uu_X$ ($E^\s_X$) projects to the \emph{unstable subspace} $F^\u_\xi$ (the \emph{stable subspace} $F^\s_\xi$) (with respect to $\tau$)  in the tangent bundle of the base $M$, which vary H\"older continuously in $\xi\in\Xi$.
Hence  the  functions $\varphi^\u,\varphi^\s\colon\Xi\to\bR$ defined in~\eqref{eq:defbasicpotentialvarphius} are H\"older continuous.

Analogously to the above, we define the \emph{stable manifold} $\cW^\s(X,T)$ and the \emph{unstable manifold} $\cW^\u(X,T)$ of $X\in\Phi$ as well as the \emph{local stable} and \emph{local unstable  manifold} of $X\in\Phi$ (with respect to $T$ and the open set $U\times I$, $I$ an interval) by
\[\begin{split}
	\cW^\s_{\rm loc}(X,T)
	&\eqdef \big\{Y\in \cW^\s(X,T)
		\colon T^{k}(Y)\in U\times I\text{ for every }k\ge 0\big\},\\
	\cW^\u_{\rm loc}(X,T)
	&\eqdef \big\{ Y\in \cW^\u(X,T)
	 	\colon T^{-k}(Y)\in U\times I\text{ for every }k\ge 0\big\},
\end{split}\]
respectively.
Note that $\cW^\s_{\rm loc}(\cdot,T)$ ($\cW^\u_{\rm loc}(\cdot,T)$) is  a lamination through $\Phi$ composed by a union of $C^1$ leaves tangent to $E^\s$ (to $E^\u=E^\cu\oplus E^\uu$). 
Given $X=(\xi,\Phi(\xi))$, we simply will work with
$$
	\cW^{\u}_{\rm loc}(X,T) 
          =  \cW^\u_{\rm loc}(\xi,\tau) \times I.
$$
%(note that this set indeed is contained in a local center unstable manifold of $X$ (with respect to $T$)).
Tangent to $E^\uu$ there is a lamination of \emph{local strong unstable manifolds} $\cW^\uu_{\rm loc}$ through $\Phi$ which subfoliates $\cW^\u_{\rm loc}$. Further, $\cW^\uu_{\rm loc}(X,T)$ is contained in the strong unstable manifold of $X$
\begin{equation}\label{eq:strstable}
	\cW^\uu(X,T)
	\eqdef \big\{Y\in M\times\bR \colon 
	\limsup_{n\to\infty}\frac1n\log d(T^{-n}(Y),T^{-n}(X)) \le -\log\kappa_\w\big\}.
\end{equation}
On the other hand the bundle $E^\cu$ is tangent to the fiber direction and so naturally also integrates to a lamination through $\Phi$, as well as $E^\s\oplus E^\cu$ integrates to a lamination through $\Phi$ which is subfoliated by $\cW^\s$.
%\footnote{This property of $T$ is also called \emph{dynamical coherence}, see~\cite{PugShuWil:97}.}. 

%Finally, there exists $C>1$ such that for every $\xi\in\Xi$ and for every $\eta\in \cW^\u_{\rm loc}(\xi,\tau)$ and $n\ge0$ we have
%\margem{needed???}
%\[	\frac1C\le
%	\frac{d_1(\tau^{-n}(\eta),\tau^{-n}(\xi)) }{\exp(S_n\varphi^\u(\xi))d_1(\eta,\xi)}
%	\le C,
%\] where 
%\begin{equation}\label{eq:defbasicpotentialvarphiu}
%	\varphi^\u(\xi)\eqdef - \log\,\lVert d\tau_{/F_\xi^\u}\rVert,
%\end{equation}
%The analogous relation is true for  local stable manifolds, considering forward iterations and the function 
%\begin{equation}\label{eq:defbasicpotentialvarphis}
%	\varphi^\s(\xi)\eqdef \log\,\lVert d\tau_{/F_\xi^\s}\rVert.
%\end{equation}	 

\begin{remark}%\label{rem:stronunst}
Since $T$ is partially hyperbolic satisfying the Standing hypotheses, at each $X=(\xi,\Phi(\xi))$ each local strong unstable manifold $\cW^\uu_{\rm loc}(X,T)$ is a graph of a function with finite derivative uniformly bounded by some constant independent on $X$.
\end{remark}
%
%\begin{example*}[Example in Section~\ref{sec:ex1} continued]
%The three bundles of the Oseledets splitting are each parallel to one of the axis. 
%	In particular, the center unstable direction (corresponding to the Lyapunov exponent $\log\lambda$) in this case is integrable with integral manifolds being the vertical lines with constant $\xi^\u$ and $\xi^\s$. 
%	The local unstable manifolds are all vertical planes with constant $\xi^\s$. 
%	The local strong unstable manifolds (with Lyapunov exponent $\log\kappa$) are horizontal lines with constant $\xi^\s$ and $x$. 
%	The local stable manifolds (with exponent $\log\mu$) are horizontal lines with constant $\xi^u$ and $x$.
%\end{example*}	

%%%%%%%%%%%%%%%%%%%%%%%%%%%%%%%%%%%%%%%%%
\subsection{Markov rectangles}\label{sec:markov}
%%%%%%%%%%%%%%%%%%%%%%%%%%%%%%%%%%%%%%%%%

Markov rectangles will provide building blocks in our proofs. Let us recall some well-known facts. 
%(see~\cite{Bow:08,KatHas:95} for details).

By hyperbolicity and local maximality of $\Xi$, there exists $\delta>0$ such that for every $\xi,\eta\in\Xi$ with $d_1(\xi,\eta)<\delta$ the intersection 
\begin{equation}\label{eq:locpro}
	[\xi,\eta]\eqdef \cW^\s_{\rm loc}(\xi,\tau)\cap \cW^\u_{\rm loc}(\eta,\tau)\in\Xi
\end{equation}
 contains exactly one point, which is in $\Xi$.

A nonempty closed set $\underline R\subset\Xi$ is called a \emph{rectangle}  if $\diam \underline R<\delta$, $\underline R=\overline{\interior (\underline R)}$ (relative to the induced topology on $\Xi$), and if for every $\xi,\eta\in \underline R$ we have $[\xi,\eta]\in \underline R$.
A finite cover of $\Xi$ by rectangles $\underline R_1,\ldots,\underline R_N$ is a \emph{Markov partition} of $\Xi$ (with respect to $\tau$) if the rectangles have pairwise disjoint interior and if $\xi\in\interior (\underline R_i)\cap \tau^{-1}(\interior (\underline R_j))$ for some $i,j$, then
\[\begin{split}
	\tau\big(\underline R_i\cap\cW^\s_{\rm loc}(\xi,\tau)\big)
	&\subset \underline R_j\cap\cW^\s_{\rm loc}(\tau(\xi),\tau),\\
	\underline R_j\cap\cW^\u_{\rm loc}(\tau(\xi),\tau)
	&\subset \tau\big(\underline R_i\cap\cW^\u_{\rm loc}(\xi,\tau)\big).
\end{split}\]
By~\cite[Chapter C]{Bow:08},  there  exists  a Markov partition with arbitrarily small diameter (where the \emph{diameter of the partition} is the largest diameter of a partition element).

Consider the shift space $\Sigma=\{0,\ldots,N-1\}^\bZ$ and the usual left shift
$\sigma\colon\Sigma\to\Sigma$ defined by $\sigma(\ldots
i_{-1}.i_0i_1\ldots)=(\ldots i_0.i_1i_2\ldots)$. We endow it with the standard metric $\rho(\underline i,\underline i')=2^{-n(\underline i,\underline i')}$,
where $n(\underline i,\underline i')=\sup\{\lvert \ell\rvert\colon
i_k=i'_k\text{ for }k=-\ell,\ldots,\ell\}$. Consider the associated \emph{transition matrix} $A=(a_{jk})_{j,k=1}^{N}$ defined by
\[
	a_{jk}
	=\begin{cases}
		1& \text{ if } \tau\big(\interior(\underline R_j)\big)\cap \interior(\underline R_k)\ne\emptyset,\\
		0&\text{ otherwise}
	\end{cases}
\]
and denote by $\Sigma_A\subset\Sigma$ the subshift of finite type for this transition matrix and consider the standard shift map $\sigma\colon\Sigma_A\to\Sigma_A$. 
For each $\underline i\in\Sigma_A$ the set $\bigcap_{n\in\bZ}\tau^{-n}(\underline R_{i_n})$ consists of a single point, we denoted it by $\chi(\underline i)$. 
The map $\chi\colon\Sigma_A\to\Xi$ is a H\"older continuous surjection, $\chi\circ\sigma=\tau\circ\chi$, and $\chi$ is one-to-one over a residual set.
If $\Xi$ is a Cantor set, then $\chi$ is a homeomorphism (see~\cite[Proposition 18.7.8]{KatHas:95}).

\begin{remark}%\label{rem:fullshift}
Since our dynamics are topologically mixing, without loss of generality we will from now on assume that $\Sigma_A=\Sigma$, that is, $a_{jk}=1$ for every index pair $jk$.
\end{remark}

Given $\xi\in\Xi$, denote by $\underline R(\xi)$ a Markov rectangle which contains $\xi$. Consider the \emph{Markov unstable rectangle}
\[
	\underline R^\u(\xi)\eqdef \underline R(\xi)\cap\cW^\u_{\rm loc}(\xi,\tau).
\]	 
Given $n\ge 1$, consider the \emph{$n$th level Markov unstable rectangle} defined by
\[
	 \underline R^\u_n(\xi)
	\eqdef   \underline R^\u(\xi)
		\cap \tau^{-1}\big(\underline R^\u(\tau(\xi))\big)
		\cap \ldots \cap
		\tau^{-n+1}\big(\underline R^\u(\tau^{n-1}(\xi))\big). 
\]

By the invariance of the continuous graph $\Phi$, given a Markov partition $\{\underline R_1,\ldots, \underline R_N\}$ of $\Xi$ (with respect to $\tau$), by defining
\begin{equation}\label{eq:inducedMarkovbasegraph}
	R_i
	\eqdef \{(\xi,\Phi(\xi))\colon \xi\in \underline R_i\}
\end{equation}
we obtain a Markov partition of $\Phi$ (with respect to $T$) (which is analogously defined, see~\cite{KatHas:95}). This partition shares the analogous properties as above and has the very same transition matrix. In particular, we point out that and if $X\in\interior ( R_i)\cap T^{-1}(\interior ( R_j))$ for some $i,j$, then
\begin{equation}\label{propsMar}\begin{split}
	T\big( R_i\cap\cW^\s_{\rm loc}(X,T)\big)
	&\subset  R_j\cap\cW^\s_{\rm loc}(T(X),T),\\
	 R_j\cap\cW^\u_{\rm loc}(T(X),T)
	&\subset T\big( R_i\cap\cW^\u_{\rm loc}(X,T)\big).
\end{split}\end{equation}
We also adopt the analogous notation of ($n$th level) \emph{Markov rectangles} $R_n$ and \emph{Markov unstable rectangles} $R^\u_n$.

%---------------------------------------------------------------------------------------------------------------------
\section{Critical regularity of the invariant graph}\label{sec:crireg}
%---------------------------------------------------------------------------------------------------------------------

We discuss now the regularity of the invariant graph $\Phi$. 
We start with the following  result that follows directly from~\cite{HirPug:70,HirPugShu:77}, see also~\cite{Sta:99}.  

\begin{lemma}\label{lem:graLip}
	Let $T$ be a $C^{1+\alpha}$ three-dimensional skew product diffeomorphism satisfying the Standing hypotheses.
%	Consider a $C^{1+\alpha}$ skew-product diffeomorphism $T\colon M\times \bR\to M\times \bR$, $T(\xi,x)=(\tau(\xi),T_\xi(x))$, where $\tau\colon M\to M$ is a $C^{1+\alpha}$ surface Axiom A diffeomorphism. Assume that there exists 
	%a basic set $\Lambda\subset M\times I$ (with respect to $T$) which projects to 
%	a basic set $\Xi\subset M$ (with respect to $\tau$) such that $T$ is fiber-wise expanding on $\Xi$. 
	Then the associated invariant graph $\Phi\colon \Xi\to \bR$ is H\"older continuous with H\"older exponent $\gamma$ for every $\gamma\in(0,\alpha]$ satisfying $\gamma<\log\lambda_\w /\log\kappa_{\rm s}$. 
\end{lemma}

In our setting, the graph is always regular in the stable leaves and has the following striking critical regularity in the unstable leaves.%
%To make this more precise, we recall some concepts.
\footnote{Even though this situation is not studied in this paper, we recall that  if  $\kappa_\s<\lambda_\w$ and $\alpha=1$, then $\Phi$ would be Lipschitz on $\Xi$.}
We say that $\Phi\colon\Xi\to \bR$ is  \emph{Lipschitz on the local unstable manifold} of
 $\xi\in\Xi$ if there exists $L(\xi)>0$ such that for every  $\eta\in \cW^\u_{\rm loc}(\xi,\tau)$ we have $\lvert \Phi(\eta)-\Phi(\xi)\rvert\le L(\xi)d_1(\eta,\xi)$.  
We say that $\Phi$ is \emph{Lipschitz on local unstable manifolds} if $\Phi$ is Lipschitz on the local unstable manifold at every $\xi$ with a Lipschitz constant which does not depend on $\xi$. 
Analogously, we define Lipschitz continuity on  local stable manifolds.

\begin{proposition}\label{prop:critical}
	Let $T$ be a $C^{1+\alpha}$ three-dimensional skew product diffeomorphism satisfying the Standing hypotheses.
	The graph of $\Phi$ restricted to $\Xi\cap \cW^\s_{\rm loc}(\xi,\tau)$ is contained in the local strong stable manifold of $X=(\xi,\Phi(\xi))$ and hence $\Phi$ is  Lipschitz on local stable manifolds.
	Moreover, only one of the following two cases occurs:
\begin{itemize}
\item [(a)] $\Phi$ is  Lipschitz on local unstable manifolds.
\item [(b)] $\Phi$ is nowhere contained in local strong unstable manifolds in the sense that for every $\xi\in\Xi$ there is  a sequence $(\eta_k)_k\subset \Xi\cap\cW^\u_{\rm loc}(\xi,\tau)$, $\eta_k\to\xi$, such that $(\eta_k,\Phi(\eta_k))\notin\cW^\uu_{\rm loc}(X,T)$.
\end{itemize}
\end{proposition}

We will conclude the proof of the above proposition towards the end of this section.
In particular, we will derive the following sufficient condition for global Lipschitz continuity.

\begin{corollary}\label{cor:LipPerio}
	If  $\Phi$ is  Lipschitz on the local unstable manifold of some %periodic
point, then $\Phi$ is  Lipschitz on local unstable manifolds.
\end{corollary}

This corollary will be a consequence of the following result  (which can be seen as a local version of~\cite[Proposition 2]{Bed:89}) and Lemma~\ref{lem:intLip} below.

\begin{proposition}\label{procor:period}
For $\xi\in \Xi$ periodic  the following facts are equivalent: 
\begin{itemize}
\item[(1)] $\Phi$ is Lipschitz on the local unstable manifold of $\xi$,
\item[(2)] $\Phi\cap\cW^\u_{\rm loc}(X,T)$ is contained in the local strong unstable manifold of $X=(\xi,\Phi(\xi))$.
\end{itemize}
\end{proposition}

To prove the above proposition, we follow very closely and extend~\cite{HadNicWal:02}, in particular since proofs there are given in the particular case of $\tau$  Anosov.

\begin{remark}
	In the case that $\tau$ is an Anosov diffeomorphism if $\Phi$ is Lipschitz, then, in fact, $E^\s$ and $E^\uu$ are jointly integrable and the tangent bundle $E^\s\oplus E^\uu$ is tangent to $\Phi$. Indeed, as $\Phi$ inherits the regularity of local stable and local strong unstable manifolds, $\Phi$ is uniformly $C^1$ on local stable manifolds and on local strong unstable manifolds and hence  Journ\'e's theorem~\cite{Jou:88} applies.  
\end{remark}

%--------------------------------------------------------------------------------------------------------------------------------------------------------
\subsection{Parametrizing local strong unstable manifolds}
%--------------------------------------------------------------------------------------------------------------------------------------------------------

Below we will use the following notations
\[
	T^n_\xi = T_{\tau^{n-1}(\xi)}\circ\ldots\circ T_\xi,\quad
	T^{-n}_\xi 
	%= (T_\eta)^{-1} \circ\ldots\circ(T_{\tau^{n-1}(\eta)})^{-1}
	= T_{\tau^{-n}(\xi)}^{-1} \circ\ldots\circ T_{\tau^{-1}(\xi)}^{-1}.
\]
Consider  the following family of auxiliary functions. Given $\xi\in\Xi$, for $n\ge1$ define $\gamma^{\u,n}_\xi \colon  \Xi\cap\cW^\u_{\rm loc}(\xi,\tau) \to \bR$ by
\[%\begin{equation}\label{gamma_n_def}
%	\Gamma^{s,n}_\xi \colon  \cW^s_{\rm loc}(\xi)\to X  \colon  
%	\eta \mapsto  T^{-n}_{\tau^n(\eta)}\big(\phi_T(\tau^n(\xi))\big) 
%	=  T^{-n}_{\tau^n(\eta)}\left(T^n_\xi(\phi_T(\xi))\right).
	\gamma_\xi^{\u,n}(\eta)
	\eqdef
	T^{n}_{\tau^{-n}(\eta)}\big(\Phi(\tau^{-n}(\xi))\big) 
	=  T^{n}_{\tau^{-n}(\eta)}\big(T^{-n}_\xi(\Phi(\xi))\big),
\]%\end{equation}
where for the equality we used the invariance relation~\eqref{e.invariant-graph} of the graph.

\begin{lemma}\label{lemmaclaim1}
For every $\xi\in\Xi$  the sequence $(\gamma^{\u,n}_\xi)_n$ converges uniformly to a function $\gamma^\u_\xi\colon \Xi\cap\cW^\u_{\rm loc}(\xi,\tau)\to\bR$ which is Lip\-schitz continuous and has a 
backward invariant  graph, in the sense that 
\begin{equation}\label{stable_manifold_invariance}
%	T(\eta,\gamma^\u_\xi(\eta))=\big(\tau(\eta),\gamma^s_{\tau(\xi)}(\tau(\eta))\big).
	T^{-1}(\eta,\gamma^\u_\xi(\eta))=\big(\tau^{-1}(\eta),\gamma^\u_{\tau^{-1}(\xi)}(\tau^{-1}(\eta))\big),
\end{equation}
%Moreover, the graph of $\gamma_\xi^\u$
which is contained in the strong unstable manifold of $X=(\xi,\Phi(\xi))$ (with respect to $T$).
% that is, $\gamma_\xi^\u$ parametrizes the local strong unstable manifold of $X$.
% and hence, in particular, $\gamma_\xi^\u$ is $C^1$ (in the Whitney sense)%
%\footnote{Recall that if $\Xi\subset \bR^n$ is a closed set then a function $\phi\colon \Xi\to\bR$ then $\phi$ is said to be \emph{$C^1$ in the Whitney sense} if there exists a function $D\colon\Xi\to\cL(\bR^n,\bR)$ and $C\ge0$ such that
%\[
%	\lVert \phi(\xi)-\phi(\eta)-A(x)(\xi-\eta)\rVert\le Cd(\xi,\eta),\quad
%	\lVert A(\xi)-A(\eta)\rVert \le Cd(\xi,\eta)
%\]
%for all $\xi,\eta\in\Xi$. Analogously, in the case that $\Xi$ is a closed subset of a Riemannian manifold.}.
%\margem{add comment on higher regularity? already in text}
Moreover, the family $(\gamma^\u_\xi)_{\xi\in\Xi}$ is equicontinuous.
\end{lemma}

\begin{proof}
Since $T$ is $C^{1+\alpha}$, the maps   $T_\xi$ depend Lipschitz continuously on $\xi$ with some Lipschitz constant $L$. 
Given $\eta\in \Xi\cap\cW^\u_{\rm loc}(\xi,\tau)$, recalling~\eqref{eq:constants}, using the invariance~\eqref{e.invariant-graph}, and the Lipschitz continuity, we have
\[
\begin{split}
\big\lvert\gamma^{\u,n+1}_\xi(\eta)-\gamma^{\u,n}_\xi(\eta)\big\rvert
%&= \left\lvert T^{n+1}_{\tau^{-n-1}(\eta)}\big(\phi_T(\tau^{-n-1}(\xi))\big) 
%		-T^{n}_{\tau^{-n}(\eta)}\big(\phi_T(\tau^{-n}(\xi))\big) \right\rvert\\
& = \big\lvert T^{n}_{\tau^{-n}(\eta)} \big(T_{\tau^{-n-1}(\eta)}(\Phi(\tau^{-n-1}(\xi)))\big)
		-  T^{n}_{\tau^{-n}(\eta)}\big(\Phi(\tau^{-n}(\xi))\big)  \big\rvert \\ 
& \leq \lambda_\s^{n}
	\left\lvert T_{\tau^{-n-1}(\eta)}(\Phi(\tau^{-n-1}(\xi))) 
		- \Phi(\tau^{-n}(\xi))\right\rvert\\ 
&=  \lambda_\s^{n}
	\left\lvert T_{\tau^{-n-1}(\eta)}(\Phi(\tau^{-n-1}(\xi))) 
		- T_{\tau^{-n-1}(\xi)}(\Phi(\tau^{-n-1}(\xi)) %\phi_T(\tau^{-n}(\xi))
		\right\rvert\\ 		
&\le \lambda_\s^{n} L d_1(\tau^{-n-1}(\eta),\tau^{-n-1}(\xi))  \\
&\leq  \lambda^{n}_\s L \kappa^{-n-1}_\w d_1(\eta,\xi) 
=L\kappa_\w^{-1}(\lambda_\s\kappa_\w^{-1})^nd_1(\eta,\xi).
\end{split}
\]
Hence, for every $m\ge n\ge1$
\[%\begin{split}
	\big\lvert\gamma^{\u,m+1}_\xi(\eta)-\gamma^{\u,n}_\xi(\eta)\big\rvert
%	&\le \big\lvert\gamma^{\u,m+1}_\xi(\eta)-\gamma^{\u,m}_\xi(\eta)\big\rvert
%		+\ldots+
%		\big\lvert\gamma^{\u,n+1}_\xi(\eta)-\gamma^{\u,n}_\xi(\eta)\big\rvert\\
%	\le	L\kappa_\w^{-1}d_1(\eta,\xi)(\lambda_\s\kappa_\w^{-1})^n\sum_{k=0}^{m-n}
%			(\lambda_\s\kappa_\w^{-1})^k
	\le	L\kappa_\w^{-1}\frac{1}{1-\lambda_\s\kappa_\w^{-1}}(\lambda_\s\kappa_\w^{-1})^nd_1(\eta,\xi).	
\]%\end{split}\]
Since $\eta$ was arbitrary and since by~\eqref{eq:constants} the last expression  converges to zero exponentially fast as $n\to\infty$,  $(\gamma^{\u,n}_\xi)_n$ is a Cauchy sequence and converges uniformly to a continuous limit $\gamma^\u_\xi$ (note that $d_1(\cdot,\cdot)$ is uniformly bounded due to the compactness of $\Xi$). 
Let us postpone the proof of the Lipschitz continuity of $\gamma_\xi^\u$ for a moment and 
instead first prove that its graph is contained in a strong unstable manifold.

Take a point $Y=(\eta,\gamma_\xi^\u(\eta))$ in the graph of $\gamma_\xi^\u$. Observe that for every $n\ge1$
\[\begin{split}
	d\big( &T^{-n}(\eta,\gamma_\xi^\u(\eta)),T^{-n}(\xi,\Phi(\xi))\big)\\
%	&= d\big( T^{-n} (\eta,\gamma_\xi^\u(\eta)),T^{-n}(\xi,\phi_T(\xi))\big)\\
	&\le d\big( T^{-n}(\eta,\gamma_\xi^\u(\eta)),T^{-n}(\eta,\gamma_\xi^{\u,n}(\eta))\big)
		+ d\big( T^{-n}(\eta,\gamma_\xi^{\u,n}(\eta)),T^{-n}(\xi,\Phi(\xi))\big).
\end{split}\]
For the latter term we have
\[%\begin{split}
	 d\big( T^{-n}(\eta,\gamma_\xi^{\u,n}(\eta)),T^{-n}(\xi,\Phi(\xi))\big)
%	& = d\big(  (\tau^{-n}(\eta),\phi_T(\tau^{-n}(\xi)))
%		 		, (\tau^{-n}(\xi),\phi_T(\tau^{-n}(\xi)))\big)\\
	 \le d_1(\tau^{-n}(\eta),\tau^{-n}(\xi))			
	\le \kappa_\w^{-n}d_1(\eta,\xi).
\]%\end{split}\]
We will now show that also the former term is of the order at most $\kappa_\w^{-n}$ and hence we will conclude that 
\begin{equation}\label{eq:concludeion}
	\limsup_{n\to\infty}
	\frac1n\log d\big( T^{-n}(\eta,\gamma_\xi^\u(\eta)),T^{-n}(\xi,\Phi(\xi))\big)
	\le -\log\kappa_\w.
\end{equation}
Thus, recalling~\eqref{eq:strstable}, we will obtain that $Y=(\eta,\gamma_\xi^\u(\eta))\in \cW^\uu_{\rm loc}(X,T)$. Since $Y$ was an arbitrary point in the graph of $\gamma_\xi^\u$, we will obtain that this graph is contained in the strong unstable manifold of $X$ and thus inherits all its regularity and, in particular, Lipschitz continuity. Moreover, it will also imply equicontinuity of the family $(\gamma^\u_\xi)_{\xi\in\Xi}$.
Indeed to estimate the former term note that
\[\begin{split}
	d\big( T^{-n}(\eta,\gamma_\xi^\u(\eta)),&T^{-n}(\eta,\gamma_\xi^{\u,n}(\eta))\big)
	 = \big\lvert T^{-n}_\eta(\gamma_\xi^\u(\eta)) 
			-	\Phi(\tau^{-n}(\xi)) \big\rvert \\
	& =\lim_{m\to\infty} \big\lvert T^{-n}_\eta(\gamma_\xi^{\u,m}(\eta))
			-	\Phi(\tau^{-n}(\xi)) \big\rvert \\
	& =\lim_{m\to\infty} \big\lvert T^{-n}_\eta \circ T_{\tau^{-m}(\eta)}^m(\Phi(\tau^{-m}(\xi)))
			-	\Phi(\tau^{-n}(\xi)) \big\rvert \\		
%	& =\lim_{m\to\infty} \big\lvert T_{\tau^{-m}(\eta)}^{m-n}(\Phi(\tau^{-m}(\xi))) 
%			-	\phi(\tau^{-n}(\xi)) \big\rvert \\		
	& =\lim_{m\to\infty} \big\lvert T_{\tau^{-m}(\eta)}^{m-n}(\Phi(\tau^{-m}(\xi))) 
			-	T_{\tau^{-m}(\xi)}^{m-n}(\Phi(\tau^{-m}(\xi))) \big\rvert .	
\end{split}\]

\begin{claim}
	For every $\ell\ge1$, $\zeta\in\Xi$, $\zeta'\in\cW^\u_{\rm loc}(\zeta,\tau)$, and $z\in \bR$ we have
\[
	\big\lvert T^\ell_\zeta(z) - T^\ell_{\zeta'}(z)\big\rvert
	\le L\sum_{k=0}^{\ell-1}\lambda_\s^kd_1(\tau^{\ell-k+1}(\zeta),\tau^{\ell-k+1}(\zeta')).
\]		
\end{claim}

\begin{proof}
Note that
\[\begin{split}
	\big\lvert T^\ell_\zeta(z) - T^\ell_{\zeta'}(z)\big\rvert
	&= \big\lvert   T_{\tau^{\ell-1}(\zeta)}\circ T^{\ell-1}_\zeta(z) 
		-  T_{\tau^{\ell-1}(\zeta')}\circ T^{\ell-1}_{\zeta'}(z)\big\rvert\\
	&\le \big\lvert  T_{\tau^{\ell-1}(\zeta)}\circ T^{\ell-1}_\zeta(z) 
		-  T_{\tau^{\ell-1}(\zeta')}\circ T^{\ell-1}_{\zeta}(z)\big\rvert \\
	&\phantom{\le}+ \big\lvert T_{\tau^{\ell-1}(\zeta')}\circ T^{\ell-1}_{\zeta} (z) 
		-  T_{\tau^{\ell-1}(\zeta')}\circ T^{\ell-1}_{\zeta'}(z)\big\rvert\\
	&\le Ld_1(\tau^{\ell-1}(\zeta),\tau^{\ell-1}(\zeta'))
	+ \lambda_\s\big\lvert T^{\ell-1}_{\zeta}(z) - T^{\ell-1}_{\zeta'}(z)\big\rvert,
\end{split}\]	
where we used Lipschitz dependence of the fiber maps and uniform expansion by the \emph{common} fiber map $T_{\tau^{\ell-1}(\zeta')}$  by at most the factor $\lambda_\s$. 
Applying the same argument $\ell$ times implies the claim.
%, we can conclude
%\[
%	\big\lvert T^\ell_\zeta(z) - T^\ell_{\zeta'}(z)\big\rvert
%	\le L\sum_{k=0}^{\ell-1} \lambda_\s^kd_1(\tau^{\ell-k+1}(\zeta),\tau^{\ell-k+1}(\zeta')),
%\]
%proving the claim.
\end{proof}

Continuing with the above calculations,
with this claim we obtain
\[\begin{split}
	\big\lvert  T_{\tau^{-m}(\eta)}^{m-n}(\Phi(\tau^{-m}(\xi))) 
			&-	T_{\tau^{-m}(\xi)}^{m-n}(\Phi(\tau^{-m}(\xi))) \big\rvert \\
	&\le L\sum_{k=0}^{m-n-1}\lambda_\s^kd_1\big(\tau^{m-n-k+1}(\tau^{-m}(\eta)),\tau^{m-n-k+1}(\tau^{-m}(\xi))\big)	\\		
	&= L\sum_{k=0}^{m-n-1}\lambda_\s^kd_1\big(\tau^{-n-k+1}(\eta),\tau^{-n-k+1}(\xi)\big)	\\		
	&\le L\sum_{k=0}^{m-n-1}\lambda_\s^k\kappa_\w^{-(n+k-1)}d_1(\eta,\xi)
	\le L\kappa_\w^{-n}\kappa_\w\sum_{k=0}^{\infty}(\lambda_\s\kappa_\w^{-1})^kd_1(\eta,\xi)\\
	&=\kappa_\w^{-n}\frac{L \kappa_\w}{1-\lambda_\s\kappa_\w^{-1}}d_1(\eta,\xi).
\end{split}\]
%Given $m>n\ge1$ denote $z=\Phi(\tau^{-m}(\xi))$. For any $k=0,\ldots,m-n-1$ we have
%\[
%	  T_{\tau^{-m}(\eta)}^{m-n}(z) 
%			-	T_{\tau^{-m}(\xi)}^{m-n}(z)
%	=	  	T_{\tau^{-n-k}(\eta)}^k\circ T_{\tau^{-m}(\eta)}^{m-n-k}(z) 
%			-	T_{\tau^{-n-k}(\xi)}^k\circ T_{\tau^{-m}(\xi)}^{m-n-k}(z)
%\]
%Hence we can estimate (the left hand side is simply $k=0$ in the sum)
%\begin{multline*}
%	 \big\lvert T_{\tau^{-m}(\eta)}^{m-n}(z) 
%			-	T_{\tau^{-m}(\xi)}^{m-n}(z) \big\rvert \\
%	\le \sum_{k=0}^{m-n-1}
%	 \big\lvert  	T_{\tau^{-n-k}(\eta)}^k\circ T_{\tau^{-m}(\eta)}^{m-n-k}(z) 
%		-	T_{\tau^{-n-k}(\eta)}^k\circ T_{\tau^{-m}(\xi)}^{m-n-k}(z) \big\rvert.
%\end{multline*}
%Using the fact that $T$ expands in each (common) fiber $T^k_{\tau^{-n-k}(\eta)}$, we have 
%\[\begin{split}
%	 \big\lvert T_{\tau^{-m}(\eta)}^{m-n}(z) 
%			-	T_{\tau^{-m}(\xi)}^{m-n}(z) \big\rvert 
%				&\le \sum_{k=0}^{m-n-1}
%	\lambda_\s^k\,
%	 \big\lvert  	 T_{\tau^{-m}(\eta)}^{m-n-k}(z) 
%		-	 T_{\tau^{-m}(\xi)}^{m-n-k}(z) \big\rvert \\	
%	&\le 	 \sum_{k=0}^{m-n-1}
%	\lambda_\s^k\,
%	d_1\big(	\tau^{-n-k}(\eta), \tau^{-n-k}(\xi)\big)\\
%	&\le 	 \sum_{k=0}^{m-n-1}
%	\lambda_\s^k\,
%	\kappa_\w^{-n-k}\,d_1(\eta,\xi)\\
%%	&\le\kappa_\w^{-n}\,\sum_{k=0}^{\infty}(\lambda_\s\kappa_\w^{-1})^k\,d_1(\eta,\xi)\\
%	&\le\kappa_\w^{-n}\,\big(1-\lambda_\s\kappa_\w^{-1})^{-1}\,d_1(\eta,\xi)
%\end{split}\]
%Note that as the latter does not depend on $m$, taking the limit $m\to\infty$, we conclude
%\[
%	d\big( T^{-n}(\eta,\gamma_\xi^\u(\eta)),T^{-n}(\eta,\gamma_\xi^{\u,n}(\eta))\big)
%	\le \kappa_\w^{-n}\,\big(1-\lambda_\s\kappa_\w^{-1})^{-1}\,d_1(\eta,\xi).
%\]
Thus,  we obtain~\eqref{eq:concludeion}.

Invariance~\eqref{stable_manifold_invariance} is easily verified. Thus the lemma is proved.
\end{proof}

\begin{remark}%\label{rem:commHoeld}
	By compactness of $\Phi$, regularity, and hyperbolicity of $T$, all functions $\gamma^\u_\xi$, $\xi\in\Xi$, have a common Lipschitz constant. 
Following the steps in the proof of Lemma~\ref{lemmaclaim1}, one can actually determine this constant; however we refrain from doing so.
\end{remark}

%--------------------------------------------------------------------------------------------------------------------------------------------------------
\subsection{Lipschitz regularity -- sufficient conditions}
%--------------------------------------------------------------------------------------------------------------------------------------------------------

\begin{lemma} \label{lem:genloc}
	Let $\gamma>\log\lambda_\s/\log\kappa_\w$. 
	Let $\xi\in\Xi$ and $\eta\in \Xi\cap\cW^\u_{\rm loc}(\xi,\tau)$. 
	If there exist  $C'>0$ and a sequence $n_k\to\infty$ so that for every $k\ge1$ we have
\begin{equation}\label{eq:condition}
	\big\lvert\Phi(\tau^{-n_k}(\eta))-\Phi(\tau^{-n_k}(\xi))\big\rvert
	\le C'd_1(\tau^{-n_k}(\eta),\tau^{-n_k}(\xi))^\gamma,
\end{equation}
then $\Phi(\eta)=\gamma^\u_\xi(\eta)$.
\end{lemma}

\begin{proof}
Since  $\Phi$ is invariant and $T$ is  fiberwise expanding~\eqref{eq:basehyp}, we have
\[\begin{split}
	\big\lvert T^{n_k}_{\tau^{-n_k}(\eta)}(\Phi(\tau^{-n_k}(\xi)))-\Phi(\eta)\big\rvert
	=&
	\big\lvert T^{n_k}_{\tau^{-n_k}(\eta)}(\Phi(\tau^{-n_k}(\xi)))-
		T^{n_k}_{\tau^{-n_k}(\eta)}(\Phi(\tau^{-n_k}(\eta)))\big\rvert\\
	\le& \lambda_\s^{n_k}\big\lvert \Phi(\tau^{-n_k}(\xi)) -\Phi(\tau^{-n_k}(\eta))\big\rvert.
\end{split}\]
By our hypothesis on the $\gamma$-H\"older regularity of $\Phi$ at $\tau^{-n_k}(\xi)$ we have
\[
	\lvert\Phi(\tau^{-n_k}(\eta))-\Phi(\tau^{-n_k}(\xi))\rvert
	\le C'd_1\big(\tau^{-n_k}(\eta),\tau^{-n_k}(\xi)\big)^\gamma
\]
and from the fact that  $\eta$ is in the local unstable manifold of $\xi$ we obtain
\[
	d_1\big(\tau^{-n_k}(\eta),\tau^{-n_k}(\xi)\big)
	\le \kappa_\w^{-n_k} d_1(\eta,\xi).
\]
So, these three facts and the definition of $\gamma_\xi^{\u,n_k}$ together imply
\[
	\big\lvert \gamma_\xi^{\u,n_k}(\eta)-\Phi(\eta)\big\rvert
	\le C'\big(\lambda_\s\kappa_\w^{-\gamma}\big)^{n_k} d_1(\eta,\xi)^\gamma.
\]
Hence, if $\lambda_\s\kappa_\w^{-\gamma}<1$ and $n_k\to\infty$, then together with Lemma~\ref{lemmaclaim1} we can conclude that $\gamma^\u_\xi(\eta)=\lim_{k\to\infty}\gamma_\xi^{\u,n_k}(\eta)=\Phi(\eta)$. 
\end{proof}

Recall that for $\delta>0$ sufficiently small every local unstable manifold contains a disk of radius $\delta$. Denote
\[
	\cW^\u_\delta(\xi,\tau)
	\eqdef \cW^\u_{\rm loc}(\xi,\tau)\cap B_\delta(\xi).
\]
The following now is an immediate consequence of Lemma~\ref{lem:genloc}.

\begin{corollary} \label{cor:genloc}
	Let $\gamma>\log\lambda_\s/\log\kappa_\w$. 
	Let $\xi\in\Xi$. 
	If there exist $C'>0$ and $\delta>0$ such that  for every $\eta\in \Xi\cap\cW^\u_\delta(\xi,\tau)$ there is a sequence $n_k\to\infty$ such that for every $k\ge1$ we have~\eqref{eq:condition}, then $\Phi(\eta)=\gamma^\u_\xi(\eta)$ for every $\eta\in  \Xi\cap\cW^\u_\delta(\xi,\tau)$. Hence, in particular, the graph of $\Phi$ restricted to $\Xi\cap\cW^\u_\delta(\xi,\tau)$ is contained in the local strong unstable manifold of $X=(\xi,\Phi(\xi))$.
%	In particular, this is true if $\xi$ is periodic and if $\phi_T|_{\cW^\u_{\rm loc}(\xi,\tau)\cap\Xi}$ is Lipschitz at $\xi$.
\end{corollary}

\begin{proof}[Proof of Proposition~\ref{procor:period}]
Given $\xi$ periodic with period $n$, we apply Corollary~\ref{cor:genloc} to $\xi$ taking $n_k=kn$.
\end{proof}

It is convenient to define the following function (see also~\cite{HadNicWal:02}) which measures in a way the ``obstructions'' to the regularity of the invariant graph $\Phi$ on local unstable manifolds. Given $\xi\in\Xi$ let
\begin{equation}\label{eq:defDeltau}
	\Delta^\u_\delta(\xi)
	\eqdef\sup_{\eta\in \Xi\cap\cW^\u_\delta(\xi,\tau)}\lvert\Phi(\eta)-\gamma^\u_\xi(\eta)\rvert.
\end{equation}

\begin{lemma}
	$\Delta^\u_\delta\colon\Xi\to\bR$ is continuous.
\end{lemma}

\begin{proof}
This follows from uniform convergence of the  sequence $(\gamma^{\u,n}_\xi)_n$ in Lemma~\ref{lemmaclaim1}. Indeed, the distance between $\gamma^{\u,n}_\xi$ and $\gamma^\u_\xi$ varies equicontinuously in $n$ and $\xi$. Now, observe that $\gamma^{\u,n}_\xi$ varies continuously in  $\xi$ and recall continuity of  the unstable manifolds $\cW^\u_{\rm loc}(\cdot,\tau)$ and continuity of the graph $\Phi$. 
\end{proof}

\begin{lemma}\label{lem:intLip}
	Assume that  $\Delta^\u_\delta(\xi)=0$ for some $\xi\in\Xi$.
	Then $\Delta^\u_\delta=0$ and hence $\Phi$ is  Lipschitz on local unstable manifolds.
\end{lemma}

\begin{proof}
By hypothesis, $\Phi(\eta)=\gamma_\xi^\u(\eta)$ for every $\eta\in \Xi\cap\cW^\u_{\delta/2}(\xi,\tau)$ and, in particular, $\Phi(\eta)$ is contained in the strong unstable manifold of $X=(\xi,\Phi(\xi))$.

Clearly, $\Delta^\u_\delta(\xi)=0=\Delta^\u_{\delta/2}(\eta)$ for every $\eta\in \Xi\cap\cW^\u_{\delta/2}(\xi,\tau)$.
 If $\delta$ was small enough, then for every $n\ge1$ we have $\tau^n(\cW^\u_{\delta/2}(\eta,\tau))\supset \cW^\u_{\delta/2}(\tau^n(\eta),\tau)$ and from invariance of the graph $\Phi$ we conclude $\Delta^\u_{\delta/2}(\tau^n(\eta))=0$.
 
 By hyperbolicity (recall~\eqref{dense}), the union of all images of $\Xi\cap\cW^\u_{\delta/2}(\xi,\tau)$ is dense in $\Xi$. Hence we obtain $\Delta^\u_{\delta/2}=0$ densely, and continuity implies $\Delta^\u_{\delta/2}=0$. 
 This proves the lemma.
\end{proof}

\begin{proof}[Proof of Corollary~\ref{cor:LipPerio}]
Is a consequence of Proposition~\ref{procor:period} and Lemma~\ref{lem:intLip}.
\end{proof}

\begin{proof}[Proof of Proposition~\ref{prop:critical}]
The proof of the first claim is as in~\cite{HadNicWal:02}. The second claim is a consequence of Lemma~\ref{lem:intLip}.
\end{proof}

Analogously to~\eqref{eq:defDeltau}, we can define a function $\Delta^\s_\delta\colon\Xi\to\bR$ considering local stable manifolds instead of local unstable manifolds. This function is also continuous and we have $\Delta^\s_\delta=0$.

\begin{corollary}
Assume that  $\Delta^\u(\xi)=0$ for some $\xi\in\Xi$.
	Then $\Phi\colon\Xi\to\bR$ is Lipschitz. 
\end{corollary}

\begin{proof}
	By Lemma~\ref{lem:intLip} and the above we have $\Delta^\u=\Delta^\s\equiv 0$ everywhere on $ \Xi $, that is, the graph is Lipschitz along unstable manifolds and along stable manifolds. Note that the local product structure of unstable and stable local manifolds  $[\xi,\eta]=\cW^\s_{\rm loc}(\xi,\tau)\cap \cW^\u_{\rm loc}(\eta,\tau)$ for $\eta$ sufficiently close to $\xi$ (see Section~\ref{sec:markov}) has the property that $\eta\mapsto d_1(\xi,[\xi,\eta])$ is Lipschitz. Thus the graph is Lipschitz on the whole $\Xi$. 
\end{proof}

For further reference in Section~\ref{sec:sizeMarkov} we  formulate the following immediate consequence of  Corollary~\ref{cor:genloc}   (recalling that assumption~\eqref{eq:constants} gives $1>\log\lambda_\s/\log\kappa_\w$, we put $\gamma=1$).

\begin{corollary}\label{cor:erika}
	Assume that $\Phi$ is not 
	 Lipschitz continuous on local unstable manifolds.  	
	Then for every $\delta>0$ there exists $C=C(\delta)>0$ such that
$
	\Delta^\u_\delta\ge C.
$	
\end{corollary}

%%%%%%%%%%%%%%%%%%%%%%%%%%%%%%%%%%%%%%%%%
\subsection{Size of  Markov unstable rectangles}\label{sec:sizeMarkov}
%%%%%%%%%%%%%%%%%%%%%%%%%%%%%%%%%%%%%%%%%

Assume that $\underline R_1,\ldots,\underline R_N$ is a Markov partition of $\Xi$ (with respect to $\tau$) and that $R_1,\ldots,R_N$ is a corresponding Markov partition of $\Phi$ (with respect to $T$) as in Section~\ref{sec:markov}, see~\eqref{eq:inducedMarkovbasegraph}.
For every $\eta\in\Xi$ consider the fiber 
\[%\begin{equation}\label{eq:fibernot}
	I_\eta\eqdef \{\eta\}\times \bR.
\]%\end{equation}	 
Given $X=(\xi,\Phi(\xi))$ and $n\ge1$, to define the ``size'' of an unstable rectangle $R^\u_n(X)$ (note that  its projection to the base is either a Cantor set
or a smooth curve where the latter case occurs when $\tau$ is an Anosov map  or 
$\Xi$ is a one-dimensional attractor), let $R^\u(\xi,n)$ be the minimal curve contained in $\cW^\u_{\rm loc}(\xi,\tau)$ containing $\underline R^\u_n(\xi)$. Let 
\[
	\widehat R^\u_n(X)
	\eqdef \big\{\cW^{\uu}_{\rm loc}(Y,T)\colon Y\in R^\u_n(X)\big\}
		\cap \big\{I_\zeta\colon \zeta\in R^\u(\xi,n)\big\}	
\]	
(compare Figure~\ref{fig.1}), which is the smallest set containing the Markov unstable rectangle (with respect to $T$) of level $n$ containing $X$ which is ``foliated'' by local strong unstable manifolds of points in this rectangle and which is bounded by fibers which project to points in the base bounding the Markov unstable rectangle (with respect to $\tau$).
Let
\[
	\widehat I_n(\xi)
	\eqdef I_\xi\cap \widehat R^\u_n(X).
\]
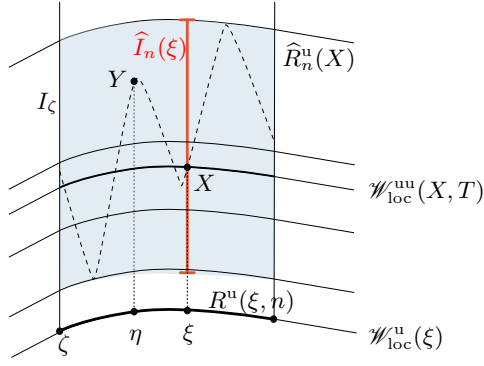
\begin{figure}
\begin{overpic}[scale=.35]{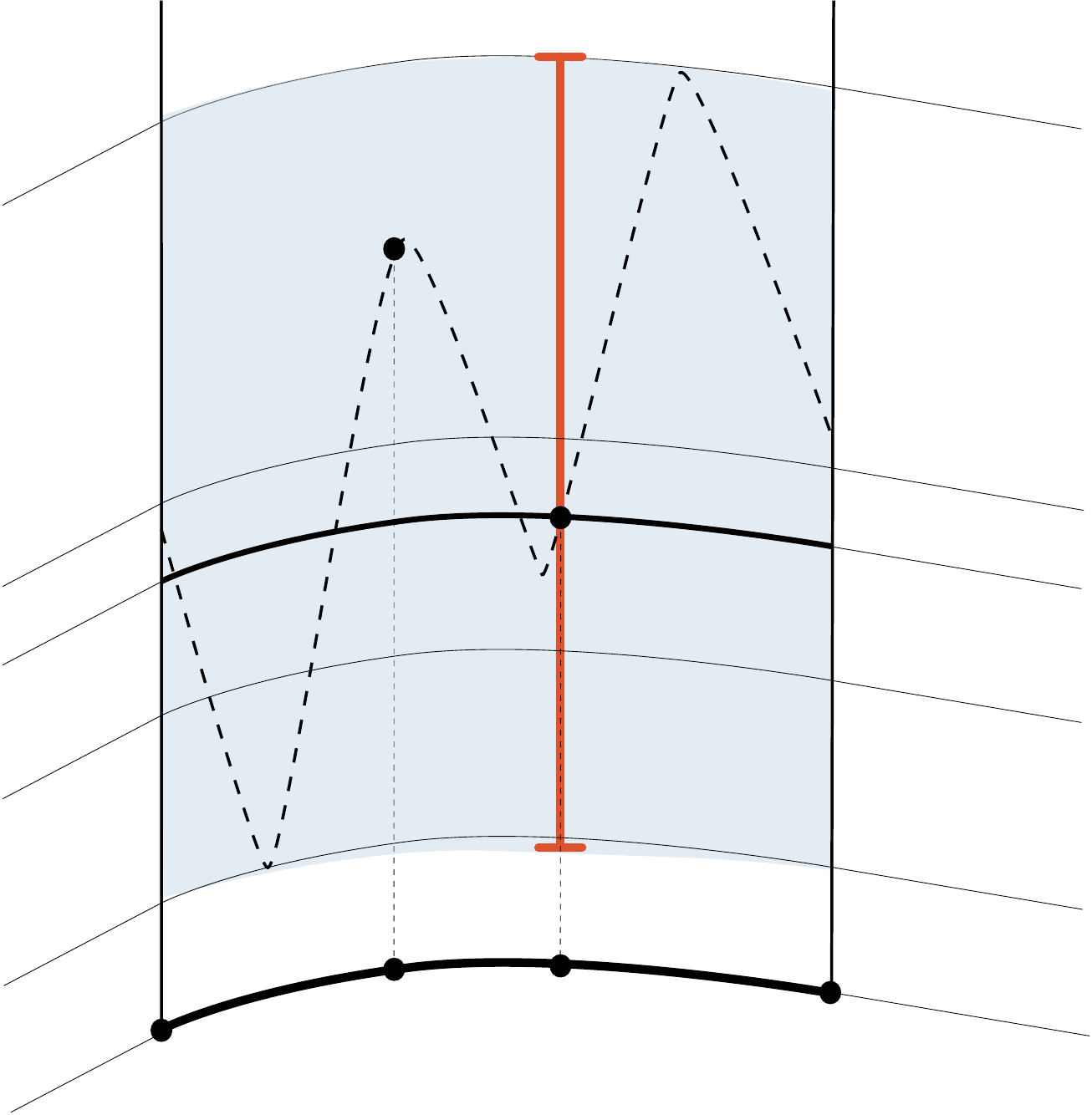}
        \put(101,4){\small $\cW^\u_{\rm loc}(\xi)$}
	\put(49,5){\small $\xi$}
	\put(34,5){\small $\eta$}
	\put(14,2){\small $\zeta$}
	\put(28,76){\small $Y$}
	\put(52,47){\small $X$}
	\put(35,85){\small \textcolor{red}{$\widehat I_n(\xi)$}}
	\put(77,82){\small $\widehat R_n^\u(X)$}
	\put(56,14){\small $R^\u(\xi,n)$}
	\put(8,70){\small $I_\zeta$}
	\put(101,45){\small $\cW^\uu_{\rm loc}(X,T)$}
\end{overpic}
\caption{Definition of the set $\widehat R^\u_n(X)$ (shaded region)}
\label{fig.1}
\end{figure}
\begin{remark}
Notice that the  segment $\widehat I_n(\xi)$, by definition, is bounded by points which are on the local strong unstable manifolds of some points in $\Phi\cap\widehat R^\u_n(X)$.
\end{remark}

Given $Y=(\eta,\Phi(\eta))\in R^\u_n(X)$, denote by 
$\ell_{\rm h}(Y)$ the minimal length of a segment in the fiber containing the set $I_\eta\cap \widehat R^\u_n(X)$. 
Define the \emph{height} of an $n$th level Markov unstable  rectangle by
\[
	\lvert R^\u_n(X)\rvert_{\rm h}\eqdef \max_{Y\in R^\u_n(X)}\ell_{\rm h}(Y).
\]
Define the \emph{width} of an $n$th level Markov unstable  rectangle to be
\[
	\lvert R^\u_n(X)\rvert_\w \eqdef 
	\lvert \underline R^\u_n(\xi)\rvert_\w\eqdef 
	\lvert   R^\u(\xi,n)\rvert,
\]
where $\lvert\cdot\rvert$ denotes the  length of a curve in $M$.

The following estimate of the width and height of a Markov unstable rectangle involves a bounded distortion argument and the invariance of strong unstable manifolds. Recall the definitions of the potentials $\varphi^\u$ and $\varphi^\cu$ in~\eqref{eq:defbasicpotentialvarphius} and~\eqref{e.potentials}.

\begin{proposition}\label{pro:dis0}
	There exists $c>1$ such that for every $X=(\xi,\Phi(\xi))$, $n\ge1$, and $\eta\in \underline R^\u_n(\xi)$ we have
\[
	\frac1c\le
	\frac{\lvert R^\u_n(X)\rvert_\w }{\exp(S_n\varphi^\u(\eta))}
	\le c.
\]	
\end{proposition}

\begin{proof}
Recall that there is $\theta>0$ such that $\varphi^\u$ is $\theta$-H\"older continuous.
	By~\eqref{eq:basehyp} for every $\eta\in\underline R^\u_n(\xi)$ we have $d_1(\eta,\xi)\le c_2 \kappa_\w^{-n}$, where $c_2$ denotes the maximal diameter of a Markov rectangle $\underline R_i$. Hence,  there exists $c_3>0$ such that 
	for every $i=0,\ldots,n-1$ we have
\[
	\lvert\varphi^\u(\tau^i(\eta))-\varphi^\u(\tau^i(\xi))\rvert
	\le c_3\kappa_\w^{-\theta(n-i)}.
\]	
This implies
\[
	\lvert S_n\varphi^\u(\eta)-S_n\varphi^\u(\xi)\rvert
	\le c_3\sum_{i=0}^{n-1}\kappa_\w^{-\theta(n-i)}
	<c_3\sum_{i=0}^\infty\kappa_\w^{-\theta i}
	=: c_4
	<\infty.
\]
Thus, by the mean value theorem and the above, there exists $\eta'\in \cW^\u_{\rm loc}(\xi,\tau)\cap R^\u(\xi,n)$ such that 
\[
	\lvert R^\u_n(X)\rvert_\w
	=\lvert R^\u_0(T^n(X))\rvert_\w\,\lVert d\tau^n|_{F^\u_{\eta'}}\rVert^{-1}
	=\lvert R^\u_0(T^n(X))\rvert_\w\, e^{S_n\varphi^\u(\eta')}.
\]
Since there are only finitely many Markov rectangles and each of them has nonempty interior, the widths of Markov unstable rectangles $\lvert R^\u(\cdot)\rvert_\w$ are uniformly bounded from below and above by positive numbers. 
This proves the proposition.
\end{proof}

We also have the following estimate for the height of Markov unstable rectangles. Its proof follows an alternative, perhaps more conceptual, way to control the size of Markov rectangles in comparison to the approach in~\cite{Bed:89}. 

\begin{proposition}\label{pro:dis}
If $\Phi$ is Lipschitz on local unstable manifolds, then for every $X\in\Phi$ and $n\ge0$ we have 
\[
	\lvert R^\u_n(X)\rvert_{\rm h}=0.
\]	
Otherwise, if $\Phi$ is not Lipschitz on local unstable manifolds, then there exists $c>1$ such that for every $X=(\xi,\Phi(\xi))$, $n\ge1$, and $\zeta\in \underline R^\u_n(\xi)$ we have
\[
		\frac1c
	\le \frac{\lvert R^\u_n(X)\rvert_{\rm h}}{\exp(S_n\varphi^\cu(\zeta))}
	\le c.
\]
\end{proposition}

\begin{proof}
By Proposition~\ref{prop:critical}, $\Phi$ is Lipschitz on local unstable manifolds if, and only if, there exists $\delta>0$ such that $\Delta^\u_\delta(\xi)=0$ for every $\xi\in\Xi$ and, in particular, the graph is contained in the local strong unstable manifold at every point $X=(\xi,\Phi(\xi))$. 
This immediately implies that the above defined height of an Markov unstable rectangle is $0$ at every $\xi\in\Xi$.

If $\Phi$ is not Lipschitz on local unstable manifolds, then by Corollary~\ref{cor:erika} for every $\delta>0$ there is $C(\delta)>0$ such that $\Delta^\u_\delta(\xi)\ge C(\delta)$ for every $\xi$. 
Now given $X=(\xi,\Phi(\xi))$ and $n\ge 1$, by the Markov property~\eqref{propsMar} we have
\[
	T^n\big(R^\u_n(X)\big)
	= T^n\big(R_n(X)\cap\cW^\u_{\rm loc}(X,T)\big)
	\supset R(T^n(X))\cap\cW^\u_{\rm loc}(T^n(X),T).
\]
In particular, it  contains some point $Y'=(\eta',\Phi(\eta'))\in \cW^\u_{\rm loc}(X',T)$,   where $X'=(\xi',\Phi(\xi'))=T^n(X)$, so that $d_1(\eta',\xi')\le\delta$ and $\lvert\gamma^\u_{\xi'}(\eta')-\Phi(\eta')\rvert\ge C(\delta)$. Note that 
$\lvert\gamma^\u_{\xi'}(\eta')-\Phi(\eta')\rvert$ is the distance between the point of intersection of the local strong unstable manifold through $X'$ with the fiber $I_{\eta'}$ and the point $Y'$. Preimages by $T^{-k}$ of these points are both in the common fiber $I_{\tau^{-k}(\eta')}$, for any $k\ge1$.

Since the fiber maps are uniformly H\"older and $T^{-1}$ is uniformly fiber contracting, we can find a constant $D>1$ (independent of $\xi\in\Xi$) such that 
\[
	\lvert T^{-n}_{\eta'}(\gamma^\u_{\xi'}(\eta'))-T^{-n}_{\eta'}(\Phi(\eta'))\rvert
	\ge D^{-1}\cdot \lvert (T^{-n}_{\eta'})'(\Phi(\eta'))\rvert
		\cdot \lvert \gamma^\u_{\xi'}(\eta')-\Phi(\eta')\rvert.
\]
And since $\tau^{-k}$ exponentially contracts the distance between $\eta'$ and $\xi'$, with $\eta=\tau^{-n}(\eta')$  we also obtain
\[
	\lvert  (T^{-n}_{\eta'})'(\Phi(\eta'))\rvert
	= \lvert  (T^{n}_\eta)'(\Phi(\eta))\rvert^{-1}
	\ge D^{-1}\lvert  (T^n_\xi)'(\Phi(\xi))\rvert^{-1}.
\]
In fact, in this inequality we can replace $\xi$ by any point $\zeta$ in $\underline R^\u_n(\xi)$. Finally, recalling the definition~\eqref{e.potentials} of $\varphi^\cu$ we obtain
\[
	\lvert R^\u_n(X)\rvert_\h
	\ge \lvert T^{-n}_{\eta'}(\gamma^\u_{\xi'}(\eta'))-T^{-n}_{\eta'}(\Phi(\eta'))\rvert
	\ge  e^{S_n\varphi^\cu(\zeta)}\cdot D^{-2}\cdot C(\delta).
\]
The upper bound follows analogously recalling that $\Phi$ is compact and hence the height of the initial Markov rectangles is uniformly bounded from above.
\end{proof}

%--------------------------------------------------------------------------------------------------------
\section{Fibered blenders}\label{sec:Blenders}
%--------------------------------------------------------------------------------------------------------

A \emph{blender} (see~\cite{BonDia:96} and~\cite{BocBonDia:}) is a hyperbolic and partially hyperbolic set $\Lambda$ of a diffeomorphism $\cT$ with splitting $E^\s\oplus E^{\cu}\oplus E^{\uu}$
(where $E^\s$ is the stable bundle and $E^{\cu}\oplus E^{\uu}$ the unstable one)
being locally maximal in an open neighborhood which has an additional special structure. Namely there is a strong unstable (expanding) cone field $\Cuu$ around the strong unstable bundle $E^{\uu}$ and an open  family $\cD$ of disks, called \emph{blender plaques} or simply \emph{plaques}, tangent to $\Cuu$ that satisfies the following invariance and covering properties: every $D\in\cD$ contains a subset $D_0$ such that $\cT (D_0)\in\cD$. 

Note that every plaque of a blender intersects the local stable manifold of $\Lambda$ (defined analogously to~\eqref{def:stamaniset} below), see Lemma~\ref{l.bobodi} below and its versions in~\cite{BocBonDia:}. 
Though there are points in $\Lambda$ whose strong unstable manifold has nothing to do with the blender in the sense that it does not contain a blender plaque.%
\footnote{An example for the hyperbolic set $\Phi_t$, $t\ne0$ small, in Section~\ref{sss.fiberedblenders} is given by the ``boundary" strong unstable manifold of the fixed point $P_0^t=(0,0,0)$ (compare Figure~\ref{fig.2}).} 
It is essential in our arguments that the family of plaques of the blender is sufficiently big assuring that the ``dynamics of the plaques" and the dynamics of the blender  are related and that the plaques capture an essential part of its dynamics.  This leads to a blender with the  \emph{germ property} defined below. In Remark~\ref{rem:blenders} we will compare these notions with other related ones in the literature.

In the definition of a blender, the family $\cD$ is open in the ambient space. However, here for our purpose it is enough to consider a (sub-)family of discs in the unstable manifold of $\Lambda$. This leads to a \emph{fibered blender} defined in Section~\ref{sec:germprop}.

In Section~\ref{ss.fiberedbh}, following the definition of a blender-horseshoe in~\cite[Section 3]{BonDia:12}, we will introduce a class of fibered blenders which have the germ property and are topologically conjugate to a shift in $N$ symbols. We will call them \emph{fibered blender-horseshoes}. 
It is easy to verify that
the horseshoes $\Phi_t$, $t\ne 0$, in 
Section~\ref{sss.fiberedblenders} are   examples of such objects (indeed they are the paradigmatic examples), see also Example~\ref{ex:blender:1}.
We note that the fibered blender-horsehoes are nonaffine generalizations of these affine horseshoes  (as were also the blender-horseshoes in \cite{BonDia:12}).

\begin{figure}
 \begin{overpic}[scale=.3]{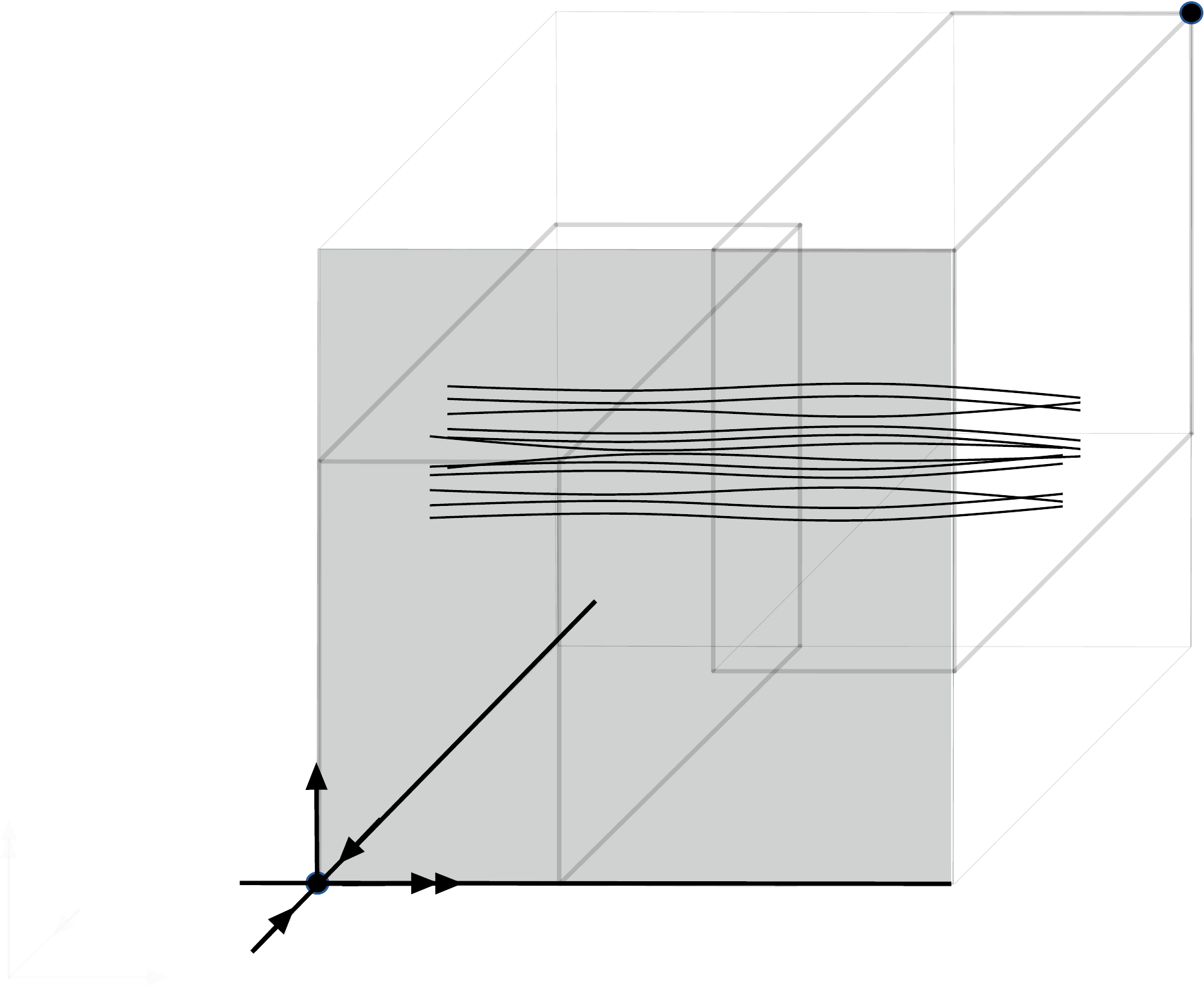} 
 	\put(95,43){\small blender plaques}
         \put(26,3){\small $P$}
         \put(81,8){\small $\cW^{\uu}_{\rm loc}(P,T)$}
         \put(1,56){\small $\cW^{\u}_{\rm loc}(P,T)$}
 \end{overpic}
\caption{A blender for the example in Section~\ref{sss.fiberedblenders}: the almost horizontal plaques cover the intersection region; the local  unstable manifold of $P=P_0^t$ (shaded region) contains the local strong unstable manifold of $P$ (horizontal line $\xi^\s=0$) which does not contain any blender plaque.}
\label{fig.2}
\end{figure}
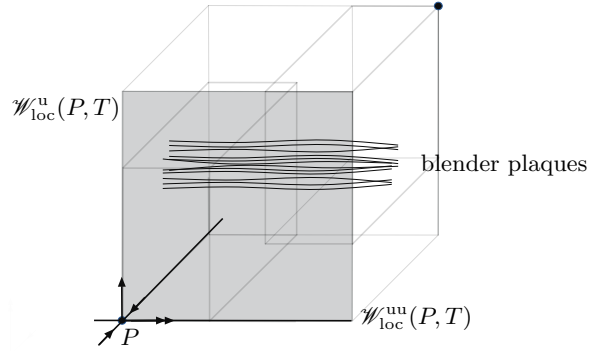

In what follows, we continue with the fibered setting from Section~\ref{sec:introduction}.

%--------------------------------------------------------------------------------------------------------
\subsection{Fibered blender}\label{sec:unstblend}
%--------------------------------------------------------------------------------------------------------

As in Section~\ref{sec:introduction}, let $U\subset M$ be neighborhood of $\Xi$ such that $\Xi=\bigcap_{k\in\bZ}\tau^k(U)$. Recall the definition of a local unstable manifold $\cW^\u_{\rm loc}(\xi,\tau)$ of a point $\xi\in\Xi$ (with respect to $\tau$ and $U$) in~\eqref{def:unsmani}. Recall that the set $\Phi$ is an invariant graph. Given $X=(\xi,\Phi(\xi))$, recalling Section~\ref{sec:2manif}, let
$$
	\cW^{\u}_{\rm loc}(X,T) 
	\eqdef  \cW^\u_{\rm loc}(\xi,\tau) \times I,
$$
where $I\subset\bR$ is some open interval.
Observe that this set indeed is contained in a local unstable manifold of $\Phi$ (with respect to $T$). Let
\begin{equation}\label{def:stamaniset}
	\cW^{\u}_{\rm loc}(\Phi,T)
	\eqdef \bigcup_{X\in \Phi}  \cW^{\u}_{\rm loc}(X,T) .
\end{equation}

We fix a \emph{cone field} $\Cuu$ around the bundle $E^\uu$
which is strictly invariant and uniformly expanding. More precisely, for every $X\in\Phi$ the open cone $\Cuu_X\subset T_X(M\times \bR)$ contains $E^\uu_X$ and the image of its closure under $dT_X$ is contained in $\Cuu_{T(X)}$ and $dT_X$ uniformly expands vectors in $\Cuu_X$. We assume that this cone field can be extended to the neighborhood $U\times I$ keeping the $dT$ invariance and expansion properties (here we assume that $X$ and $T(X)\in U\times I$).%
\footnote{Note that this cone field can be extended to \emph{any} small neighborhood of $\Phi$ keeping the invariance and expansion properties. The point here is that the neighborhood $U$ is fixed \emph{a priori}.} We continue to denote the extension
of this cone field by $\Cuu$.

A \emph{plaque} associated to the cone field $\Cuu$ is a finite union $D=\bigsqcup_{i=1}^m D_i$   of pairwise disjoint closed $C^1$-curves $D_i$ 
(the \emph{decomposition} of $D$)
homeomorphic to  closed intervals and tangent to the cone field $\Cuu$
(i.e., $T_X D \subset \Cuu_X$).
Given a curve $D_i$, we denote by $\lvert D_i\rvert$ its length and define the length of a plaque $D$ by $\lvert D\rvert=\sum_{i=1}^m\lvert D_i\rvert$.

\begin{definition}[Fibered blender]
A family $\fB=(\Phi,U\times I,\Cuu,\cD)$  is a \emph{fibered blender} for $T$ if it satisfies the following properties:
the set $\Phi$ is hyperbolic, partially hyperbolic, and locally maximal in $U\times I$. The cone field $\Cuu$ is a strong unstable expanding one-dimensional cone field defined on $U\times I$  which is forward invariant. The family $\cD$ is a family of 
 plaques associated to $\Cuu$  satisfying the following \emph{relative}, \emph{open and covering}, and \emph{expanding} properties:

\begin{enumerate}
\item[\textbf{FB1}] (relative) Every plaque $D\in \cD$ is contained in $\cW^{\u}_{\rm loc}(\Phi,T)$ and 
its decomposition $D=\bigsqcup_{i=1}^m D_i$ is such that $D_i\subset \cW^{\u}_{\rm loc}(X_i,T)$ for some $X_i\in \Phi$;
\item[\textbf{FB2}] (open and covering) There is $\varepsilon_{\cD}>0$ such that for every plaque $D'$ 
$\varepsilon_\cD$-close to some plaque $D \in \cD$ and contained in
$\cW^{\u}_{\rm loc}(\Phi,T)$ the set $T(D')$ contains a plaque in $\cD$.
\item[\textbf{FB3}] (expanding) There is $\varkappa> 1$ such that for every  plaque $D\in \cD$ there is  $D_0\subset D$ such that  $\lvert D_0\rvert \le \varkappa^{-1} \lvert D\rvert$ and $T(D_0)\in \cD$.
\end{enumerate}
\end{definition}

\begin{remark}[Blenders and fibered blenders]\label{rem:blenders}
The term \emph{fibered} refers to the fact that we consider the fibered setting from Section~\ref{sec:introduction}. 
The term \emph{relative} refers to the fact that we consider only plaques contained in local unstable manifolds.
As in \cite{BocBonDia:}, a fibered blender is persistent by perturbations preserving the fiber structure
(see also Proposition~\ref{pro:persistent}). We will provide in Section~\ref{ss.fiberedbh} important examples where a 
fibered blender is persistent.

Note that the role of the set $\Phi$ in the definition of a fibered blender is in some sense only instrumental and the important objects are the plaques. In general, the set of plaques could be small in the sense of
not capturing all the dynamics of $\Phi$. For instance, 
there could be unstable leaves of $\Phi$ that do not contain any plaque.

Note that in our setting of hyperbolic graphs  the important object is the  set $\Phi$.
Bearing this in mind,
below for fibered blenders we will introduce a  \emph{germ  property} that guarantees that 
the hyperbolic set of a
fibered blender  has a sufficiently rich family of plaques that captures all the dynamics
of the hyperbolic set and  is also sufficiently rich to go on in the dimension arguments. 
 We need that essentially every leaf  $\cW^\u_{\rm loc}(X,T)$, $X\in \Phi$, contains some plaque
(this is the meaning of the term ``capture").
\end{remark}

The following key result is well known in the realm of blenders, see
~\cite[Remark 3.12 and Lemma 3.14]{BocBonDia:}, for completeness we include its proof.

\begin{lemma}\label{l.bobodi}
	Let $\fB=(\Phi,U\times I,\Cuu,\cD)$ be a fibered blender. Then every $D\in \cD$ intersects $\cW^\s_{\rm loc}(\Phi, T)$.
\end{lemma}

\begin{proof}
Let $D\in\cD$. We define a nested sequence of subsets of $D$ as follows. Let $D_0$ be a subset given by item FB3) in the definition of a fibered blender. Assume that we have already defined subplaques $D_k\subset D_{k-1}\subset\ldots\subset D_0$  such that $T^{i+1}(D_i)\in\cD$, for every $i=0,\ldots,k$. Then let $\hat D_{k+1}$ be a subset of $T^{k+1}(D_k)\in\cD$ given by FB3), that is $T(\hat{D}_{k+1})\in \cD$, and let $D_{k+1}=T^{-(k+1)}(\hat D_{k+1})$. By construction and FB3), there exists a point $z\in\bigcap_{k\ge0}D_k$ such that its forward orbit $\{z,T(z),\ldots\}$ is contained in $U\times I$.
% \margem{check local maximal of $\sigma$ in $U$ and $T$ in $U\times I$}
  Hence $z\in\cW^\s_{\rm loc}(\Phi,T)$.  
\end{proof}

The following result is an immediate consequence of the above lemma based
on the relative property of a fibered blender.

\begin{corollary}\label{c.bobodi}
		Let $\fB=(\Phi,U\times I,\Cuu,\cD)$ be a fibered blender. Then every $D\in \cD$ contains a point in $\Phi$.
\end{corollary}

\begin{proof}
Given $D\in\cD$, by the relative property we have $D\subset \cW^\u_{\rm loc}(\Phi,T)$. By  Lemma~\ref{l.bobodi}, $D$ contains a point $X$ in the local stable manifold of $\Phi$. Hence $X\in \cW^\u_{\rm loc}(\Phi,T)\cap \cW^\s_{\rm loc}(\Phi,T)$. Since  $\Phi$ is locally maximal in  $U\times I$, we get $X\in\Phi$.
\end{proof}

%--------------------------------------------------------------------------------------------------------
\subsection{Germ property}\label{sec:germprop}
%--------------------------------------------------------------------------------------------------------

We will now explore in more detail the Markov structure in the fibered blender.
Recall the notation in 
 Section~\ref{sec:markov}.
Given $\xi\in\Xi$,  recall that $R^\u(\xi,m)$ denotes the minimal curve contained in $\cW^\u_{\rm loc}(\xi,\tau)$ containing the $m$th level Markov unstable rectangle $\underline R^\u_m(\xi)$. Consider the \emph{$m$th level $\u$-box} of $\xi$ defined by
\[
	B^\u_m(\xi)
	\eqdef R^\u(\xi,m)\times I.
\]
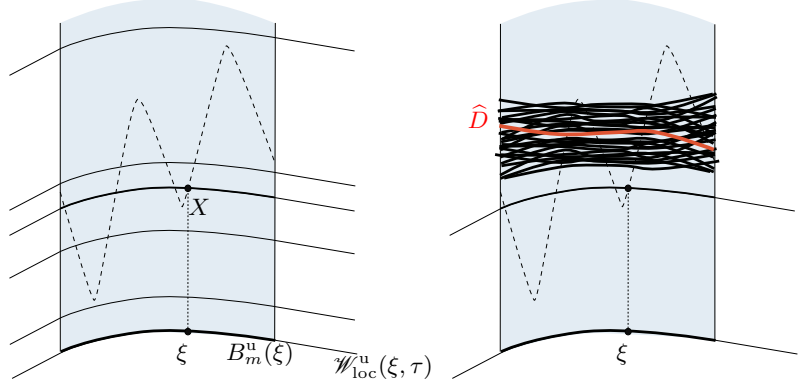
\begin{figure}
\begin{overpic}[scale=.35]{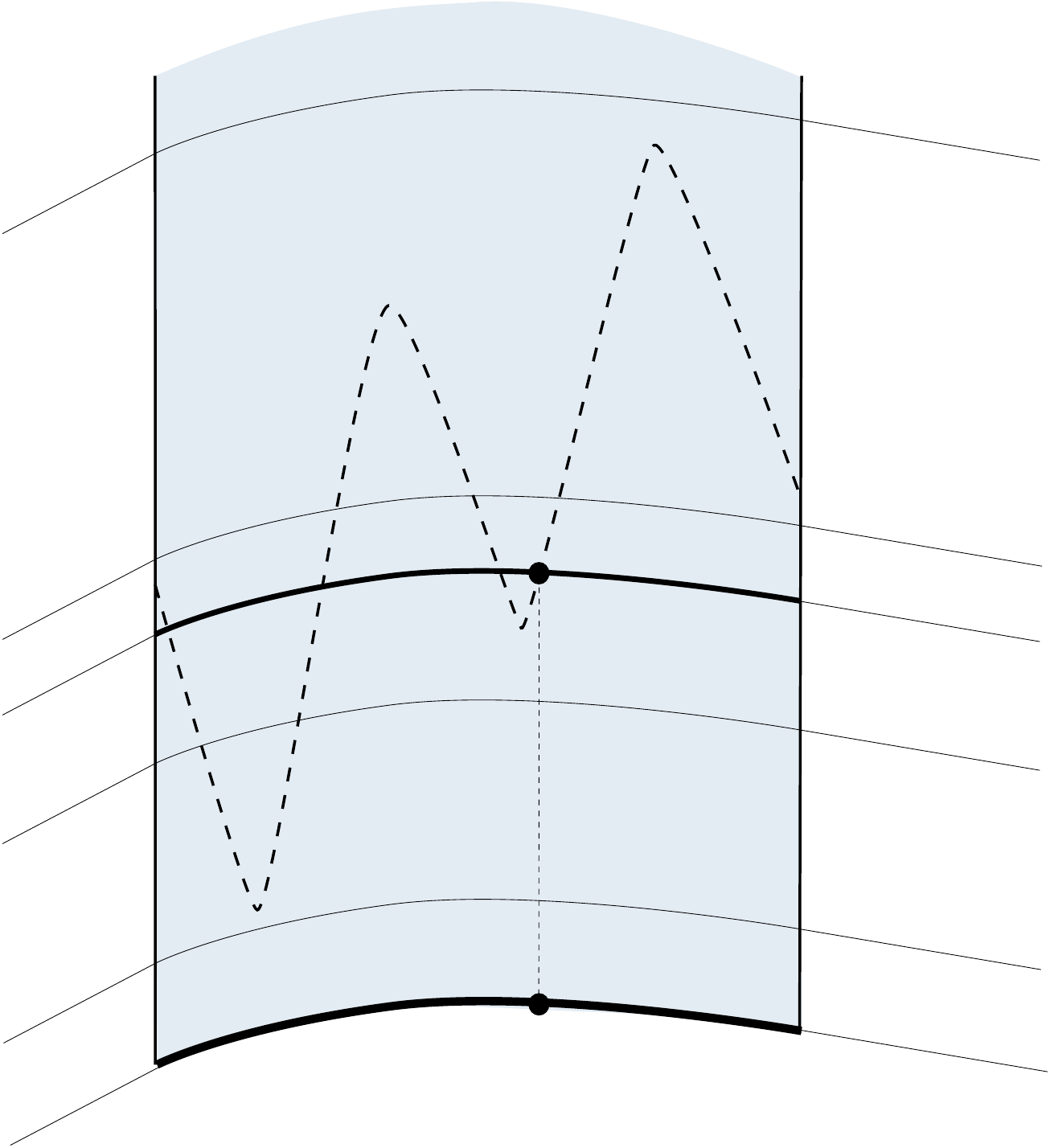}
        \put(44,5){\small $\xi$}
        \put(47,43){\small $X$}
        \put(85,1.7){\small $\cW^\u_{\rm loc}(\xi,\tau)$}
        \put(57,5){\small $B^\u_m(\xi)$}	
 \end{overpic}
 \hspace{1cm}
\begin{overpic}[scale=.35]{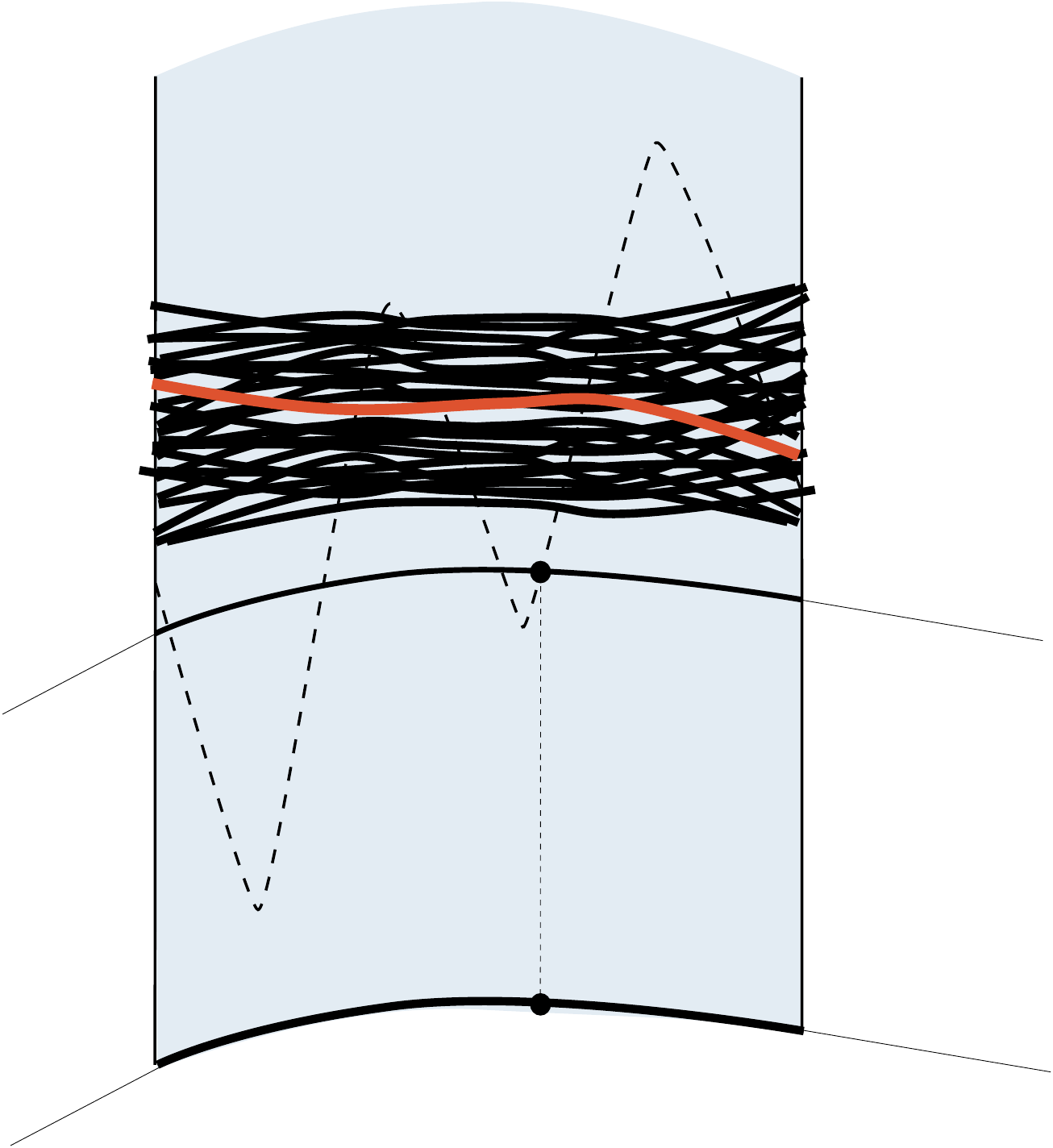}
        \put(44,5){\small $\xi$}
        \put(5,66){\small \textcolor{red}{$\widehat D$}}
 \end{overpic}
\caption{Left: $m$th level $\u$-box of $\xi$, $B^\u_m(\xi)$ (shaded region). Right: germ plaque $\widehat D$ of a well-placed germ plaque in a relative blender}
\label{fig.3}
\end{figure}

Given $X\in\Phi$, for a set $C\subset \cW^\u_{\rm loc}(X,T)$ and $Y=(\eta,y)\in C$, let $I(C,Y)$ denote the connected component of $C\cap I_\eta$ which contains $Y$, where $I_\eta=\{\eta\}\times \bR$. Let 
\[
	\lvert C\rvert_\h
	\eqdef \inf\big\{\lvert I(C,Y)\rvert\colon Y\in C\big\}.
\]

\begin{definition}[Germ property]\label{def:3}
	A fibered blender $\fB=(\Phi, U\times I,\Cuu,\cD)$ has the \emph{germ property}  if there  is $\delta>0$ such that
	for every $\xi\in\Xi$ 
and every $m\ge0$ the $m$th level $\u$-box $B=B^\u_m(\xi)$ satisfies the following properties:
\begin{itemize}
\item[(a)] 
   There is a closed curve $J_m=J_m(\xi)\subset I_\xi$ such that every $Z\in J_m$ is in some plaque of $\cD$.
\item[(b)] 
Let $D_Z\in\cD$ be any plaque containing a point $Z\in B$ and denote by $\hat D_Z$  the connected component of $D_Z\cap B$ containing $Z$.
There is a family of plaques $ D_Z$, $Z\in J_m$ such that $\hat\cD_\fR\eqdef \{\hat D_Z \colon Z\in J_m\}$  is a continuous foliation of the set
\[
	\fR
	\eqdef \bigcup_{Z\in J_m}\hat D_Z.
\]	
\item[(c)] We have	
\begin{itemize}
\item $\lvert T^m(\fR)\rvert_\h\ge\delta$,
\item $T^i(\fR)\subset U\times I$ for all $i=0,\ldots,m$, and
\item for every $\hat D_Z$, $Z\in J_m$, we have $T^m(\hat D_Z)\in\cD$  
\end{itemize}
\end{itemize}
(compare Figure~\ref{fig.3}). We call any such $\fR$ a \emph{germ rectangle} and $\hat D_Z$ a \emph{germ plaque} (with respect to $B$).
\end{definition}

\begin{remark}%\label{r.leitorcritico}
Property (a) says that the fibered blender is sufficiently rich capturing the dynamics of the hyperbolic set in the sense that for \emph{every} $\xi\in\Xi$ the local unstable leaf of $X=(\xi,\Phi(\xi))$ (with respect to $T$) contains blender plaques.

Note also that the choice of a plaque $D_Z$ in $\cD$ is arbitrary, we require that the above holds for any such 
a choice and some continuation to a  foliation by plaques in $\cD$.

Finally,  (b) and (c) will imply (see Corollary~\ref{c.germcorollary}) that the fibered blender is sufficiently rich in the sense that (locally) in every $m$th level $\u$-box the projection of the points of $\Phi$ in this box is an interval whose image  blows up to a uniform minimal size (height, that is, its size in the center unstable fiber direction).  
\end{remark}

As an immediate consequence of  Corollary~\ref{c.bobodi} we get the following result which is the key ingredient to show the lower bound for the box dimension (see Claim~\ref{cl:claimkey}). Given $\xi\in\Xi$, $m\ge0$, $B^\u_m(\xi)$, $J_m$, and a germ rectangle $\fR=\bigcup_Z \hat D_Z$ and its associated foliation $\hat \cD_\fR=\{\hat D_Z\}$  as in Definition~\ref{def:3}, denote by $\pi_{\hat\cD_\fR}\colon \fR\to J_m$ the projection along the leaves of this foliation. Then we have
\[	
	\pi_{\hat\cD_\fR}\big(\Phi\cap \fR\big)=J_m.
\]	
We get the following corollary.

\begin{corollary}\label{c.germcorollary}
	Let $\fB=(\Phi, U\times I,\Cuu,\cD)$ be a fibered blender with the germ property.
Then any   germ  plaque of a germ rectangle  contains a point of $\Phi$.
\end{corollary}

%--------------------------------------------------------------------------------------------------------
\subsection{Fibered blender-horseshoes}\label{ss.fiberedbh}
%--------------------------------------------------------------------------------------------------------

In this section we translate the definition of a blender-horseshoe in \cite{BonDia:12} to our 
 fibered setting and provide examples of fibered blenders. This translation follows closely the exposition in \cite[Section 3.2]{BonDia:12} (though here we consider a horseshoe conjugate to a shift of $N$ symbols instead of one conjugate to a shift of two symbols only, 
moreover (as observed above) we will focus on local unstable manifolds only, see FBH3) and FBH4). The rough idea of our definition is that every blender-horseshoe with a skew product structure is a fibered blender
with the germ property and 
hence provides a class of examples of the objects considered in Sections~\ref{sec:unstblend} and
\ref{sec:germprop}.

We assume that the hyperbolic set $\Phi$ and the open sets $U\subset M$, $I\subset \bR$ are as
in Section~\ref{sec:2manif}. In what follows we assume that $\Xi$ (and hence $\Phi$)  is conjugate to the full shift in
$N$ symbols. We will discuss  more general cases at the end of this section. We consider the following properties.

\begin{itemize}
\item [\textbf{FBH1}] There are two fixed points $P_0$ and $P_{1}\in\Phi$, called \emph{reference saddles},  such that for
any point $X$ in $\Phi$ the 
local stable manifolds $\cW^\s_{\rm loc}(P_i,T)$ each intersects 
$\cW^\u_{\rm loc}(X,T)$ in just one point that we denote by $X_i$, $i=0,1$.
 
\item [\textbf{FBH2}] 
There is a continuous cone field $\Cuu$  around the bundle $E^\uu$ defined on $U\times I$
which is strictly invariant and uniformly expanding. 
\end{itemize}

A curve contained in $\cW^\u_{\rm loc}(X,T)$ for some $X\in\Phi$ and tangent to $\Cuu$ is called a \emph{$\Cuu$-curve}.

\begin{itemize}
\item [\textbf{FBH3}]  There is a continuous family $\{\fR(X)\}_{X\in\Phi}$ of closed sets (rectangles) such that
\begin{itemize}
\item[(i)] $\fR(X)\subset\cW^\u_{\rm loc}(X,T)$ and $\fR(Y)=\fR(X)$ if $\cW^\u_{\rm loc}(Y,T)=\cW^\u_{\rm loc}(X,T)$,
\item[(ii)] each rectangle is bounded by four curves: $D^{+,-}=D^{+,-}(X)$ are $\Cuu$-curves and $F^{\ell,r}=F^{\ell,r}(X)$ are contained in fibers,
\item[(iii)] $X_0,X_1$ are contained in the interior of $\fR(X)$, and
\item[(iv)] every $\Cuu$-curve containing $X_0$ or $X_1$ is disjoint from $D^-\cup D^+$.
\end{itemize}
\end{itemize}

A $\Cuu$-curve $D\subset \fR(X)$ for some $X\in\Phi$ is  \emph{$\uu$-complete} if it  intersects both (boundary) fibers $F^\ell(X)$ and $F^r(X)$ of $\fR(X)$. A $\Cuu$-curve in $\fR(X)$ is \emph{well located} if it is $\uu$-complete and disjoint from $D^\pm(X)$.

\begin{itemize}
\item [\textbf{FBH4}] 
For  every $X\in \Phi$ any pair of  $\Cuu$-curves containing $X_0$ and $X_1$, respectively, are disjoint.
In particular, for every $X\in\Phi$, there is no $\uu$-complete $\Cuu$-curve in $\fR(X)$ simultaneously containing $X_0$ and $X_1$.
\end{itemize}

Every well located $\Cuu$-curve $D$ in $\fR(X)$ splits 
this rectangle into two connected components which we denote by $C^\pm= C^\pm(D)$,
$\fR(X)\setminus D=C^+(D)\cup C^-(D)$.
There  are three pairwise different possibilities:
\begin{itemize}
\item 
either
 $X_0\in D$ or
$X_1\in D$,
\item 
$X_0$ and $X_1$ are in the same component $C^\pm$, and
\item
$X_0$ and $X_1$ are in different components $C^-$ and $C^+$.
\end{itemize}
A   $\Cuu$-curve which is well located and satisfies the first possibility is \emph{extremal}. One satisfying the last possibility is \emph{in between} $\cW^\s_{\rm{loc}} (P_0,T)$ and $\cW^\s_{\rm{loc}} (P_1,T)$, or simply, \emph{in between}. 

We denote by $\cD$ the family of all well located $\Cuu$-curves in between. This family will turn out to be   the invariant family of plaques in a fibered blender.
 We denote by $\cD^\ext$ the family of 
$\Cuu$-curves in $\cD$ together with all extremal $\Cuu$-curves.

We will need a geometrically more precise version  of the covering property in the definition of a fibered blender ``every plaque $D\in\cD$ contains a subplaque $\hat D$ such that $T(\hat D)\in\cD$''.
Let us state this condition more precisely. We call a \emph{strip} a set which is homeomorphic to a rectangle foliated by curves in $\cD^\ext$. To emphasize the role of the foliating curves for a strip $S$, we write $S=\{D_t\}_{t\in[0,1]}$.
  A strip is \emph{$\u$-complete} or simply \emph{complete} if its boundary contains a pair of (different) extremal $\Cuu$-curves.  The \emph{height} of a strip $S$ is $h(S)\eqdef \max\{ |I_\xi\cap S|_\h\colon \xi \in\Xi\}$. A set $S'$ is a \emph{substrip} if it is a subset of some strip and  $T(S')$ is again a strip. 

Given a strip $S=\{D_t\}_{t\in[t_1,t_2]}$ we consider an extension of it to a complete strip $S'=\{D_t'\}_{t\in[0,1]}$ (i.e., $D_t'=D_t$ for every $t\in [t_1,t_2]$). If the strip $S$ is not complete, its extension to a complete strip is not unique. Note that these extensions always exist. Note that, by construction, every strip contains at least one substrip.

\begin{remark}
\label{r.hsize}
By condition FBH4),
	there is $\nu>0$ such that the height of every complete strip is at least $\nu$.  
\end{remark}

\begin{itemize}
\item [\textbf{FBH5}] 
%There is $\delta>0$ with the following property.
For every $D\in\cD^\ext$ and every complete strip $S=\{D_t\}_{t\in[0,1]}$ 
such that $D=D_s$ for some $s\in [0,1]$ there  are parameters
$0=t_0<t_1<t_2<\cdots <t_N=1$ such that:
\begin{itemize}
\item for each $i=0,\ldots, N-1$ there exists a family of subsets $D_t^i$ of $D_t$, $t\in [t_i,t_{i+1}]$ such  that 
%\margem{added $\bigcup$}
$\bigcup_{t\in [t_i,t_{i+1}]} T(D^i_t)$ is a complete strip,
\item	the union of  (disjoint) substrips 
\[
S_i\eqdef\bigcup_{t\in[t_i,t_{i+1}]}D_t^i, 
\quad i=0,\ldots, N-1
\] 
covers $\Phi\cap\cW^\u_{\rm loc}(X,T)$ for some $X\in \Phi$.
%\item
%The height of the strip  $S'= \{D_t\}_{t\in [t_1,t_2]}$ is at least $\delta$.
%\margem{central strip of $S$}
%\item
%fibered expansion. \margem{necessitamos dizer que isto \'e tudo e que o resto escapa...}
%\margem{cont cone field implies superposition}
\end{itemize}
\end{itemize}

\begin{remark}\label{r.cuuexp}
Since the foliating $\Cuu$-curves are tangent to the uniformly expanding cone field $\Cuu$
we have $\lvert D^i_t\rvert\le \varkappa^{-1} \lvert D_t\rvert$
%\margem{changed to $D^i_t$}
% $\lvert \hat D_t\rvert\le \lambda \lvert D_t\rvert$ and 
%$\lvert\check D_t\rvert\le \lambda \lvert D_t\rvert$
 for some $\varkappa>1$.
\end{remark}

\begin{remark}\label{r.BMarkov}
Given any $X=(\xi,\Phi(\xi))$ and its rectangle $\fR(X)$, 
for every complete strip $S$ in
$\fR(X)$ and every $m\ge 0$ it holds
$B^\u_m(\xi)\cap \Phi\subset S_i$ for some $i=0,\ldots, N-1$. 
Note also that $(B^\u_m(\xi)\cap \Phi)\cap  S_j=\emptyset$ for every $j\ne i$.
%Given any $X=(\xi, x) \in \Lambda$, its rectangle $R(X)$, any complete strip $S$ in
%$R(X)$, and $m\ge 0$ either 
%$(B^\u_m(\xi)\cap \Lambda)\subset \hat S$ or 
%$(B^\u_m(\xi)\cap \Lambda)\subset  \hat S$. Note also
%that
%$(B^\u_m(\xi)\cap \Lambda)\cap \hat S=\emptyset$ or 
%$(B^\u_m(\xi)\cap \Lambda)\cap  \hat S=\emptyset$.
\end{remark}

\begin{definition}[Fibered blender-horseshoe]
A hyperbolic and partially hyperbolic set $\Phi$ 
satisfying the Standing hypotheses  in Section~\ref{ss:setttting}
is a \emph{fibered blender-horseshoe} if there are 
an open set $U\times I$, a cone field $\Cuu$, and a family of $\Cuu$-curves $\cD$ satisfying conditions FBH1)--FBH5). We call the family $\cD$ the family of \emph{plaques} of the fibered blender-horseshoe.
\end{definition}

\begin{remark}[Fibered blender-horseshoes \emph{versus} blender-horseshoes]%\label{r.comparision} 
We briefly compare the definition of a fibered blender-horseshoe with 
the original blender-horseshoe in \cite{BonDia:12}.
Our definition is a translation of that definition to our setting where some conditions are 
written in a more suitable way adapted to our needs.
Properties FBH1)--5) are related to properties BH1)--6) in the definition of a blender-horseshoe in~\cite{BonDia:12}
as follows:
property FBH1) is a consequence of the Markov properties in BH1). 
Properties FBH2) and FBH4) correspond exactly to BH2) and BH4), respectively.
FBH3) is a consequence of BH3)--4).
Finally, FBH5) joins BH5) and BH6). 

Let us observe that property FBH5) implies that a fibered blender-horseshoe is a
fibered blender.
Properties FBH1)-FBH4) are used to guarantee the germ property. Also observe that condition 
FBH1) has only an instrumental role to define the family of plaques of the blender. 
%Condition FBH2) can be relaxed omitting the existence of the expanding conefield
%$\Cu$ and just requiring that the rectangles $R(X)$ are contained in local unstable manifolds
%of the hyperbolic set.
\end{remark}

Recall that given a hyperbolic set $\Phi$ of a diffeomorphism $T$ in the manifold $M\times \bR$ there is a $C^1$ neighborhood $\cU$ of $T$ such that every  $\tilde T\in \cU$ has a hyperbolic set $\Phi_{\tilde T}$ which is close to $\Phi$ and such that the restriction of $\tilde T$ to 
$\Phi_{\tilde T}$ is conjugate to the restriction of $T$ to $\Phi$, we call $\Phi_{\tilde T}$ the \emph{continuation} of $\Phi$ for $\tilde T$. Analogously, we denote by $P_{\tilde T}$ the continuation of a saddle $P$ for $\tilde T$.
We are interested in the particular fibered context. Denote by $\Diff^1_{\rm skew} (M\times \bR)$
the subset of the space of $C^1$ diffeomorphisms $\Diff^1 (M\times \bR)$ consisting
of diffeomorphisms with the skew product structure in \eqref{e.skew-product-structure}.
Following the arguments in \cite[Lemma 3.9]{BonDia:12} we get the following.

\begin{proposition}\label{pro:persistent}
Fibered blender-horseshoes exist and are persistent in the sense that for a  given fibered blender-horseshoe 
$\Phi$ of $T\in \Diff^1_{\rm skew} (M\times \bR)$ with reference saddles $P$ and $Q$  there is a neighborhood $\cU$ of $T$ in $\Diff^1_{\rm skew}(M\times \bR)$ such that for all 
$\tilde T\in \cU$ the continuation of $\Phi$ for $\tilde T$ is a fibered blender-horseshoe with reference saddles $P_{\tilde T}$ and $Q_{\tilde T}$.
\end{proposition}

\begin{theorem}\label{t.fbhisgerm}
	Every fibered blender-horseshoe is a fibered blender which has the germ property.
\end{theorem}
 
\begin{proof}
The fact that a fibered blender-horseshoe is a fibered blender follows immediately from
FBH5): just embed a plaque $D_t\in \cD$ in a complete strip and consider any subset $D^i_t$ 
as in FBH5) 
 and note that by Remark~\ref{r.cuuexp}
$\lvert D^i_t\rvert \le \varkappa^{-1}  \lvert D_t\rvert$, $\varkappa>1$.
 
To get the germ property
take any $\xi\in\Xi$ and any $m\ge1$ and consider $B^\u_m(\xi)$. 
This set is contained in some rectangle $\fR(X)\subset \cW^{\u}_{\rm loc}(X,T)$. 
We consider a complete strip $S$ in $\cW^{\u}_{\rm loc}(X,T)$ and its associated  substrips $S_i$, $i=0,\ldots,N-1$, as in FBH5).
Then the strips $T(S_i)$ are complete by definition.
By Remark~\ref{r.BMarkov} we have
$B^\u_m(\xi)\cap \Phi\subset S_{i_0}$ for some $i_0$ and we denote by $S^{(0)}_{i_0}=S_{i_0}$. We let $S^{(1)}\eqdef T(S^{(0)}_{i_0})$ and
 note that, by construction, $S^{(1)}$ is complete. Thus by Remark~\ref{r.hsize} we have  $\lvert S^{(1)}\rvert_\h \ge \nu$. Moreover, $S^{(1)}$ intersects $\Phi$.
 
We can argue analogously with the complete strip $S^{(1)}$ and the 
set $B^\u_{m-1}(\tau(\xi))$ and consider the corresponding substrip $S^{(1)}_{i_1}$ that contains the set $B^\u_{m-1}(\tau(\xi))\cap\Phi$.  Arguing recursively, for $j=0,\dots,m-1$, 
we obtain complete strips $S^{(j)}$ and substrips $S^{(j)}_{i_j}$ of $S^{(j)}$ such that $S^{(j+1)} = T (S^{(j)}_{i_j})$.
Consider the final strip $S^{(m)}$ and let $\fR\eqdef S^{(0)}_{i_0}=T^{-m}(S^{(m)})$.
By construction, the set $\fR$ is a rectangle satisfying conditions (b) and (c) in the germ property
with $\delta=\nu$, the interval $J_m=\fR\cap I_\xi$, and 
$\hat D_Z= T^{-m}(D_Z)$, where the strip $S^{(m)}$ is foliated by plaques $D_Z$.
\end{proof} 

%\margem{mais simples, X01 aparece naturalmente}
	
%--------------------------------------------------------------------------------------------------------------------	
\subsection{Examples}%\label{ss.examples}
%--------------------------------------------------------------------------------------------------------------------	

\begin{figure}
a) \begin{overpic}[scale=.2]{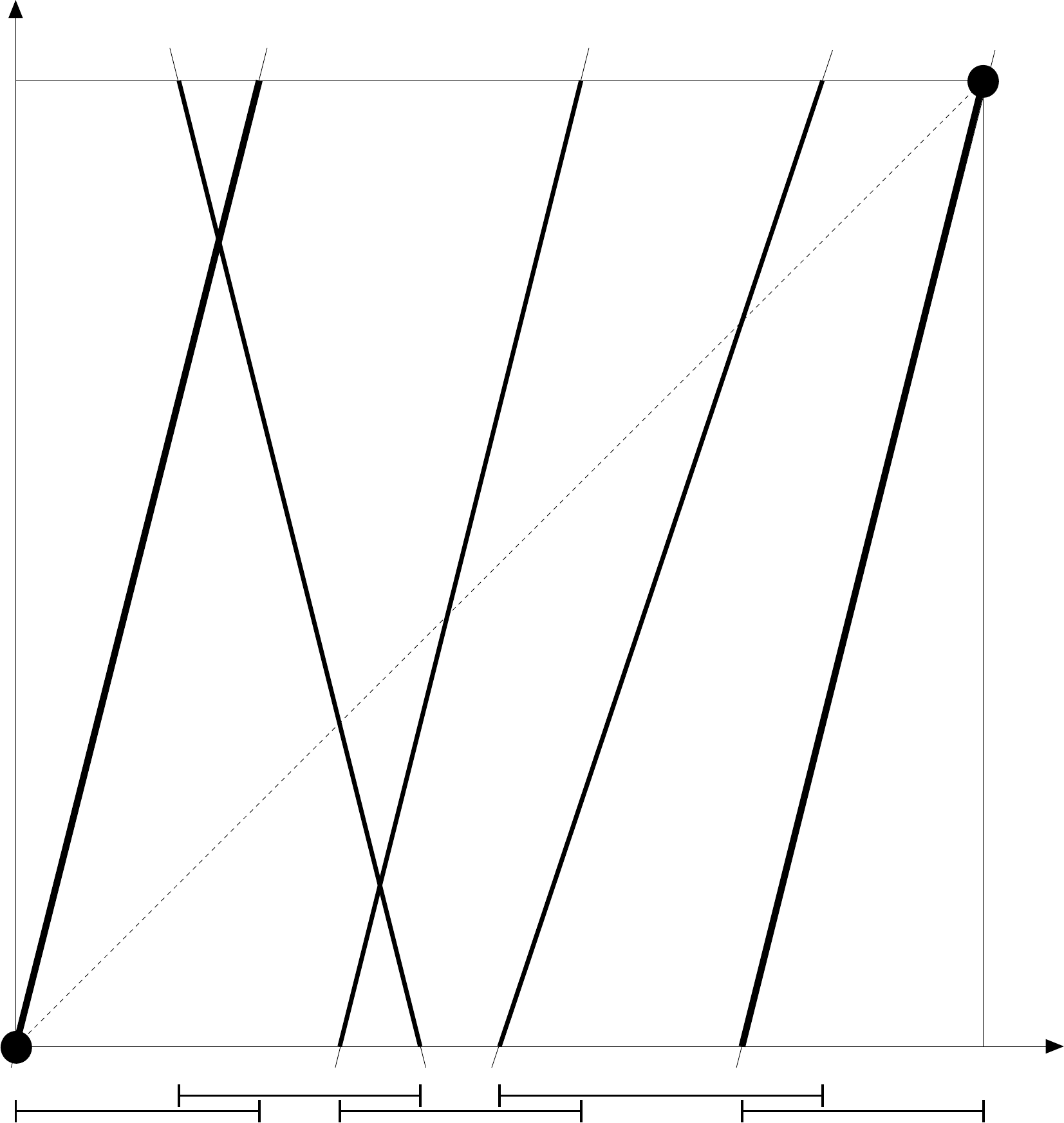} 
% \put(40,100){\small a)}
\end{overpic}
\hspace{0.1cm}
 b) \begin{overpic}[scale=.2]{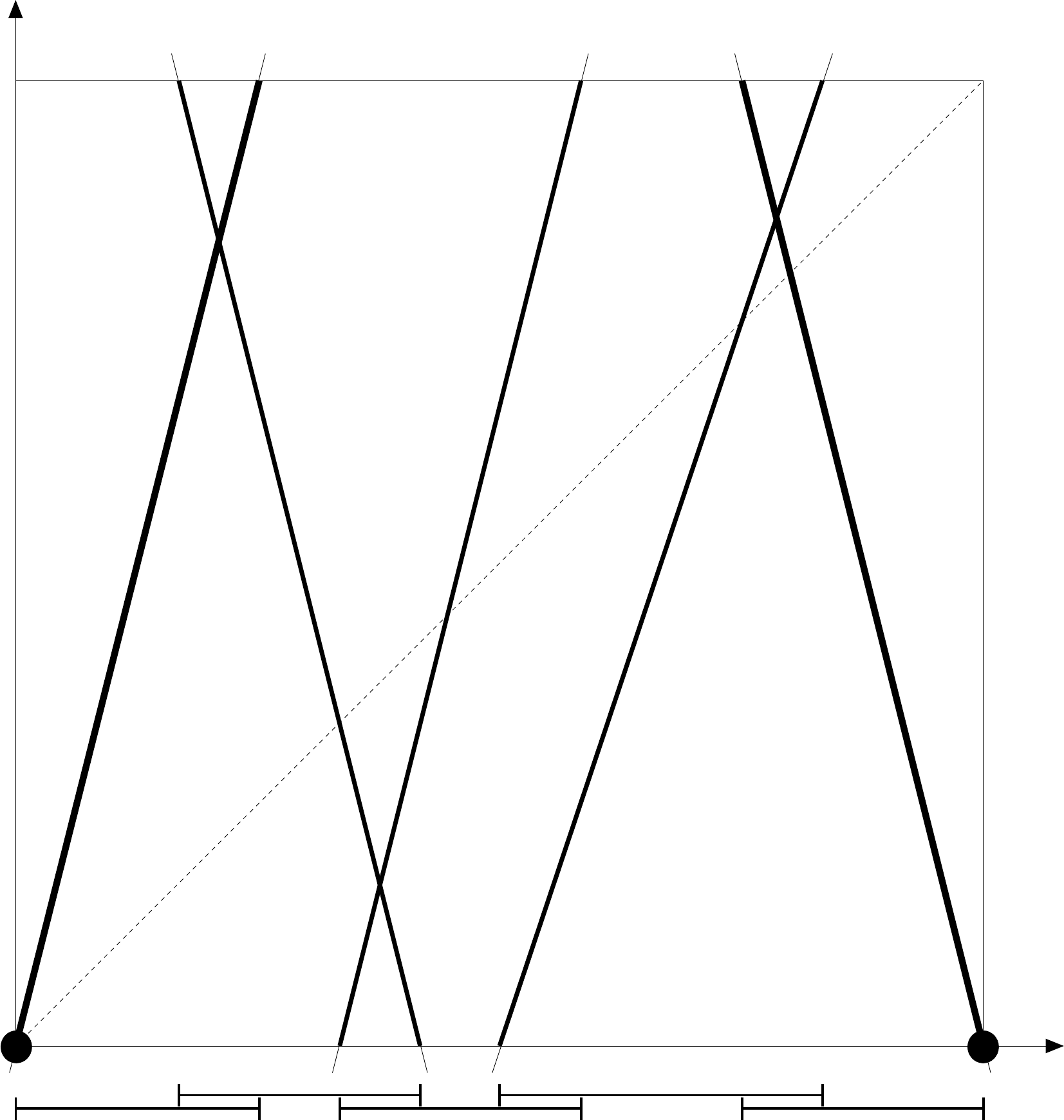} 
% \put(40,100){\small b)}
 \put(10,37){\small $T_0$}
 \put(70,37){\small $T_{N-1}$}
 \put(0,-5){\small $p_0$}
 \put(83,-5){\small $x_1$}
\end{overpic}
\hspace{0.1cm}
c) \begin{overpic}[scale=.2]{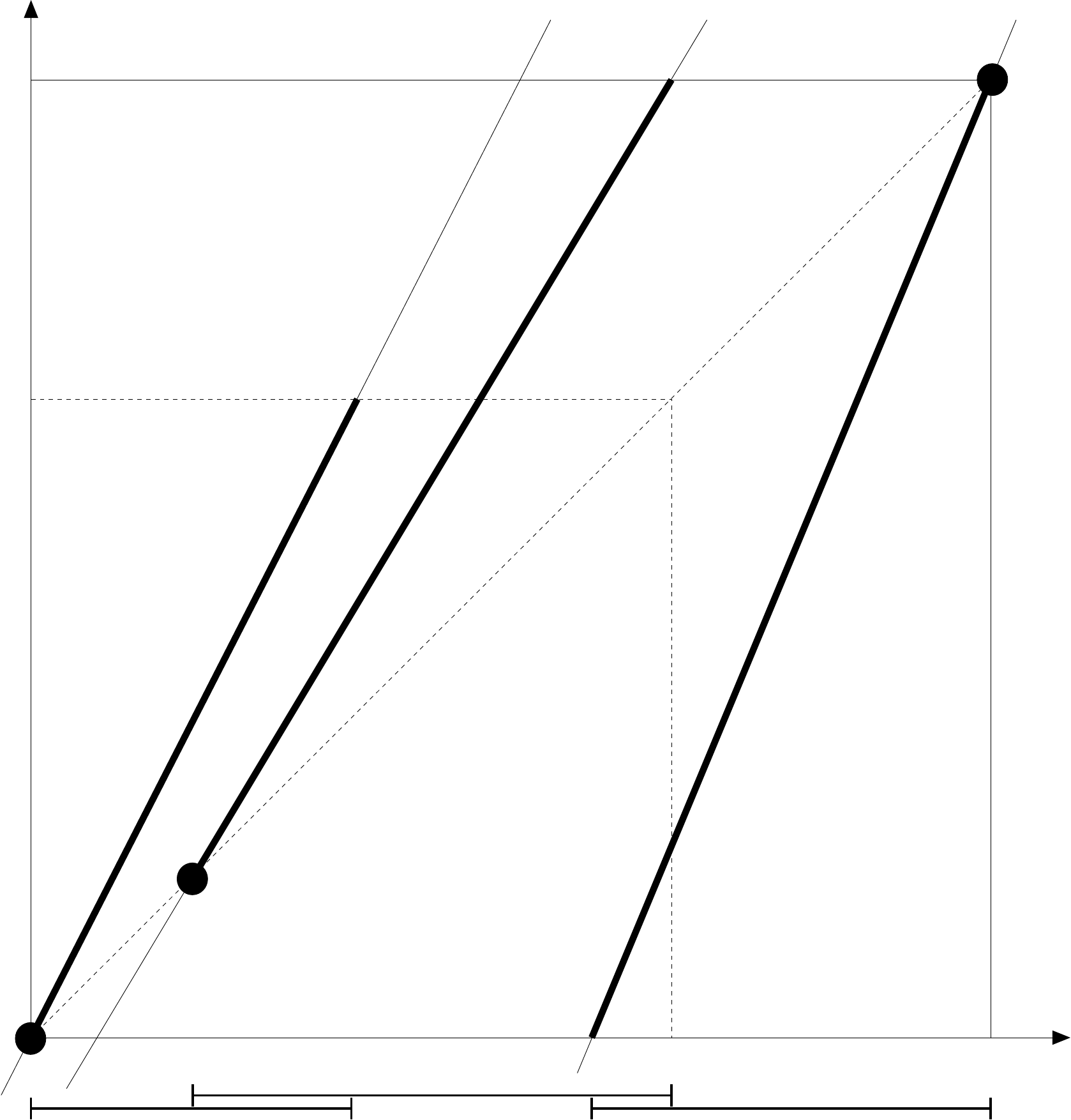} 
 %\put(40,100){\small c)}
 \put(10,37){\small $T_0$}
 \put(70,37){\small $T_1$}
 \put(30,37){\small $T_2$}
 \put(0,-5){\small $x_0$}
 \put(12,-5){\small $x_2$}
 \put(85,-5){\small $x_1$}
\put(60,-5){\small $y_2$}
\put(50,-5){\small $y_1$}
\put(28,-5){\small $y_0$}
\end{overpic}
\caption{Fiber maps for examples of step skew products}
\label{f.nopreservation}
\end{figure}
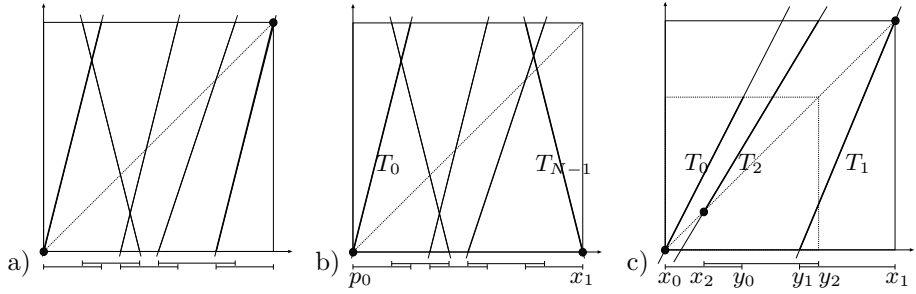

As mentioned already, our definition of a fibered blender-horseshoe does not aim for full generality. In what follows we present some examples and also some variations of these fibered blender-horseshoes satisfying its essential properties. 

For simplicity, in the following examples, we consider affine models and study maps of the form
$T\colon \Sigma\times \bR \to \Sigma\times \bR$ (that is, instead of~\eqref{e.skew-product-structure} already associating the symbolic code $\underline i=\chi^{-1}(\xi)$ of a point $\xi\in\Xi$) with step skew product structure
$$	
T(\underline i,x)
	=(\sigma(\underline i),T_{i_0}(x)),
$$
where $\Sigma=\{0,\ldots,N-1\}^{\bZ}$, $N\in\bN$ and $\sigma$ is the shift map. We also refrain from giving the details of the construction.

\begin{example}[Example in Section~\ref{sec:ex1} continued]\label{ex:blender:1}
Assuming $t\ne0$ and  $\lambda\in(1,2)$, $\Phi_t$
 satisfies all properties of a fibered blender-horseshoe. Clearly, it defines a step skew product with hyperbolic and partially hyperbolic dynamics and immediately provides the associated cone fields and plaque families. If $\lambda\in(1,2)$, then any complete strip contains two substrips whose projections to the $(\xi^\u,x)$-plane are the rectangles $S_0^t$ and $S_1^t$ defined in~\eqref{eq:rectanglesS01}. Note that the projection of these rectangles to the (vertical) $x$-axis overlap in the interval $[t/\lambda,t/(\lambda(\lambda-1))]$. In this case we have that both fixed points $P_0^t=(0^\bZ,0)$ and $P_1^t=(1^\bZ,t/(\lambda-1))$ have disjoint local strong unstable manifolds and in particular satisfy property FBH3) item (iv) whenever the cone field is chosen small enough. 
\end{example}

\begin{example}
Figure~\ref{fig.11} gives variations of  Example~\ref{ex:blender:1}  with three symbols  (Figure~\ref{fig.11} b) and c)) and where a fibered blender-horseshoe can only be found in some genuine subset (Figure~\ref{fig.11} a)).
\end{example}

Let us point out that our conditions do not explicitly require that the fiber maps preserve  orientation.
Although the fact that the hyperbolic set is confined between the local stable manifolds of $P_0$ and $P_1$, respectively, (compare property FBH3) item (iii)) implies that the fiber maps of the skew product, defining the fibered blender-horseshoe
and relating these reference saddles, do preserve the orientation, as for example depicted in the Figure~\ref{f.nopreservation}$\,$a).

\begin{example}
We can adapt  the definition of a fibered blender-horseshoe to a model as in Figure~\ref{f.nopreservation}$\,$b).
Let $p_0$ be the fixed point of the fiber map $T_0$ and 
$P_0=(0^{\mathbb{Z}}, p_0)$. Given $X\in\Phi$, we
define the point $X_0$ as before and let $X_1=(\underline i, x_1)$ be the point in 
$T^{-1} (\cW^\s_{\rm loc} (P_0, T)\cap \fR(T(X)))$ with $T_{N-1}(x_1)=p_0$, compare Figure~\ref{f.nopreservation}$\,$b), where $\fR(\cdot)$ denote the corresponding rectangles for FBH3). 
The fact that the hyperbolic invariant set in $\Sigma\times \bR$ is a fibered blender with the germ property in such a model follows now exactly as in the proof of
Theorem~\ref{t.fbhisgerm}.
\end{example}

\begin{figure}
a) \begin{overpic}[scale=.3]{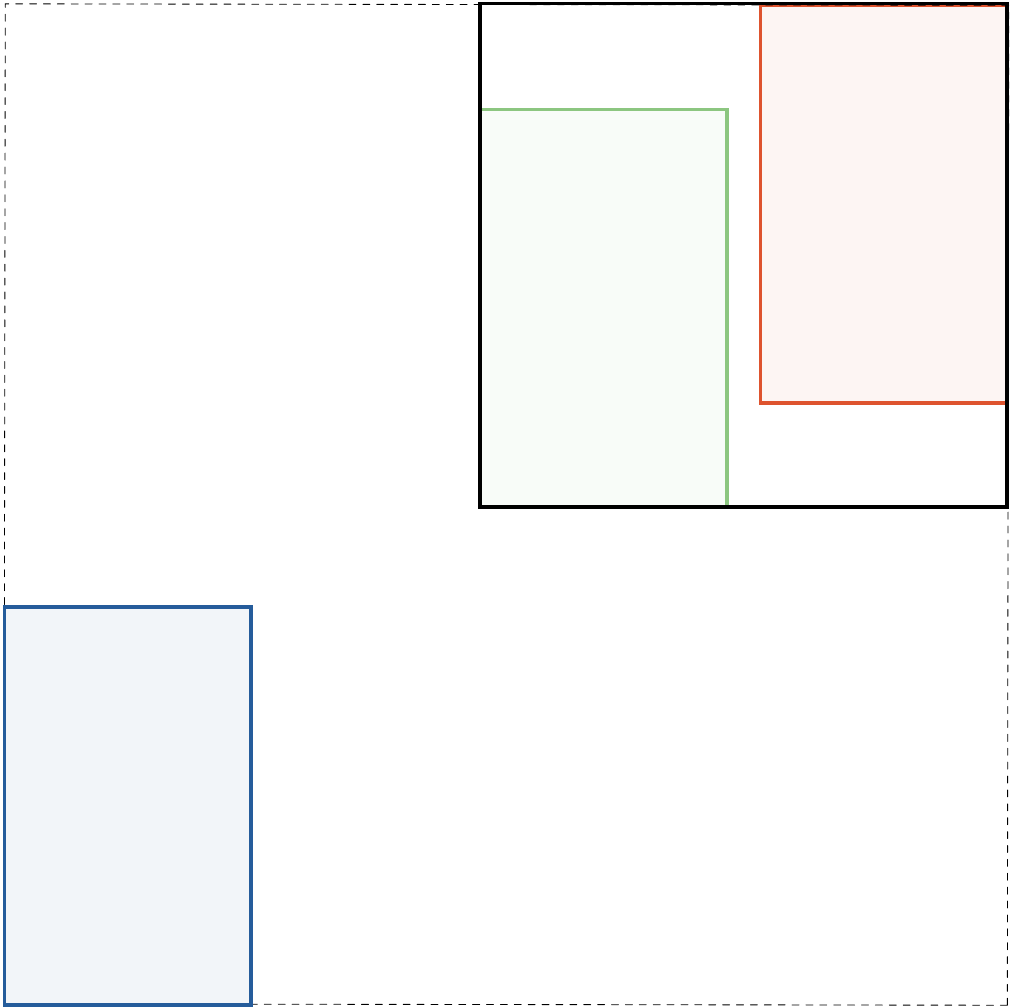} 
\end{overpic}
\hspace{0.5cm}
 b) \begin{overpic}[scale=.3]{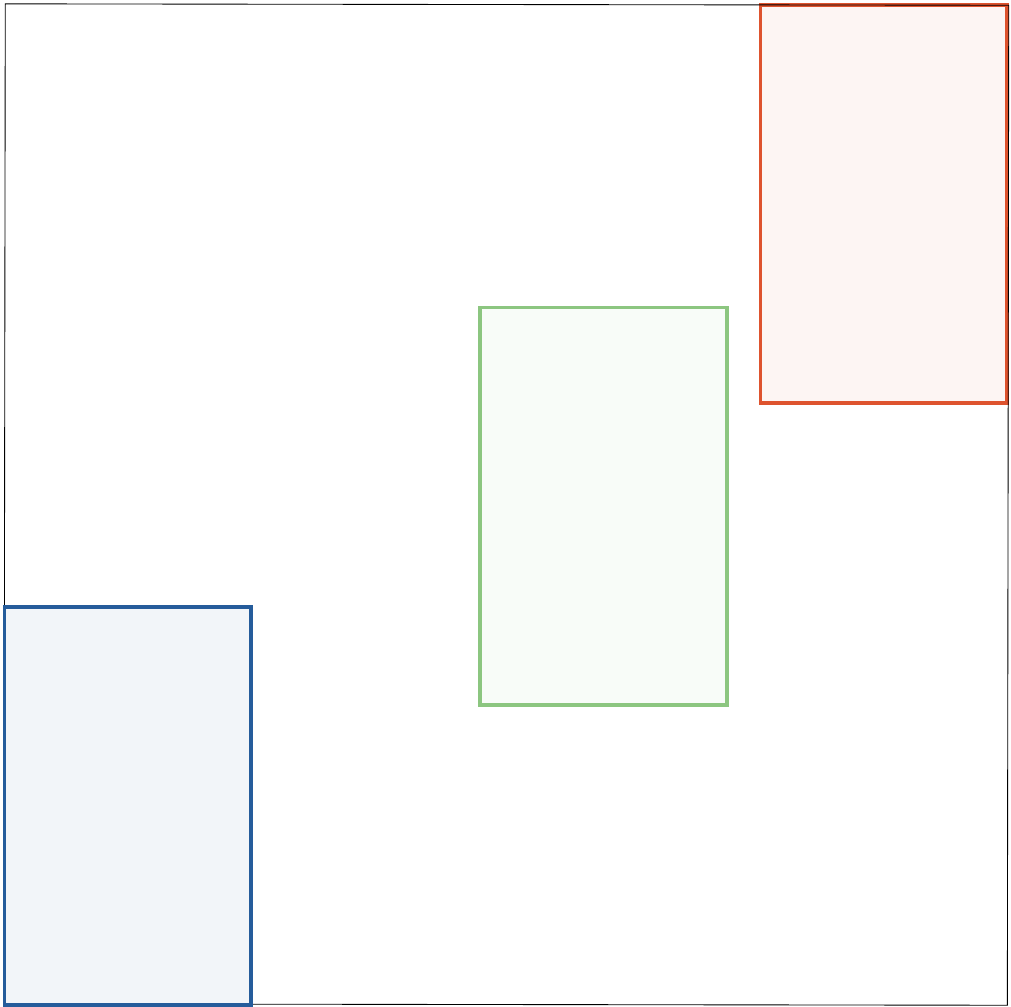} 
\end{overpic}
\hspace{0.5cm}
 c) \begin{overpic}[scale=.3]{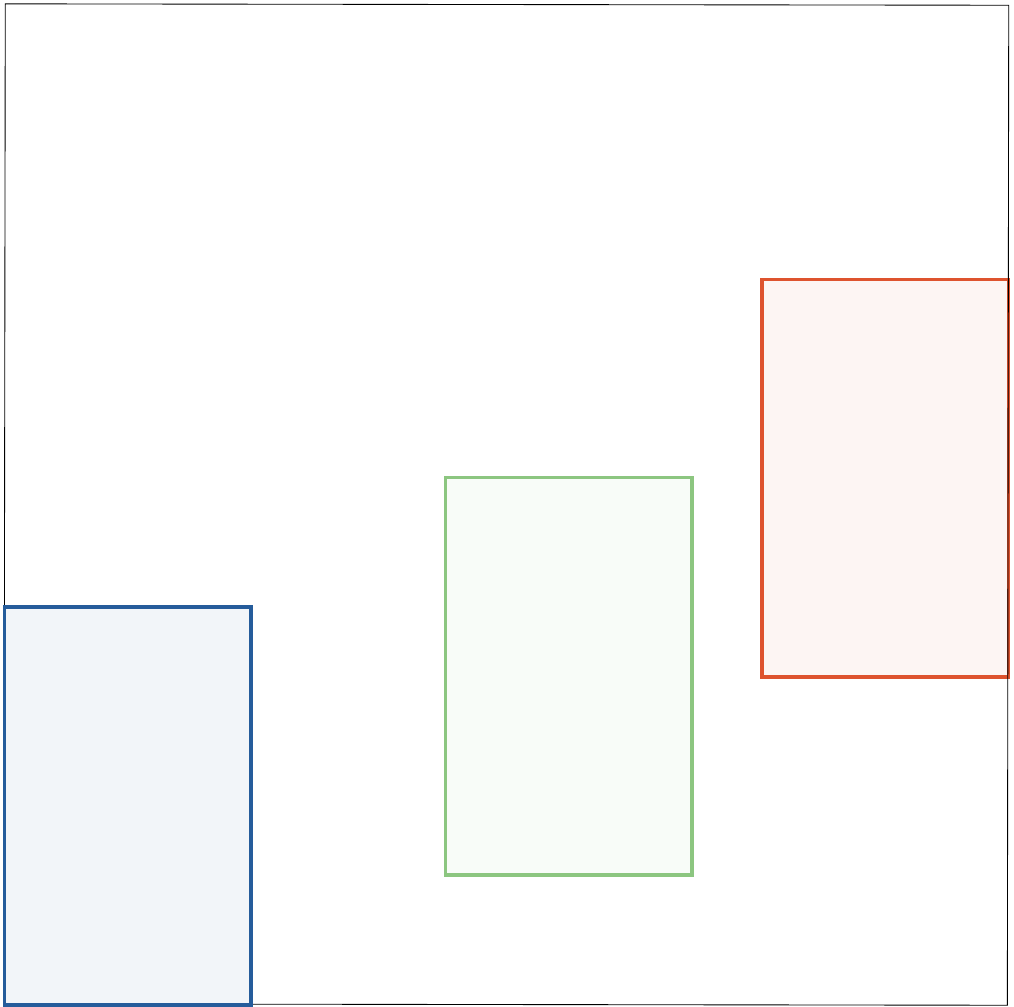} 
\end{overpic}
\caption{Fibered blender-horseshoes in a  subset (Figure a)) and  in the full invariant set (Figures b) and c))}
\label{fig.11}
\end{figure}

\begin{example}\label{exa:6}
We can also consider mixing subshifts $\Sigma_A\subset\Sigma$ as, for example, the
one associated to the transition matrix
$$
A=\left( \begin{matrix}
1& 1& 0\\
0& 1& 1\\
1& 1& 1
\end{matrix} \right)
$$
with $N=3$ and the fiber maps $T_0$, $T_1$, and $T_2$ 
 depicted in Figure~\ref{f.nopreservation}$\,$c). 
 Consider intervals $I_0=[x_0,y_0]$, $I_1=[y_1,x_1]$, and 
$I_2=[x_2,y_2]$, with $x_0<x_2<y_0<y_1<y_2<x_1$,
and expanding maps $T_0$, $T_1$, and $T_2$ defined on these intervals
such that $T_i(x_i)=x_i$, $T_0(I_0)=I_0\cup I_2=[x_0, y_2]$, 
$T_1(I_1)=I_0\cup I_1\cup I_2=[x_0, x_1]$,
and $T(I_2)=I_1\cup I_2=[x_2,x_1]$.
%(note that $y_1=T_2^{-1}(x_0)$, $y_2=T_1^{-1}(x_1)$, $y_0=T_0^{-1}(y_2)$). 

We only provide a picture proof indicating the %choice of  the set of plaques and the 
definition of the complete strips and substrips (compare FBH5)). 
The two fixed points are $X_0=(0^{\bZ},x_0)$ and $X_1=(1^{\bZ},x_1)$.
To prove that  we have a fibered blender-horseshoe, % with germ property, 
consider 
$\u$-boxes $B^\u_2(jk)=\{\underline i \in \Sigma_A \colon i_0=j, i_1=k\}$.
The set of plaques $\cD$ is depicted in Figure~\ref{f.plaquessubshift} according to the coordinate
$i_0$. Note that each plaque $D$ contains a $\hat D$ such that $T(\hat D)$ is a plaque.
\begin{figure}
\begin{overpic}[scale=.5,%grid,tics=5
]{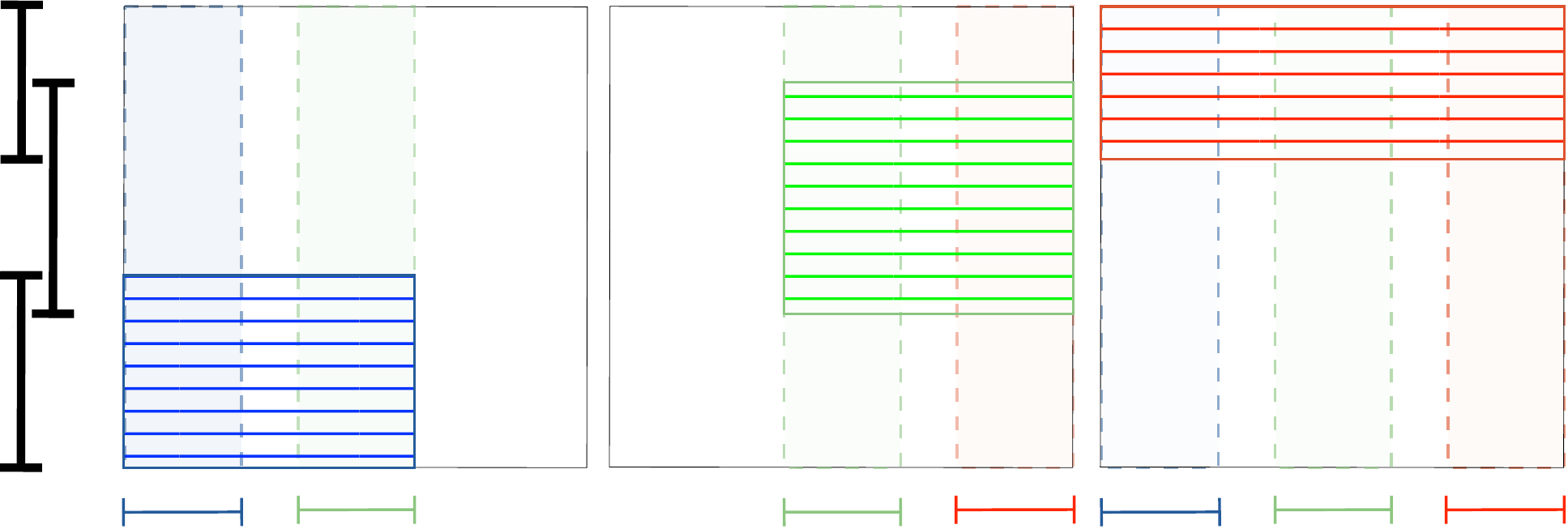} 
 \put(-3.5,10){\small $I_0$}
  \put(-3.5,20){\small $I_2$}
   \put(-3.5,28){\small $I_1$}
 \put(7,-3){\tiny $B^\u_2(00)$}
  \put(18,-3){\tiny $B^\u_2(02)$}
 \put(49,-3){\tiny $B^\u_2(22)$}
  \put(60,-3){\tiny $B^\u_2(21)$}
   \put(70,-3){\tiny $B^\u_2(10)$}
  \put(81,-3){\tiny $B^\u_2(12)$}
 \put(92,-3){\tiny $B^\u_2(11)$}
\end{overpic}\\[1cm]
\hspace{0.6cm}
\begin{overpic}[scale=.5%,grid,tics=5
]{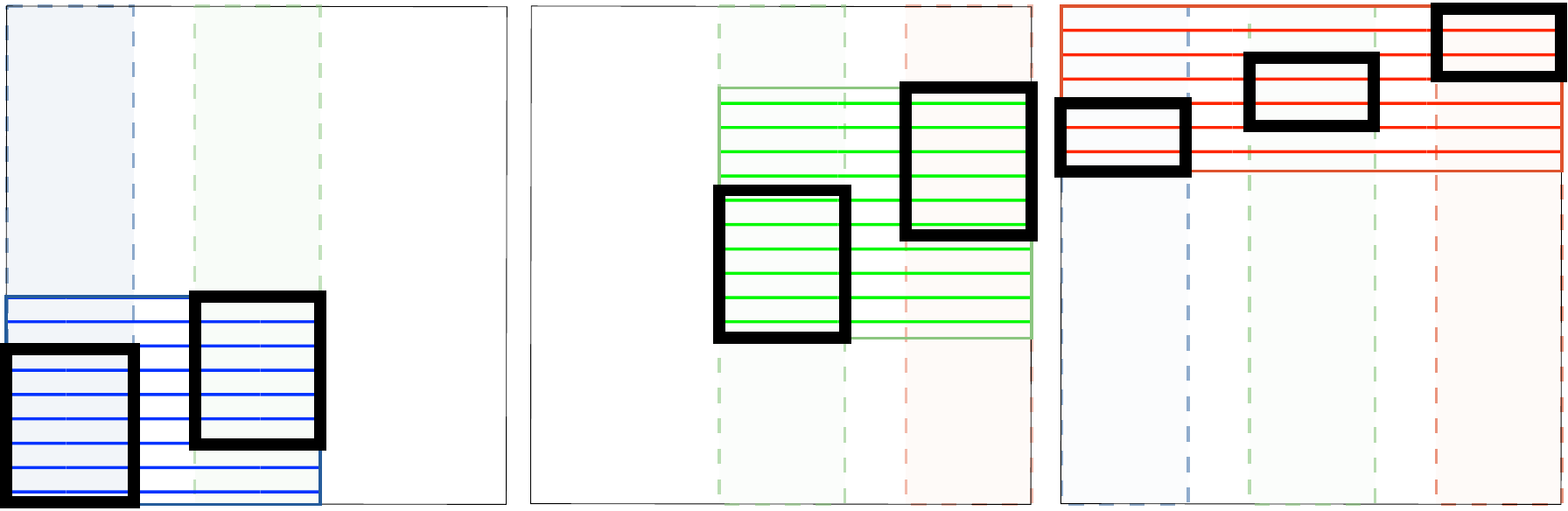} 
 \put(10,16){\small $S^0$}
  \put(4,-4){\small $S^0_0$}
  \put(15,0){\small $S^0_2$}
  \put(50,28){\small $S^2$}
  \put(49,6){\small $S^2_2$}
  \put(59,13){\small $S^2_1$}
  \put(75,33){\small $S^1$}
  \put(70,17){\small $S^1_0$}
  \put(82,20){\small $S^1_2$}
   \put(93,23){\small $S^1_1$}
\end{overpic}
\caption{Strips and $\u$-boxes in Example~\ref{exa:6}}
\label{f.plaquessubshift}
\end{figure}
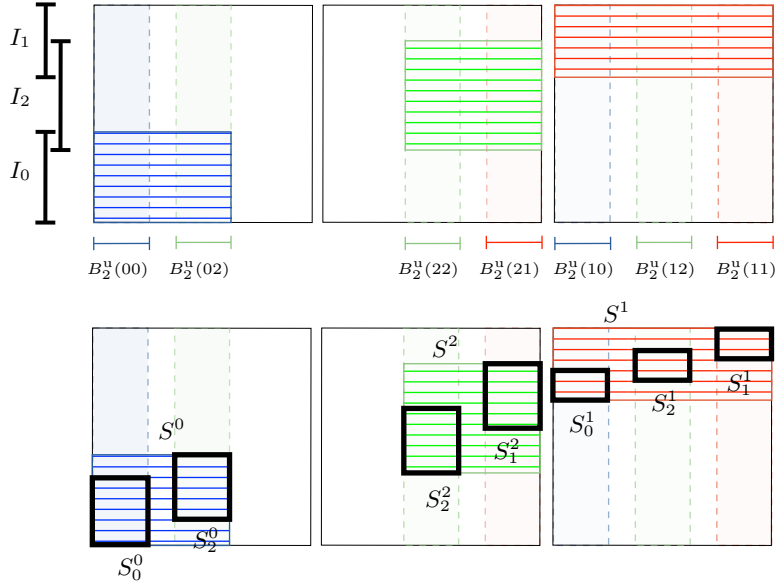
As for the fibered blender-horseshoes, we can consider strips:
there are three possibilities of complete strips $S^0$, $S^1$, and $S^2$ (also depending on the coordinate $i_0$) which
are also depicted in that figure. We also define substrips $S^i_j$ such that $T(S^i_j)\supset S^j$
when $S^i_j$ is nonempty.
Note that (since there are no transitions from $0$ to $1$ and from $2$ to $0$) the sets
$S^0_1$ and $S^2_0$ are both empty. The complete strips are the sets
$$
	S^0 = B^\u_1(0) \times [x_0,y_0],\quad
	 S^1 = B^\u_1(1) \times [y_1,x_1],\quad
	S^2 = B^\u_1(2) \times [x_2,y_2].
 $$
\end{example}

\begin{remark}[Example in Section~\ref{sec:ex2} continued]%\label{rem:anosovbound}
We observe that a hyperbolic and partially hyperbolic graph $\Phi$ over an Anosov diffeomorphism of a surface  $M=\Xi$ as in Section~\ref{sec:ex2} can be considered as a ``boundary case'' of  a fibered blender with germ property. Indeed, consider the \emph{one} plaque family $\cD$ provided by the family of all local strong unstable manifolds of points in $\Phi$. They clearly provide a continuous foliation of every local unstable manifold by plaques which are contained in local unstable manifolds (property FB1)). As~\eqref{propsMar} analogously holds for all local strong unstable manifolds $\cW^\uu_{\rm loc}$ we verify property FB3).
\end{remark}

%--------------------------------------------------------------------------------------------------------
\section{Some thermodynamics in shift spaces}\label{sec:shiftspace}
%--------------------------------------------------------------------------------------------------------

We recall some basic facts about thermodynamic objects in shift spaces. Standard
references are, for example,~\cite{Bow:08,Pes:97}, though we provide an almost
self-contained pedestrian approach.

%%%%%%%%%%%%%%%%%%%%%%%%%%%%%%%%%%%%%%%%%
\subsection{Future and past shifts}%\label{subsec:stablepart}
%%%%%%%%%%%%%%%%%%%%%%%%%%%%%%%%%%%%%%%%%

Consider the shift space $\Sigma=\{0,\ldots,N-1\}^\bZ$ and the usual left shift
$\sigma\colon\Sigma\to\Sigma$ as above (Section~\ref{sec:markov}).
We denote by $\Sigma^+=\{0,\ldots,N-1\}^\bN$ the space of right-sided infinite
sequences on $N$ symbols and consider the shift
$\sigma^+\colon\Sigma^+\to\Sigma^+$ defined by
$\sigma^+(i_0i_1\ldots)=(i_1i_2\ldots)$. Denote by
$\pi^+\colon\Sigma\to\Sigma^+$ the projection $\pi^+(\underline i)=\underline
i^+=(i_0i_1\ldots)$. Note that the projection of any $\sigma$-invariant set is
$\sigma^+$-invariant. 
Given $n\ge1$, we denote by $\Sigma_n^+\eqdef\{0,\ldots,N-1\}^n$ the set of
one-sided finite sequences of length $n$. Given $n,m\in\bZ$, $n\le m$, and a
finite sequence $(j_n\ldots j_m)$, we denote by $[j_n\ldots j_m]=\{\underline
i\in\Sigma\colon i_k=j_k \text{ for }k=n,\ldots,m\}$ the corresponding
\emph{cylinder set} of length $m-n+1$. Analogously, we define one-sided infinite
cylinders $[\ldots j_{-2}j_{-1}.]=\{\underline i\in\Sigma\colon i_k=j_k\text{
  for }k\le-1\}$ and $[j_0j_1\ldots]=\{\underline i\in\Sigma\colon i_k=j_k\text{
  for }k\ge0\}$. If $0\le n\le m$, let $[j_n\ldots j_m]^+=\pi^+([j_n\ldots
  j_m])$.

For symbolic systems and continuous potentials, the topological
pressure can be computed in a simplified way by choosing representatives from all
the cylinders of a given length. More precisely, if $\psi:\Sigma^+\to \bR$ is
continuous, then 
\begin{equation}\label{eq:pressureinshift}
	P_{\sigma^+}(\psi) 
	= \lim_{n\to\infty} \frac{1}{n} \log \sum_{(i_0\ldots i_{n-1})\in\Sigma_n^+}
		\exp\big(S_n\psi(\underline i^+)\big) ,
\end{equation}
where for each finite sequence $(i_0\ldots i_{n-1})$ the sequence $\underline
i^+\in[i_0\ldots i_{n-1}]^+$ is any  completion of the given prefix to
an infinite sequence.

%%%%%%%%%%%%%%%%%%%%%%%%%%%%%%%%%%%%%%%%%
\subsection{Joint Birkhoff averages}
%%%%%%%%%%%%%%%%%%%%%%%%%%%%%%%%%%%%%%%%%

Given a continuous function $\psi\colon\Sigma\to\bR^2$, $\psi(\underline
i)=(\psi_1(\underline i),\psi_2(\underline i))$, and a vector
$\va=(a_1,a_2)\in\bR^2$, we consider the \emph{level set of joint Birkhoff
  averages}
\[
	L(\psi,\va) \eqdef\Big\{\underline i\in\Sigma \colon 
	\lim_{n\to\infty}\frac{1}{ n}S_n\psi_1(\underline i)=
	a_1, \lim_{n\to\infty}\frac{1}{ n}S_n\psi_2(\underline i)=a_2\Big\}.
\]
Note that $L(\psi,\va)$ is $\sigma$-invariant.  Let $L_{\rm irr}(\psi)$ be the
set of all sequences at which at least one of the averages does not
converge (which is also $\sigma$-invariant).  Let 
\[
	D(\psi) \eqdef\{\va\in\bR^2\colon L(\psi,\va)\ne\emptyset\}.
\]
Note that $D(\psi)$ is compact.
Observe the so-called
\emph{multifractal decomposition} of $\Sigma$
\[
	\Sigma = \bigcup_{\va\in D(\psi)}L(\psi,\va)\cup L_{\rm irr}(\psi).
\]

%%%%%%%%%%%%%%%%%%%%%%%%%%%%%%%%%%%%%%%%%
\subsection{Multifractal analysis on the future shift}\label{subsec:mfa}
%%%%%%%%%%%%%%%%%%%%%%%%%%%%%%%%%%%%%%%%%

As later on we will focus on Markov \emph{unstable} rectangles, we now restrict
our considerations to $\sigma^+\colon\Sigma^+\to\Sigma^+$.
%(The arguments are equally true for $\sigma^-\colon\Sigma^-\to \Sigma^-$ after
%the appropriate changes.)
For that, given a potential $\psi\colon\Sigma\to\bR$, we now recall a classical way to define associate thermodynamical quantities ($\widehat\psi$ the unstable part of the potential $\psi$ and its pressure $P_{\sigma^+}(\widehat\psi)$ associated to $\psi$ and $P_\sigma(\psi)$).

First, given $\psi\colon\Sigma\to\bR$ and $n\ge0$ define
\[
	\var_n\psi
	\eqdef \sup\{\lvert\psi(\underline j)-\psi(\underline i)\rvert
		\colon \underline i\in[j_{-n}\ldots j_n]\}
\]
and let $\cF$ denote the family of all continuous functions
$\psi\colon\Sigma\to\bR$ for which $\var_n\psi\le b\alpha^n$ for all $n\ge0$ for
some positive constants $b$ and $\alpha\in(0,1)$.  It is easy to see that if $\psi$ is H\"older  continuous, then it is in $\cF$.

%Every bi-infinite sequence $\underline i\in\Sigma$ can be split as $\underline
%i=\underline i^-.\underline i^+$ with $\underline i^+\in\Sigma^+$ and
%$\underline i^-=(\ldots i_{-1}.)\in\{0,\ldots,N-1\}^{-\bN}$.
Given a H\"older continuous function $\psi\colon\Sigma\to\bR$, by~\cite[1.6
  Lemma]{Bow:08} there is $\phi\in\cF$ such that $\phi$ is \emph{homologous} to
$\psi$, that is, there is $u\in\cF$ such that
\[
	\phi = \psi - u+u\circ\sigma^{-1},
\]
and that for every $\underline j\in\Sigma$ we have
\[
	\phi(\underline i)=\phi(\underline j)\quad \text{ for every }
	\underline i\in [ j_{0}j_{1}\ldots].
\]
Hence, the function $\widehat\psi\colon\Sigma^+\to\bR$, $\widehat\psi(\underline i^+)\eqdef \phi(\underline i)$ 
for any $\underline i\in[i_0i_1\ldots]$, is well-defined and called  the \emph{unstable part} of $\psi$.%
\footnote{Note that there is an equivalent way to define $\widehat\psi$ by
\[
	\log\widehat\psi(\underline i^+)
	\eqdef \lim_{n\to\infty}\log\frac{\mu([i_0\ldots i_n])}{\mu([i_1\ldots i_n])},
\]
where $\mu$ is the Gibbs equilibrium state of $\psi$ (with respect to $\sigma$)
(see~\cite{Bow:08} and~\cite[Appendix II]{Pes:97}).}

\begin{remark}\label{rem:cohomobir}
By~\cite[Proposition A2.2]{Pes:97}  we have
\[
	P_{\sigma}(\psi)
	= P_{\sigma^+}(\widehat\psi).
\] 
Note that the homology relation immediately implies for every $\underline i\in\Sigma$
\[
	\lim_{n\to\infty}\frac1nS_n\psi(\underline i)
	=\lim_{n\to\infty}\frac1nS_n\widehat\psi(\underline i^+)	
%	=\lim_{k\to\infty}\frac1kS_k\psi^\u(\underline i),
\] 
(where we simultaneously use the symbol $S_n$ for the Birkhoff sum with respect
to $\sigma$ and $\sigma^+$, respectively) whenever the limit on either side
exists. More precisely, for every $n\ge1$ we have
\[
	S_n (\widehat\psi \circ \pi^+)
	= S_n\psi - u + u\circ\sigma^{-n} ,
\]
which immediately implies the following lemma.

\begin{lemma}\label{l.cohomology}
	For every $\varepsilon>0$ there is $n_0=n_0(\varepsilon,\psi)\ge1$ such that for every $n\ge n_0$ we have
\[
	\lvert S_n (\widehat\psi\circ \pi^+) -S_n\psi\rvert<\varepsilon.
\]	
\end{lemma}
\end{remark}

Given two H\"older continuous functions $\psi_1,\psi_2\colon\Sigma\to\bR$,
$\psi=(\psi_1,\psi_2)$ and their unstable parts
$\widehat\psi_1,\widehat\psi_2\colon\Sigma^+\to\bR$, $\widehat\psi=(\widehat\psi_1,\widehat\psi_2)$, for
$\va=(a_1,a_2)\in D(\psi)$ consider
\[
	L^+(\widehat\psi,\va) \eqdef \Big\{\underline i^+\in\Sigma^+ \colon 
	\lim_{n\to\infty}\frac{1}{ n}S_{n}\widehat\psi_1(\underline i^+)=a_1,
	\lim_{n\to\infty}\frac{1}{ n}S_{n}\widehat\psi_2(\underline i^+)=a_2
    	\Big\}
\]
and let
\[
	D^+(\widehat\psi)
	\eqdef \{\va\in\bR^2\mid L^+(\widehat\psi,\va)\neq\emptyset\}.
\]
By Remark~\ref{rem:cohomobir}, it is easy to see that the domains for nonempty
level sets coincide, that is, $D^+(\widehat\psi)=D(\psi)$.  Moreover,
$L^+(\widehat\psi,\va)=\pi^+(L(\psi,\va))$ for every $\va\in D(\psi)$.

Since we work mainly on the one-sided shift space $\Sigma^+$ from now on, we
skip the notation $\widehat{\,}$ and simply assume that $\psi_1,\psi_2\colon
\Sigma^+\to\bR$.
Given
$\psi=(\psi_1,\psi_2)$, $\va\in D^+(\psi)$, $\theta>0$, and $n\ge1$, let
\[\begin{split}
	M(\psi,\va,\theta,n) 
	\eqdef \card\Big\{ (j_0\ldots j_{n-1})&\in\Sigma_n^+\colon
	 \exists\,\underline i^+\in[j_0\ldots j_{n-1}]^+\\
&	\text{ with } \Big\lvert\frac1nS_n\psi_\ell(\underline i^+) -a_\ell\Big\rvert<\theta \textrm{ for } \ell=1,2\Big\} .
\end{split}\]
Further, we define
\[\begin{split}
	\underline \cS(\psi,\va,\theta) 
	&\eqdef \liminf_{n\to\infty} \frac1n \log M(\psi,\va,\theta,n),
\\
	\overline \cS(\psi,\va,\theta) 
	&\eqdef \limsup_{n\to\infty} \frac1n\log M(\psi,\va,\theta,n)  .
\end{split}\]

\begin{proposition}\label{p.joint-limits}
	For every $\va\in D^+(\psi)$ we have 
\[
	\lim_{\theta\to 0} \overline \cS(\psi,\va,\theta) 
	= \lim_{\theta\to 0}\underline \cS(\psi,\va,\theta)\eqdef \ch(\psi,\va).
\] 
%The common value of the two limits will be denoted by $\ch(\psi,\va)$.
\end{proposition}

\begin{remark}[Pedestrian approach]\label{rem:pedest}
The letter $\ch$ indicates that the number $\ch(\psi,\va)$ is closely related to the \emph{topological entropy} of $\sigma^+$ on the level set $L^+(\psi,\va)$ (see~\cite{Bow:73} for the definition of topological entropy on noncompact sets).  Proposition~\ref{p.joint-limits} as well as the further results in the
remainder of this section are contained in a similar form in \cite{Bed:89},
where they are proved by invoking thermodynamic formalism. However, this leads to some technical restrictions that
we want to avoid here, so we take a more direct and elementary and also more general approach and provide a self-contained proof.%
\footnote{In fact, Proposition~\ref{p.joint-limits} remains true even
  if $\psi$ is not H\"older continuous, but just continuous. This can be shown
  by approximating $\psi$ with a H\"older continuous function $\widetilde \psi$ and
  then using the statement for $\widetilde \psi$ together with the fact that if
  $\|\psi-\widetilde\psi\|<\delta$, then for any $\va\in\bR^2$, $\theta>0$ and $n\in\bN$
  \[
   M(\widetilde\psi,\va,\theta,n) \geq M(\psi,\va,\theta-\delta,n)
	\quad\text{ and }\quad
   M(\psi,\va,\theta,n) \geq M(\widetilde\psi,\va,\theta-\delta,n).	
  \]} 
  In particular, we do not require 
$\va\in\interior(D^+(\psi))$.
\end{remark}

We immediately note the following direct consequence of Proposition~\ref{p.joint-limits}.

\begin{corollary}\label{c.growth-rates}
  For any $\va\in D^+(\psi)$ and $\varepsilon>0$ there exists $\theta_0=\theta_0(\va,\varepsilon)>0$ and for every $\theta\in(0,\theta_0)$ there exists $n_0=n_0(\va,\varepsilon,\theta)\ge1$ such that for  $n\geq n_0$ we have
\[
n(\ch(\psi,\va)-\varepsilon) \leq \log M(\psi,\va,\theta,n) \leq
n(\ch(\psi,\va)+\varepsilon) .
\]
\end{corollary}

To prove Proposition~\ref{p.joint-limits}, we need the following elementary distortion result, which we state without proof.

\begin{lemma}[Bounded distortion]\label{l.shift-distortion}
	%If $\psi\colon\Sigma^+\to\bR$ is H\"older continuous, then t
There exists
        a constant $C>0$ such that for all $(j_0\ldots j_{n-1})\in\Sigma_n^+$ and $\underline
        i^+,\underline k^+\in[j_0\ldots j_{n-1}]^+$ we have
\[
	\lvert S_n\psi_\ell(\underline i^+) - S_n\psi_\ell(\underline k^+)\rvert \leq C ,\quad
	\ell=1,2.
\]
\end{lemma}

The above lemma implies the following one which we also state without proof.

\begin{lemma}\label{lem:inextremis}
	Given $\theta>0$ there exists $n_1=n_1(\theta)\ge1$ such that for every $n\ge n_1$ for every $\underline i^+\in\Sigma^+$ there exists $\va=(a_1,a_2)\in D^+(\psi)$ satisfying
\[
	\Big\lvert \frac1nS_n\psi_\ell(\underline i^+)-a_\ell\Big\rvert<\theta,\quad
	\ell=1,2.
\]	
\end{lemma}

\begin{proof}[Proof of Proposition~\ref{p.joint-limits}]

  Given $\bi=(i_0\ldots i_{n-1})\in\Sigma_n^+$ and $\underline j^+\in\Sigma^+$, we denote the one-sided
  infinite sequence $(i_0\ldots i_{n-1}j_0j_1\ldots)$ by $\bi\,\underline j^+$ and adopt the
  analogous notation for pairs of finite sequences. For any $\underline
  i^+\in[\bi]^+$, by Lemma~\ref{l.shift-distortion} 
  we  have
    \[
    \left\lvert S_{n+m}\psi(\bi\,\underline j^+) - S_n\psi(\underline i^+) -
    S_m\psi(\underline j^+)\right\rvert = \left|S_n\psi(\bi\,\underline
    j^+)-S_n\psi(\underline i^+)\right| \leq C .
  \]
  Given $m,r\ge1$, applying the same argument repeatedly for finite sequences $\bi^1$, $\ldots$,
  $\bi^m\in\Sigma_r^+$ and corresponding one-sided infinite sequences $\underline
  i^k\in[\bi^k]^+$, $k=1,\ldots,m$, we obtain
  \[
   \left| S_{mr}\psi(\bi^1\bi^2\ldots \bi^{m-1}\underline i^m)
   	 - \sum_{k=1}^m S_r\psi(\underline i^k)\right| 
	 \leq mC .
  \]
    This implies that 
  \begin{equation}\label{e.almost_super-multiplicativity}
   M(\psi,\va,\theta,mr) \geq M(\psi,\va,\theta-C/r,r)^m .
    \end{equation}
Moreover, for any $L\ge1$ and $n\in\{0,\ldots,L-1\}$ as
\[
	\frac1LS_L\psi
	=\frac1L\big(S_{L-n}\psi+S_n\psi\circ(\sigma^+)^n\big)
\]
it is immediate that
  \begin{equation} \label{e.cut-off}
    M(\psi,\va,\theta,L)\geq M(\psi,\va,\theta-n\|\psi\|/L,L-n) ,
  \end{equation}
where $\lVert\psi\rVert\eqdef\sup\,\lvert\psi\rvert$.
    
  Now, fix $\va\in D^+(\psi)$, $\theta>0$, and $\delta>0$ and  take $r\in\bN$ such
  that $C/r<\theta/4$ and 
  \[
     \frac1r\log M(\psi,\va,\theta/2,r) \geq \overline \cS(\psi,\va,\theta/2)-\delta .
  \]
  Consider $L\ge1$ large enough such that $r\lVert\psi\rVert/L\leq \theta/4$ and take $m=m(L)\ge1$ and $n=n(L)\in\{0,\ldots,r-1\}$ satisfying $L=mr+n$. Then with~\eqref{e.cut-off}  we have
\[%\begin{split}
	M(\psi,\va,\theta,L)
	\ge M\big(\psi,\va,\theta-\frac{n\lVert\psi\rVert}{L},mr\big)
	\ge M\big(\psi,\va,\theta-\frac\theta4,mr\big).
\]%\end{split}\]  
Thus, with~\eqref{e.almost_super-multiplicativity} we obtain
\[
	M\big(\psi,\va,\frac{3\theta}{4},mr\big)
	\ge M\big(\psi,\va,\frac{3\theta}{4}-\frac Cr,r\big)^m
	\ge M\big(\psi,\va,\frac{\theta}{2},r\big)^m.
\]
Hence, taking the limit $L\to\infty$ and hence $m(L)\to\infty$, we obtain
\begin{eqnarray*}
	\underline \cS(\psi,\va,\theta) 
     	& = & \liminf_{L\to\infty} \frac1L\log M(\psi,\va,\theta,L) \\ 
   	& \geq &  \liminf_{L\to\infty} \frac{\lfloor L/r\rfloor}{L} \log M(\psi,\va,\theta/2,r)\\ 
   	& = & \frac1r\log M(\psi,\va,\theta/2,r)\ \geq \ \overline \cS(\psi,\va,\theta/2)-\delta .
\end{eqnarray*}
%  Then for any $L\in\bN$ with $r\|\psi\|/L\leq \theta/4$ we obtain
% \[
%   M(\psi,\va,\theta,L) \stackrel{(\ref{e.cut-off})}{\geq}
%   M(\psi,\va,3\theta/4,\lfloor L/r\rfloor)
%   \stackrel{(\ref{e.almost_super-multiplicativity})}{\geq} M(\psi,\va,\theta/2,r)^{\lfloor L/r\rfloor}
% \]
% and hence
% \begin{eqnarray*}
%   \underline \cS(\psi,\va,\theta) 
%   & = & \liminf_{L\to\infty} \frac1L\log M(\psi,\va,\theta,L) \\ 
%   & \geq &  \liminf_{L\to\infty} \frac{\lfloor L/r\rfloor}{L} \log M(\psi,\va,\theta/2,r)\\ 
%   & = & \frac1r\log M(\psi,a,\theta/a,r)\ \geq \ \overline \cS(\psi,\va,\theta/2)-\delta .
% \end{eqnarray*}
As $\delta>0$ was arbitrary, this shows $\underline
\cS(\psi,\va,\theta)\geq\overline \cS(\psi,\va,\theta/2)$. The latter, in turn,
immediately implies $\lim_{\theta\to 0} \underline
\cS(\psi,\va,\theta)=\lim_{\theta\to 0}\overline \cS(\psi,\va,\theta)$ and thus
completes the proof.
\end{proof}

\begin{proposition}\label{prolem:lemma}
Given a H\"older continuous function $\psi\colon\Sigma^+\to (-\infty,0)^2$, the number
%The number
\[
	t \eqdef \sup_{\va=(a_1,a_2)\in D^+(\psi)}\frac{\ch(\psi,\va)+a_2}{-a_1}
\]	
satisfies $P_{\sigma^+}(t\psi_1+\psi_2)=0$ and is uniquely determined by this equation.
\end{proposition}

\begin{proof}
First, observe that compactness of $D^+(\psi)$ implies that $t$ is finite.

Fix $\varepsilon>0$. First, suppose $t'<t$. Choose some $\va \in D^+(\psi)$ such that 
\[
	\frac{\ch(\psi,\va)+a_2}{-a_1} > t'
\]
and note that since $a_1<0$ this implies
\begin{equation}\label{e.entropybound1}
  \ch(\psi,\va) > -t'a_1-a_2 .
\end{equation}
Choose $\theta_0=\theta_0(\va,\varepsilon)>0$ according to Corollary~\ref{c.growth-rates}. Taking $\theta\in(0,\theta_0)$, choose also $n_0=n_0(\va,\varepsilon,\theta)\ge1$ according to this corollary such that for every $n\ge n_0$ we have
%Further, let $\delta>0$ and $\theta\in(0,\delta)$ be such that there exists some $n_0\ge1$ which satisfies
\[%begin{equation}\label{e.growth-rates}
	n(\ch(\psi,\va)-\varepsilon) 
	\leq \log M(\psi,\va,\theta,n) \leq n(\ch(\psi,\va)+\varepsilon) .
\]%end{equation}
%for all $n\geq n_0$. 
Assume first that $t'\geq 0$. With~\eqref{eq:pressureinshift}, but only taking the sum over sequences which contribute to $M(\psi,\va,\theta,n)$, we have
\begin{eqnarray*}
  P_{\sigma^+}(t'\psi_1+\psi_2) 
  & \geq & \lim_{n\to\infty} \frac1n\log \Big( M(\psi,\va,\theta,n)
  	\exp\big(n(t'(a_1-\theta)+a_2-\theta)\big) \Big)\\ 	
  &\ge& \ch(\psi,\va)-\varepsilon+t'(a_1-\theta)+a_2-\theta\\
  \text{by~\eqref{e.entropybound1} }&>&-t'a_1-a_2-\varepsilon+t'(a_1-\theta)+a_2-\theta\\	
 & =& -\varepsilon-\theta(t'+1).%\\
%  & \geq & \ch(\psi,\va)-\varepsilon+t'(a_1-\theta)+a_2-\theta\\ 
%  & = &\ch(\psi,\va)-\varepsilon+ (t'+\varepsilon)a_1+a_2- \varepsilon a_1 - t'\theta-\theta \\
 % \text{by~\eqref{e.entropybound1} }&> & -\varepsilon (1+a_1) - (1+t')\theta ,
\end{eqnarray*}
As $\theta>0$ and $\varepsilon>0$ were arbitrarily small, this shows  $P_{\sigma^+}(t'\psi_1+\psi_2)\geq 0$. 
Assuming now $t'<0$, analogously we come to the same conclusion. 

Second, suppose now that $t'>t$. Let $A=\frac{1}{3}\min\{|a_1|\colon \va=(a_1,a_2)\in D^+(\psi)\}$ and
fix some $\delta$ sufficiently small.
%\begin{equation}\label{e.delta_choice}
%	\delta \in (0,\min\{A,\varepsilon A/(2+t)\}) .
%\end{equation}
Cover the set $D^+(\psi)$ by  $2\theta_j$-squares $(Q_j)_{j=1}^\ell$
\[
		Q_j=(a^j_1-\theta_j,a^j_1+\theta_j)\times(a^j_2-\theta_j,a^j_2+\theta_j) 
\]
for appropriately chosen $\va^j=(a^j_1,a^j_2)\in D^+(\psi)$ and $\theta_j\in(0,\delta)$. 
Thereby, we assume that the $\theta_j>0$ are such that $\theta_j<\theta_0(\va^j,\varepsilon)$ where the latter number is as in Corollary~\ref{c.growth-rates}.
Recalling that $D^+(\psi)$ is  compact, such a finite cover by open squares exists. 
For every $j$ let then $n_j\ge n_0(\va^j,\varepsilon,\theta_j)$, where the latter is as in Corollary~\ref{c.growth-rates} and let $n_0\ge\max_{j=1,\ldots,\ell}n_j$.
Hence, for every $j=1,\ldots,\ell$ and $n\ge n_0$ we have  
\[
	n(\ch(\psi,\va^j)-\varepsilon) 
	\leq \log M(\psi,\va^j,\theta_j,n) 
	\leq n(\ch(\psi,\va^j)+\varepsilon).
\]
%for all $j=1,\ldots,\ell$ and $n\geq n_0$. 
Note that by the definition of $t$, for all $j=1,\ldots,\ell$ we have 
\begin{equation}\label{e.entropybound2}
 \ch(\psi,\va^j) \leq  -ta_1^j-a_2^j.
\end{equation}
Moreover,  by Lemma~\ref{lem:inextremis} we can assume that
\[
	\Big(\frac1n S_n\psi_1(\underline i^+),\frac1n S_n\psi_2(\underline i^+)\Big)
	\in\bigcup_{j=1}^\ell Q_j
\] 
for all $\underline i^+\in\Sigma^+$ if $n\geq n_0$. Then, assuming first that $t'\geq 0$, we have
\[\begin{split}
	P_{\sigma^+}&(t'\psi_1+\psi_2) 
	\leq  \lim_{n\to\infty}\frac1n\log\sum_{j=1}^\ell M(\psi,\va^j,\theta_j,n)
   		\exp(n(t'(a^j_1+\theta_j)+a^j_2+\theta_j)) \\ 
   	& \leq \lim_{n\to\infty}\frac1n\log \sum_{j=1}^\ell
 		\exp\left(n\left(\ch(\psi,\va^j)+\varepsilon+ta^j_1+a^j_2
 						+t'\theta_j+\theta_j\right)\right)\\ 
 \text{by~\eqref{e.entropybound2}}
 	&\le  \lim_{n\to\infty} 
		\frac1n\log\left(\ell\exp(n(\varepsilon+(1+t')\delta))\right).%\\
%\text{by~\eqref{e.delta_choice}} 
%	&\le -\varepsilon A  .
\end{split}\]
As $\delta$ and $\varepsilon$ were arbitrarily small, this implies $P_{\sigma^+}(t'\psi_1+\psi_2) \le 0$.
Again, the case $t'<0$ can be treated in the same way and leads to the same
result. 

Altogether, we thus obtain that $P_{\sigma^+}(t'\psi_1+\psi_2)$ is nonnegative if $t'<t$ and nonpositive if $t'>t$.  By the continuity of
the pressure function, this implies $P_{\sigma^+}(t\psi_1+\psi_2)=0$ as required. Further, since $\psi_1,\psi_2$ are negative functions the number $t$
is unique with that property. This proves  the proposition.
\end{proof}

%--------------------------------------------------------------------------------------------------------
\section{Dimension of stable and unstable slices}
\label{newsec:proofss}
%--------------------------------------------------------------------------------------------------------
 
The aim of this section is to prove the following two partial results towards the proof of Theorems~\ref{t.anosov},~\ref{the:one-dimensional}, and~\ref{the:1}.

\begin{proposition}\label{pro:localdimension-stable}
	Let $T$ be a three-dimensional skew product diffeomorphism satisfying the Standing hypotheses. Then for every $X\in\Phi$ we have	
\[
	\dim_{\rm H}(\Phi\cap\cW^\s_{\rm loc}(X,T))
	= \dim_{\rm B}(\Phi\cap\cW^\s_{\rm loc}(X,T))	
	= d^\s,
\]	
where $d^\s$ is as in~\eqref{eq:localdimension}.
\end{proposition} 

\begin{proof}%[Proof of Proposition~\ref{pro:localdimension-stable}]
Let $X=(\xi,\Phi(\xi))$, $\xi\in\Xi$, be an arbitrary point in the graph. By Theorem~\ref{teo:seminal}, we have 
\[
	\dim_{\rm H}(\Xi\cap\cW^\s_{\rm loc}(\xi,\tau))
	= \dim_{\rm B}(\Xi\cap\cW^\s_{\rm loc}(\xi,\tau))	
	= d^\s.
\]	
%where $d^\s$ is as in~\eqref{eq:localdimension}.
By Proposition~\ref{prop:critical},  $\Phi$ is Lipschitz on local stable manifolds. Notice also that the projection $(\xi,\Phi(\xi))\mapsto\xi$ is Lipschitz.  Hence, as the dimensions are invariant under bi-Lipschitz maps (see Section~\ref{sec:defdim}), the claim follows.
\end{proof}
 
 \begin{proposition}\label{pro:localdimension-unstable}
	Let $T$ be a three-dimensional skew product diffeomorphism satisfying
the Standing hypotheses. 
Assume that $\Phi\colon\Xi\to\bR$ is not Lipschitz continuous,	where $\Xi\subset M$ is a basic set (with respect to $\tau$).  
		Then for every $X\in\Phi$ we have	
\[
	\overline\dim_{\rm B}(\Phi\cap\cW^\u_{\rm loc}(X,T))	
	\le d,
\]	
where $d$ is the unique real number satisfying
\[
	P_{\tau|_\Xi}(\varphi^\cu+(d-1)\varphi^\u)=0.
\]
If, in addition, $T$ satisfies the hypotheses of either Theorems~\ref{t.anosov},~\ref{the:one-dimensional}, or~\ref{the:1} then for every $X\in\Phi$ we have	
\[
	\dim_{\rm B}(\Phi\cap\cW^\u_{\rm loc}(X,T))	
	= d.
\]	
\end{proposition}

To prove this proposition, we  study the Birkhoff averages of the potentials $\varphi^\u$ and $\varphi^\cu$ % in our original fibered system. 
which control the
size of Markov unstable rectangles. Covering the %graph, or rather its
restriction of the invariant graph to a local unstable manifold, with a suitable collection of such
rectangles,  will allow to compute the box dimension. Thereby, it turns out
that we can prove the proposition (under any of the three additional hypotheses of Theorems~\ref{t.anosov},~\ref{the:one-dimensional}, or~\ref{the:1}), since the differences are
marginal and can easily be discussed alongside. We split the proof in a natural
way into the upper and the lower estimate on the box dimension. 

\medskip
\noindent\textbf{Caveat:} Throughout the
remaining section, we assume that $\Phi$ is not Lipschitz.

%------------------------------------------------------------------------------------------
\subsection{Symbolic coding of local unstable manifolds.}
%------------------------------------------------------------------------------------------

We first recall the basic facts concerning the symbolic coding.
Recall that a Markov partition of $\Xi$ provides us with a H\"older semi-conjugacy (conjugacy if $\Xi$ is a Cantor set) $\chi\colon\Sigma\to\Xi$ (recall all ingredients in Section~\ref{sec:markov}). 
%from the shift map $\sigma\colon\Sigma\to\Sigma$ to $\tau\colon\Xi\to\Xi$.  
Fix $X=(\xi,\Phi(\xi))$. Choose $\underline i=(\ldots
i_{-1}.i_0\ldots)\in\Sigma$ such that $\chi(\underline i)=\xi$.  
%Note that the natural projection $\pi_1\colon\Phi\to\Xi$ projects the Markov
%unstable rectangle $R^\u(\xi)=\cW^\u_{\rm loc}(X)\cap R(X)$ (with respect to
%$T|_\Phi$) onto the Markov unstable rectangle $\underline
%R^\u(\xi)=\cW^\u_{\rm loc}(\xi)\cap\underline R(\xi)$.
Note that the symbolic coding of every $\eta\in \underline R(\xi)\cap\cW^\u_{\rm loc}(\xi,\tau)$ starts with the same symbol $i_0$ and that the local unstable
manifold of $\xi$ contains the $\chi$-image of the cylinder $[\ldots
  i_{-1}.i_0]$,
\[
	\underline R^\u(\xi)
	= \underline R(\xi)\cap\cW^\u_{\rm loc}(\xi,\tau)
	= \chi\big([\ldots i_{-1}.i_0]\big).
\]
Correspondingly, for every $\eta=\chi(\ldots i_{-1}.i_0j_1\ldots)$, $Y=(\eta,\Phi(\eta))$, $n\ge0$ we have
%the $\chi$-image of the $n$th level cylinder $[\ldots  i_{-2}i_{-1}.i_0\ldots i_{n-1}]\subset [\ldots i_{-1}.i_0]$ equals the $n$th Markov unstable rectangle
\[
	\underline R^\u_n(\eta)
	%= \underline R_n(\eta) \cap \cW^\u_{\rm loc}(\eta,\tau)
	= \chi\big([\ldots i_{-2}i_{-1}.i_0j_1\ldots j_{n-1}]\big).
\]
 So there is a natural coding between $\Sigma^+$ and the Markov unstable rectangle $\underline R^\u(\xi)$.
%\[
%	\chi^+_\xi
%	\colon\Sigma^+\to \underline R^\u(\xi) .%\cW^\u_{\rm loc}(\xi,\tau)\cap \underline R(\xi)
%\]
The same carries over to the unstable rectangle $R^\u(X)= R(X)\cap\cW^\u_{\rm loc}(X,T)$.

Given the H\"older continuous functions $\varphi^\u$ and $\varphi^\cu$ 
defined in~\eqref{eq:defbasicpotentialvarphius} and~\eqref{e.potentials}, 
we consider the lifted functions $\varphi^\u\circ\chi, \varphi^{\cu}\circ\chi\colon\Sigma\to\bR$.
Note that both functions are strictly negative.  Also note that they are  H\"older continuous  (though possibly with different H\"older exponents than $\varphi^\u,\varphi^\cu$). 
Denote by $\psi_1,\psi_2\colon\Sigma^+\to\bR$ their unstable parts 
\[
	\psi_1
	\eqdef \widehat{\varphi^\u\circ\chi},\quad
	\psi_2
	\eqdef \widehat{\varphi^\cu\circ\chi}
\]
and recall that they are also H\"older continuous. In the following we will use the methods developed in Section~\ref{subsec:mfa}.
In particular, by Remark~\ref{rem:cohomobir} for every $t\in\bR$
\begin{equation}\label{eq:equalityallpressure}
 	P_{\sigma^+}(\psi_2+t\psi_1)
	= P_{\sigma}(\varphi^\cu\circ\chi+t\varphi^\u\circ\chi)
	= P_{\tau|_\Xi}(\varphi^\cu+t\varphi^\u)
\end{equation}
(for the latter equality see also~\cite[Chapter 4]{Bow:08}).
Moreover, by Proposition~\ref{prolem:lemma} the number
\begin{equation}\label{eq:deftfinal}
	t = \sup_{\va=(a_1,a_2)\in D^+(\psi)}\frac{\ch(\psi,\va)+a_2}{-a_1}
\end{equation}
satisfies $P_{\sigma^+}(t\psi_1+\psi_2)=0$ and is uniquely determined by this equation.

The coding naturally induces the level sets of
\emph{one-sided} Birkhoff averages studied in Section~\ref{subsec:mfa}. 
%In order to relate these to the respective level sets for the original dynamics in $\Xi$,
Given $\va=(a_1,a_2)$ let
\[
	\cL(\varphi^\u,\varphi^\cu,\va) 
	\eqdef \Big\{\xi\in\Xi \colon 
		\lim_{ n\to\infty}\frac{1}{n}S_n\varphi^\u(\xi)=a_1, 
		\lim_{ n\to\infty}\frac{1}{n}S_n\varphi^{\cu}(\xi)=a_2\Big\}
\]
and note that with $\psi=(\psi_1,\psi_2)$ we have $L^+(\psi,\va)=(\pi^+\circ\chi^{-1})(\cL(\varphi^\u,\varphi^\cu,\va))$.

Finally, as we assume that $\Phi$ is not Lipschitz,  Propositions~\ref{pro:dis0} and \ref{pro:dis} provide the following estimates of the sizes of Markov unstable rectangles
\begin{equation}\label{e.width}
	\frac1c\le
	\frac{\lvert R^\u_n(Y)\rvert_\w }{\exp(S_n\varphi^\u(\zeta))}
	\le c,\quad
%\end{equation}
%and
%\begin{equation} \label{e.height}
		\frac1c
	\le \frac{\lvert R^\u_n(Y)\rvert_{\rm h}}{\exp(S_n\varphi^\cu(\zeta'))}
	\le c,
\end{equation}
where $c>1$ is a fixed constant, 
 $Y=(\eta,\Phi(\eta))$, and $\zeta,\zeta'\in \underline R^\u_n(\eta)$ are arbitrary points.
Finally also note that, due to Lemma~\ref{l.cohomology}, the 
Birkhoff sums can be equally controlled in terms of the sums of the unstable parts of the symbolic potentials. Hence, the lemma below follows directly from~\eqref{e.width}.

\begin{lemma}\label{lem:Markovunstabpot}
  There exists a constant $C>1$ such that for any $X=(\xi,\Phi(\xi))$, $n\ge1$, and $Y=(\eta,\Phi(\eta))\in R^\u_n(X)$, we have
\[%  \begin{eqnarray}\label{e.width_symbolic}
   	\frac{1}{C} 
	 \leq  \frac{\lvert R^\u_n(Y)\rvert_\w}
			{\exp(S_n\psi_1(\underline i^+))} %\circ\pi^+)(\underline i))} 
			\leq C , \quad
   	\frac{1}{C} 
	 \leq  \frac{\lvert R^\u_n(Y)\rvert_\h}
			{\exp(S_n\psi_2(\underline i^+))} \leq C ,%\label{e.height_symbolic}
\]%  \end{eqnarray}
where $\underline i^+=\pi^+(\underline i)$ and $\underline i\in\chi^{-1}(\eta)$  
 is an arbitrary preimage of $\eta$. % under the semiconjugacy $\chi$.
\end{lemma}

%%%%%%%%%%%%%%%%%%%%%%%%%%%%%%%%%%%%%%%%%
\subsection{Estimating box dimension from above}
%%%%%%%%%%%%%%%%%%%%%%%%%%%%%%%%%%%%%%%%%

The aim of this section is to prove that $d$ as in Proposition~\ref{pro:localdimension-unstable} provides an upper bound for the box dimension of $\Phi\cap\cW^\u_{\rm loc}(X,T)$ for $X\in\Phi$.
In order to do so, we will first cover the
local unstable manifold $\cW^\u_{\rm loc}(\xi,\tau)$ of the point $\xi\in\Xi$ with $X=(\xi,\Phi(\xi))$ by
a collection of Markov rectangles (or rather Markov intervals) of approximately
the same size $r>0$. This is often referred to as a {\em Moran cover}. The
difficulty is then to estimate the vertical size of the corresponding Markov
rectangle in $\cW^\u_{\rm loc}(X,T)= \cW^\u_{\rm loc}(\xi,\tau)\times I$. 

The main problem is that the level of the Markov
rectangles in the Moran cover is not constant. Indeed, since we require that the
rectangles have approximately the same size, the level of each Markov rectangle depends on the local expansion
and hence on the behavior of the Birkhoff average. Nevertheless, we can divide our Moran cover into collections of
rectangles on which the Birkhoff averages of both $\varphi^\u$ and $\varphi^\cu$
take approximately the same values. Then we consider the symbolic cylinders
corresponding to these rectangles and use 
Corollary~\ref{c.growth-rates} to estimate their maximal
number. Since the Birkhoff average of $\varphi^\cu$ provides us with an estimate
for the height of the rectangles via (\ref{e.width}), we can thus obtain a
bound on the number of squares of side length $r$ that we need to cover the graph
in each rectangle. Summing up over all elements of the Moran cover will then
yield the desired upper bound on $\dim_{\rm B}(\Phi\cap\cW^\u_{\rm loc}(X,T))$.

%---------------------------------------------------------------
\subsubsection{Construction of the Moran covers} \label{subsubsec:Morancover}
%---------------------------------------------------------------

For  a given parameter $r>0$,  which will be the approximate size
 of Markov rectangles in $\cW^\u_{\rm loc}(\xi,\tau)$,  we define an
 appropriate partition $\cC(r)$ of $\Sigma^+$ that we call {\em symbolic Moran
   cover} of parameter $r$ (relative to the potential $\psi_1$).  For every
 $\underline i^+\in\Sigma^+$ let $n(\underline i^+)\ge1$ be the smallest positive
 integer $n$ such that
\[
	 S_n\psi_1(\underline i^+)<\log r.
\]
Note that, as $\psi_1$ is continuous and negative, there exist positive integers
$n_1=n_1(r)$ and $n_2=n_2(r)$ such that $n_1\le n(\underline i^+)\le n_2$ for all
$\underline i^+\in\Sigma^+$. Given $n\in\{n_1,\dots,n_2\}$, let
$\widetilde\cC_n(r)$ denote the family of (disjoint) cylinders $[i_1\ldots i_n]^+$ which
contain an infinite sequence $\underline i^+\in \Sigma^+$ with $n(\underline i^+)=n$. Let $\ell_n=\card\widetilde\cC_n(r)$.

We now define the partition $\cC(r)$ recursively.
We start with index $n=n_1$. Let $\cs_{n_1}=\Sigma^+$ and $\cC_{n_1}(r)=\widetilde\cC_{n_1}(r)$. Assuming that all these objects are already defined for $k=n_1,\ldots,n$ and that $S_n\ne\emptyset$, let\begin{eqnarray*}
  \cs_{n+1} 
  & \eqdef & \cs_n\setminus  \bigcup\{C\colon C\in\widetilde\cC_n(r)\},\\ 
  \cC_{n+1}(r) 
  & \eqdef & \big\{C\colon C\in\widetilde\cC_{n+1}(r),C\subset \cs_{n+1}\big\}.
\end{eqnarray*}
Since $n(\cdot)\le n_2$, we eventually arrive at $\cs_{n^*}=\emptyset$ for some
$n^*\leq n_2$. Then we stop the recursion and define the family
\begin{equation}\label{eq:morann}
	\cC(r)\eqdef
	\{C\colon C\in\cC_n(r),n=n_1,\ldots,n^\ast\}
\end{equation}
which partitions $\Sigma^+$ into pairwise disjoint cylinders. Note that each $\cC_n(r)$ contains exactly those cylinders $C\in\cC(r)$ that have length $n$.

\subsubsection{Cardinality of the Moran covers} \label{subsubsec:cardMoran}
We  fix $\varepsilon>0$. Let $A\eqdef\min_\xi\{\lvert\psi_1(\xi)\rvert,\lvert\psi_2(\xi)\rvert\}$ and $\lVert\psi\rVert\eqdef\max\{\lVert\psi_1\rVert,\lVert\psi_2\rVert\}$. 

For every $\va\in D^+(\psi)$ let $\theta_0(\va)=\theta_0(\va,\varepsilon)>0$ as in  Corollary~\ref{c.growth-rates}. Thus, since $D^+(\psi)$ is compact there is a finite cover with cardinality $m$, for some $m=m(\varepsilon)\ge1$, 
\begin{equation}\label{eq:cover}\begin{split}
	D^+(\psi)
	&\subset\bigcup_{j=1}^m
		Q_j,
		\quad\text{ where }\quad\\
	Q_j&\eqdef\Big(a^j_1-\frac{\theta_j}{2},a^j_1+\frac{\theta_j}{2}\Big)
		\times\Big(a^j_2-\frac{\theta_j}{2},a^j_2+\frac{\theta_j}{2}\Big)
\end{split}\end{equation}
with $\va^j=(a^j_1,a^j_2)\in D^+(\psi)$ and $\theta_j=\theta_0(\va^j)$. 
Define $\widehat\theta=\widehat\theta(\varepsilon)>0$ and $\widetilde\theta=\widetilde\theta(\varepsilon)>0$ by
\[
	\widehat\theta\eqdef\min_{j=1,\ldots,m}\theta_j,\quad
	\widetilde\theta\eqdef\max_{j=1,\ldots,m}\theta_j
\] 
and observe that $\widehat\theta(\varepsilon)\to0$ and $\widetilde\theta(\varepsilon)\to0$ as $\varepsilon\to0$.

Let $\theta\in(0,\widehat\theta)$ be small enough such that
\[%begin{equation}\label{eq:hcoiseKRtheta1}
	\frac{4\theta}{A} \, \lVert\psi\rVert
	< \frac{\widehat\theta}{3}
	%\frac{1}{2\lVert\psi\rVert}}
\]%end{equation}
%\begin{equation}\label{eq:hcoiseKRtheta}
%	1-\frac{2\theta}{A}
%	\le \sqrt{1-\frac{\widehat\theta}{2}\frac{1}{2\lVert\psi\rVert}}
%\end{equation}
and for every index  $j$ choose $n_j=n_0(\va^j,\varepsilon,\theta)\ge1$ as in  Corollary~\ref{c.growth-rates} and let $n_0=\max_jn_j$. 

Recall  that $\log r$ and all $a_\ell+\theta$, $\ell=1,2$, are all negative numbers.
Let
\[
	K(r)
	\eqdef\left\lvert\frac{\log r}{\log r-\lVert\psi\rVert}\right\rvert
\]
and observe that $\lim_{r\to0}K(r)=1$.
Now choose $r=r(\varepsilon,\theta)>0$ small enough to ensure that
\begin{equation} \label{eq:defKthetar}
	\frac{\lvert\log r\rvert}{\lVert\psi\rVert+\theta} > n_0
\end{equation}
and that 
\begin{equation}\label{eq:theta0comparisonfinal1}
	2\lVert\psi\rVert\Big(1-\big(1-\frac{2\theta}{A})K(r)\Big)
	\le \frac{\widehat\theta}{2}.
\end{equation}

For this choice of $r$ we now consider  the Moran cover $\cC(r)$ of order $r$  as in Section~\ref{subsubsec:Morancover}. Recall that a Moran cover consists of pairwise disjoint cylinders of variable length, see~\eqref{eq:morann}. For every index  $j$ in the cover~\eqref{eq:cover} let
\[ \begin{split}
	\cC(r,\theta_j,j) 
	\eqdef \Big\{[i_1\ldots i_n]^+&\in\cC(r) \colon \\
	\exists\,
        \underline i^+ \in [i_1\ldots i_n]^+ %\colon n(\underline i^+)=n,
       & \text{ with }n(\underline i^+)=n,
       \Big( \frac1nS_n\psi_1(\underline i^+),\frac1nS_n\psi_2(\underline i^+)\Big)
       	\in Q_j\Big\},
\end{split}\]
that is, we select only a certain number of cylinders from the Moran cover. 
Note that  for different pairs  this selection is not necessarily disjoint, but this will not matter for our purposes. We only need an upper bound on their cardinality stated in Claim~\ref{cla:atmos} below.

For that end, for every $j$ let 
 \[
    n_j
    \eqdef \left\lceil\frac{\log r - \lVert\psi_1\rVert}{a_1^j+\theta_j}\right\rceil .
\]
Given $[i_1\ldots i_n]^+\in\cC(r,\theta_j,j)$ and $\underline i^+\in [i_1\ldots i_n]^+$
with  $(\frac1nS_n\psi_1(\underline i^+),\frac1n S_n\psi_2(\underline i^+))\in Q_j$, by definition we have
\begin{equation}\label{eq:estimatea}
	n(a_1^j-\theta_j)
	\le S_n\psi_1(\underline i^+)
	<n(a_1^j+\theta_j),\quad
	n(a_2^j-\theta_j)
	\le S_n\psi_2(\underline i^+)
	< n(a_2^j+\theta_j).
\end{equation}
Since $n(\underline i^+)=n$, 
\[
	S_n\psi_1(\underline i^+)<\log r,\quad
	S_{n-1}\psi_1(\underline i^+)\ge\log r.
\]
Further (recalling that $\psi_1<0$) note that
\[	
	\log r \le S_{n-1}\psi_1(\underline i^+) 
	\le S_n\psi_1(\underline i^+) +  \lVert\psi_1\rVert 
	<n(a_1^j+\theta_j) + \lVert\psi_1\rVert.
\]
Hence, with~\eqref{eq:defKthetar} and~\eqref{eq:estimatea} (recall that $a_1^j+\theta<0$) we obtain 
\begin{equation}\label{eq:waersc}
	n_0
	<\frac{\lvert\log r\rvert}{\lVert\psi_1\rVert+\theta}
	\le \left\lvert\frac{\log r}{ a_1^j-\theta_j}\right\rvert 
	\le n 
	\le \left\lvert\frac{\log r-\lVert\psi_1\rVert}{a_1^j+\theta_j}\right\rvert 
	\le n_j.
	%= N_i
	%\le n_i.
\end{equation}

We now state the following estimate.

\begin{claim}\label{cla:atmos}
With the above choice of quantifiers, for every index pair $jk$ we have 
\[
	\card\cC(r,\theta_j,j)
	\le \exp(n_j(\ch(\psi,\va^j)+\varepsilon)),
\] 
where $\ch(\psi,\va^j)$ is defined as in Proposition~\ref{p.joint-limits}. 
\end{claim}

\begin{proof} 
Observe that by our choice of quantifiers, by Corollary~\ref{c.growth-rates} for every index  $j$  there are at most $\exp(n_j(\ch(\psi,\va^j)+\varepsilon))$ cylinders of length $n_j$ which contain a sequence with finite  Birkhoff averages of $\psi_\ell$ of level $n_j$ being roughly $a^j_\ell\pm\theta_j$, for $\ell=1,2$, respectively. Hence, if all cylinders in $\cC(r,\theta_j,j)$ would have length $n_j$, the claim would follow immediately (since $n_j\ge n_0$ by~\eqref{eq:waersc}). 
As this is in general not the case, we need to relate the  cylinders in $\cC(r,\theta_j,j)$ to those of length $n_j$. 

To that end,
  let $C\in \cC(r,\theta_j,j)$ be of length $n$ and
  choose some $\underline i^+\in C$ with $n(\underline i^+)=n$ and $(\frac1n S_n\psi_1(\underline i^+),\frac1n S_n\psi_2(\underline i^+))\in
  Q_j$ as above. We have (recall again $n_j\ge n$)
  \[\begin{split}
    \Big\lvert \frac{1}{n_j} S_{n_j}\psi_1(\underline i^+) 
    &- \frac1n S_n\psi_1(\underline i^+)\Big\rvert \\
	&\leq  \frac{1}{n_j} \left|S_{n_j}\psi_1(\underline i^+)-S_n\psi_1(\underline i^+)\right| 
				+\left|\left(\frac{1}{n_j}-\frac1n\right)S_n\psi_1(\underline i^+)\right| \\ 
	&\leq \frac{ n_j-n}{n_j} \lVert\psi_1\rVert + \left|\frac{1}{n_j}-\frac1n\right|
    n\|\psi_1\| \ \leq \ 2 \|\psi_1\| \left(1-\frac{n}{n_j}\right) .%\hspace{7eM}
  \end{split}\]
The analogous estimates can be obtained for $\psi_2$. Using (\ref{eq:waersc}), this yields
\begin{eqnarray*}
  1-\frac{n}{n_j} 
  &\leq & 1 
  	- \frac{\lvert a_1^j+\theta_j\rvert}{\lvert a_1^j-\theta_j\rvert} \cdot 
		\left\lvert\frac{\log r}{\log r - \|\psi_1\|}\right\rvert\\
	&=& 1-\Big(1-\frac{2\theta_j}{	\lvert a_1^j-\theta_j\rvert}\Big)\cdot 
		\left\lvert\frac{\log r}{\log r - \|\psi_1\|}\right\rvert 
	\le 1- \Big(1-\frac{2\theta}{A}\Big)	\cdot K(r).
\end{eqnarray*}
By~\eqref{eq:estimatea} and~\eqref{eq:theta0comparisonfinal1} we obtain 
\[
	\Big\lvert \frac{1}{n_j} S_{n_j}\psi_1(\underline i^+) 
   	- a^j_1\Big\rvert 
	\le \frac{\theta_j}{2}+\frac{\widehat\theta}{2}
	\le \theta_j.
\]
The analogous estimates are true for $\psi_2$.
Hence,  every cylinder in $\cC(r,\theta_j,j)$ contains a sequence with finite Birkhoff averages of $\psi_\ell$ of level $n_j$ being roughly $a^j_\ell\pm\theta_j$, $\ell=1,2$.
But by the above there are at most  $\exp(n_j(\ch(\psi,\va^j)+\varepsilon))$ such $n_j$-level cylinders and hence at most that number of cylinders in $\cC(r,\theta_j,j)$
as claimed.
\end{proof}

\subsubsection{Final estimates} \label{subsubsec:finallower}
Recall that in Section~\ref{subsubsec:cardMoran} we fixed $\varepsilon>0$ and then in~\eqref{eq:cover} chose a finite cover $\{Q_j\}_{j=1}^m$ of $D^+(\psi)$ by sufficiently small open squares of sizes $\theta_j$ bounded between $\widehat\theta(\varepsilon)=\min_{j}$ and $\widetilde\theta(\varepsilon)=\max_{j}\theta_j$ with $\widehat\theta(\varepsilon)\to0,\widetilde\theta(\varepsilon)\to0$ as $\varepsilon\to0$. Then  for sufficiently small $r>0$ depending on those choices we verified Claim \ref{cla:atmos}.

 Consider now the Moran cover $\cC(r)$ of $\Sigma^+$ for such $r$. Observe that, by construction
\[
	\Sigma^+
	= \bigcup_{j=1,\ldots,m}\bigcup\{C\colon C\in\cC(r,\theta_j,j)\}.
\]
%Moreover, as the cylinder sets in $\cC(r)$ cover all of $\Sigma^+$,
For any index  $j$, a cylinder $[i_1\ldots i_n]^+\in\cC(r,\theta_j,j)$ projects to a Markov unstable rectangle $\underline R^\u_n(\eta)$, where $\eta\in\chi([i_1\ldots
  i_n]^+)$. Hence, the collection of the corresponding Markov rectangles $R^\u_n(Y)$, $Y=(\eta,\Phi(\eta))$, 
  forms a cover of $R^\u(X)\subset\Phi\cap\cW^\u_\mathrm{loc}(X,T)$. For the width and height of
these rectangles, Lemma~\ref{lem:Markovunstabpot} yields the following estimates
\[
%	\le c^{-1}e^{S_n\psi_1(\eta)}
	 \lvert  R^\u_n(Y)\rvert_\w 
	\le Ce^{S_n\psi_1(\underline i^+)}
	\le Cr
	\quad\text{ and }\quad
	\lvert  R^\u_n(Y)\rvert_\h 
	\le Ce^{S_n\psi_2(\underline i^+)}
	\le Ce^{n(a_2^j+\theta_j)}.
\]
Hence, we can cover $R^\u_n(Y)$ by at most $\exp(n(a_2^j+\theta_j))/r$ balls of
radius $Cr$.  Using Claim~\ref{cla:atmos}, this implies that we can cover the
union of all Markov rectangles in the Moran cover, and hence all of
$R^\u(X)$, %\Phi\cap\cW^\u_\mathrm{loc}(X,T)$, 
by at most $N(Cr) $  balls of radius $Cr$, where
\[\begin{split}
	N(Cr) 
	& \eqdef \sum_{j=1}^m\exp( n_j(\ch(\va^j)+\varepsilon)) \cdot
        \frac{\exp(n(a_2^j+\theta_{j}))}{r}\\ 
        &\le m\cdot\frac1r\cdot \max_{j=1,\ldots,m}
        		\exp\Big(n_j(\ch(\va^j)+\varepsilon)+n(a_2^j+\theta_{j})\Big).
\end{split}\]
%(Here $(m+1)^2$ is the number of squares in the $2\theta$-grid $(Q_{jk})_{j,k=0}^m$ defined above.)  
Thus, %
\[\begin{split}
  \overline\dim_{\rm B}(R^\u(X))%\Phi\cap\cW^\u_{\rm loc}(X,T)) 
  & \leq  \limsup_{r\to 0} \frac{\log    N(Cr)}{-\log r} \\ 
  &\leq 1 + \limsup_{r\to0} 
  	\max_{j=1,\ldots,m}\left( \frac{n_j}{-\log r}(\ch(\psi,\va^j)+\varepsilon) 
						+ \frac{n}{-\log r}(a_2^j+\widetilde\theta)\right).
						%\\ 
%\text{by definition of }n_j
%&\le 	1 + \limsup_{r\to0} 
%  	\max_{j,k=0,\ldots,m}\left( \frac{\log r-\lVert\psi_1\rVert}{-\log r(a_j+\theta_{jk})}(\ch(\psi,\va_{jk})+\varepsilon) 
%						+ \frac{n}{-\log r}(a_k+\theta)\right)\\ 					
 % \text{by }\eqref{eq:waersc}
\end{split}\]
By the definition of $n_j$ and the relation for $n$ in~\eqref{eq:waersc} we hence obtain 
\[
	\overline\dim_{\rm B}(R^\u(X))
	\le 1+ \max_{j=1,\ldots,m}
		\left(\frac{-1}{a_1^j+\widehat\theta}(\ch(\psi,\va^j)+\varepsilon)
			+\frac{-1}{a_1^j-\widetilde\theta}(a_2^j+\widetilde\theta)\right).
\]
Taking the limits $\widehat\theta,\widetilde\theta\to0$ and $\varepsilon \to 0$, we finally get
\[
  \overline\dim_{\rm B}(R^\u(X))%\Phi\cap\cW^\u_{\rm loc}(X,T))  
  \leq 1+ \max_{\va=(a_1,a_2)\in D^+(\psi)}\frac{\ch(\psi,\va)+a_2}{-a_1} .
\]

By~\eqref{eq:deftfinal} together with~\eqref{eq:equalityallpressure} the right hand side is the unique number $d$ with $P_{\sigma^+}(\psi_2+(d-1)\psi_1)=0=P_{\tau|_\Xi}(\varphi^\cu+(d-1)\varphi^\u)$. 
%\end{proof}
This finishes the proof of the upper bound for the upper box dimension of $R^\u(X)\subset \Phi\cap\cW^\u_{\rm loc}(X,T)$. As this estimate holds for every $X\in\Phi$, the upper bound for the box dimension of $ \Phi\cap\cW^\u_{\rm loc}(X,T)$  in Proposition~\ref{pro:localdimension-unstable} follows. 

%%%%%%%%%%%%%%%%%%%%%%%%%%%%%%%%%%%%%%%%%
\subsection{Estimating box dimension from below}
%%%%%%%%%%%%%%%%%%%%%%%%%%%%%%%%%%%%%%%%%
The lower estimate for the box dimension is now easier, since we can restrict to just one vector $\va\in D^+(\psi)$.

\begin{lemma}\label{lem:heresnow}
	For every $\va=(a_1,a_2)\in D^+(\psi)$ and every $X\in\Phi$ we have
\[
	\underline\dim_{\rm B}\left(\Phi\cap\cW^\u_\mathrm{loc}(X,T)\right)
	\ge \frac{\ch(\psi,\va)+a_2}{-a_1}+1 .
\]	
\end{lemma} 

\begin{proof} 
We fix $\varepsilon>0$ and choose $\theta=\theta_0(\va,\varepsilon)$ and $n_0=n_0(\va,\varepsilon,\theta)$ as in Corollary~\ref{c.growth-rates}.
Then, given $n\geq n_0$, there exist at least $\exp(n(\ch(\psi,\va)-\varepsilon))$ (mutually
disjoint) cylinders of length $n$ containing some $\underline i^+\in\Sigma^+$
with $\lvert\frac1nS_n\psi_\ell(\underline i^+)-a_\ell\rvert\le\theta$, $\ell=1,2$. Each such cylinder 
corresponds to an unstable Markov rectangle $\underline R^\u_n(\eta)$, where
$\eta$ is in the preimage under $\chi$ of some sequence $\underline i$ with
$\pi^+(\underline i)=\underline i^+$. We label these points $\eta$ by
$\xi_1,\ldots,\xi_m$, where $m\geq \exp(n(\ch(\psi,\va)-\varepsilon))$, and obtain $m$
Markov rectangles $\underline R^\u_n(\xi_1),\ldots,\underline R^\u_n(\xi_m)$ with mutually disjoint interiors. Let $X_k=(\xi_k,\Phi(\xi_k))$.

\begin{claim}\label{cl:claimkey} 
There exist constants $C,\widetilde C>1$ (independent of $\varepsilon,\va,\theta,n$) such that for every $k=1,\ldots,m$, at least $\widetilde Ce^{n(a_2-a_1)}$
%\[
%	\frac{\lvert R^\u_n(X_k)\rvert_\h}{\lvert R^\u_n(X_k)\rvert_\w}
%\]	
squares of diameter $C^{-1}e^{n(a_1-\theta)}$ each are needed to cover
$\Phi\cap R^\u_n(X_k)$. 
\end{claim}

We now distinguish two different cases corresponding to either i) the hypotheses of Theorems~\ref{t.anosov} or ~\ref{the:one-dimensional} or ii) the hypotheses of
Theorem~\ref{the:1}.

\begin{proof}[Proof of Claim~\ref{cl:claimkey} in case i)]
If either $\Xi=M$ or if $\Xi$ is a one-dimensional attractor, then observe that for every $\xi\in\Xi$, the Markov unstable rectangle $\underline R^\u(\xi)$ is a closed curve contained in the local unstable manifold of $\xi$ and with $X=(\xi,\Phi(\xi))$ the Markov unstable rectangle $R^\u(X)$ is the graph of a continuous closed curve. By Lemma~\ref{lem:Markovunstabpot}, for every $X_k\in R^\u(X)$, the widths and heights of the Markov unstable rectangles of level $n$ can be estimated by 
\[
%	\le c^{-1}e^{S_n\psi_1(\eta)}
	 \lvert  R^\u_n(X_k)\rvert_\w 
	\ge C^{-1}e^{n(a_1-\theta)}
	\quad\text{ and }\quad
	\lvert  R^\u_n(X_k)\rvert_\h 
	\ge C^{-1}e^{n(a_2-\theta)},
\]
respectively. 
As the graph stretches fully over the entire rectangle (there are no gaps since there are no gaps in its projection to the base $\cW^\u_{\rm loc}(\xi,\tau)$), subdividing we yield that for each $R^\u_n(X_k)$ there are at least $C^{-1}e^{n(a_2-\theta)}/(C^{-1}e^{n(a_1-\theta)})$ squares with pairwise disjoint interior of size $C^{-1}e^{n(a_1-\theta)}$ which each contain a point in the rectangle. 
\end{proof}

Note that the crucial argument in the above proof is its very last sentence. In the case of a general Cantor set $\Xi$ this argument does not anymore apply (and in fact would overestimate the number of elements used to cover the fractal graph $\Phi$).

\begin{proof}[Proof of Claim~\ref{cl:claimkey} in case ii)]
For $\Phi$ being a fibered blender with the germ property, by item (c) of the Germ property there exists $\delta>0$ such that for every $X_k=(\xi_k,\Phi(\xi_k))$ and for every $n\ge1$ the $n$th level $\u$-box of $\xi_k$,
\[
	B^\u_n(\xi_k)
	=R^\u(\xi_k,n)\times I,
\]
where $R^\u(\xi_k,n)$ denoted the minimal curve containing $\underline R^\u_n(\xi_k)$,
contains a set $\fR$ of the form
\[
	\fR=\bigcup_{Z\in J_n}\hat D_Z
\]
which is continuously foliated by a family  $\{\hat D_Z\colon Z\in J_n\}$ of germ plaques such that $\lvert T^n(\fR)\rvert_\h\ge\delta$. Moreover, by Corollary~\ref{c.germcorollary}, each such germ plaque contains a point of $\Phi$.
Applying the same arguments as in the proof of Proposition~\ref{pro:dis}, we get the following estimate for the width and height of $\fR$:
\[
	\lvert \fR\rvert_\w\ge C^{-1}e^{n(a_1-\theta)}
	\quad\text{ and }\quad
	\lvert \fR\rvert_\h\ge e^{n(a_2-\theta)}\cdot D^{-2}\cdot\delta,
\]
where $D$ is the distortion constant as in this proof. Hence, to cover $\Phi\cap \fR^\u_n(X_k)$, we need at least $D^{-2}\delta e^{n(a_2-\theta)}/(C^{-1}e^{n(a_1-\theta)}$ squares, each of size $C^{-1}e^{n(a_1-\theta)}$.
\end{proof}

With Claim~\ref{cl:claimkey},  we need at least $N(\va,\varepsilon,\theta,n)$ squares to cover $R^\u(X)$ by squares of size $C^{-1}e^{n(a_1-\theta)}$, where
\[%\begin{split}
	N(\va,\varepsilon,\theta,n)
	\eqdef
	e^{n(\ch(\psi,\va)-\varepsilon)}
	%	\frac{\widetilde Ce^{n(a_2-\theta)}}{C^{-1}e^{n(a_1-\theta)}}.
	\widetilde Ce^{n(a_2-a_1)}.
\]%\end{split}\]
Thus,
\[
	\underline\dim_{\rm B}(R^\u(X))
	\ge \liminf_{n\to0}\frac{\log N(\va,\varepsilon,\theta,n)}{-\log (C^{-1}e^{n(a_1-\theta)})}
	\ge \frac{\ch(\psi,\va)-\varepsilon+a_2-a_1}{-a_1+\theta}.
\]
Letting $\theta\to0$ and then $\varepsilon\to0$ proves the lemma.
\end{proof}

As $\va\in D^+(\psi)$ was arbitrary in Lemma~\ref{lem:heresnow}, with the same observations as at the end of Section~\ref{subsubsec:finallower}, we have shown that $d$ as in Proposition~\ref{pro:localdimension-unstable} provides a lower bound for the box dimension of $\Phi\cap\cW^\u_{\rm loc}(X,T)$. This finishes the proof of Proposition~\ref{pro:localdimension-unstable}.

%----------------------------------------------------------------------------------------------------
\section{Regular (local) product structure and proof of the main theorems}\label{sec:locpro}
%----------------------------------------------------------------------------------------------------

The formula for the box dimension of the full set $\Phi$ is based on two crucial facts. First,  the dimension of direct products of sets we recalled in Section~\ref{sec:defdim}. Second, we rely on the fact that the local product structure (locally) enables to describe $\Phi$ as a direct product up to a ``sufficiently regular change of coordinates''.
To be more precise, given $X\in\Phi$ let $Y, Z\in\Phi$ 
%\in \Phi\cap \cW^\cu_{\rm loc}(X,T)$ and $X'\in\Phi\cap\cW^\s_{\rm loc}(X,T)$ 
be points sufficiently close to $X$. Recall that by~\eqref{eq:locpro} the local stable manifold of $Y$ intersects the local unstable manifold of $Z$ in a point  in $\Phi$,
\[
	[Y,Z]
	\eqdef \cW^\s_{\rm loc}(Y,T)\cap\cW^\u_{\rm loc}(Z,T) 
	 \in\Phi
\]
(compare Figure~\ref{fig.5}). 
\begin{figure}
\begin{overpic}[scale=.45]{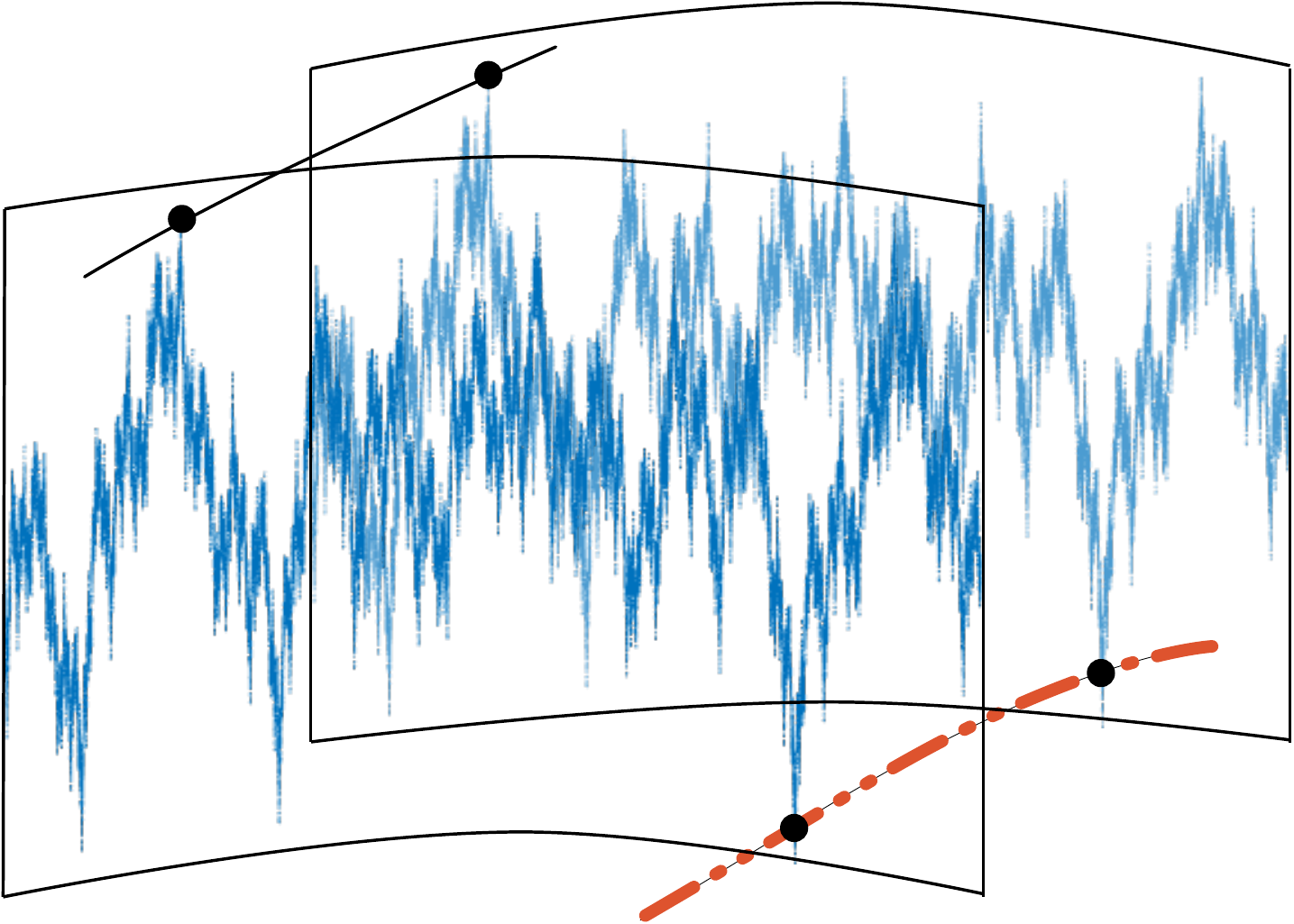}
        \put(85,14){\small $Z$}
        \put(61,2){\small $X$}
        \put(15,49.5){\small $Y$}
        \put(27,0){\small $\cW^\s_{\rm loc}(X,T)$}
        \put(80,1){\small $\cW^\u_{\rm loc}(X,T)$}
        \put(101,12){\small $\cW^\u_{\rm loc}(Z,T)$}
 \end{overpic}
\caption{Local stable holonomy $h^\s_{X,Z}$}
\label{fig.5}
\end{figure}
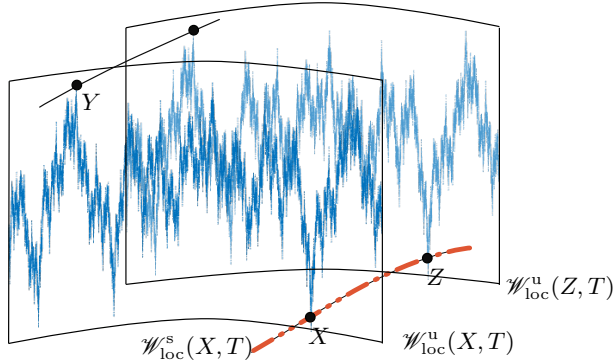
Fixing a point $X\in\Phi$, $\Phi\cap\cW^\u_{\rm loc}(X,T)$ lies in a small two-dimensional disk which is transverse to the stable lamination.
Hence, taking $Z\in\cW^\s_{\rm loc}(X,T)$ and varying $Y\in \Phi\cap\cW^\u_{\rm loc}(X,T)$, we obtain that locally $\Phi\cap\cW^\u_{\rm loc}(Z,T)$ is obtained as the image of $\Phi\cap\cW^\u_{\rm loc}(X,T)$ under the \emph{local stable holonomy} $h_{X,Z}^\s\colon\Phi\cap\cW^\u_{\rm loc}(X,T)\to\Phi\cap\cW^\u_{\rm loc}(Z,T)$ obtained by sliding   $Y$ along its local stable manifold to $h_{X,Z}^\s(Y)$. 
Analogously, one defines the \emph{local unstable holonomy} $h_{X,Y}^\u\colon \Phi\cap\cW^\s_{\rm loc}(X,T)\to\Phi\cap\cW^\s_{\rm loc}(Y,T)$ obtained by sliding along local unstable manifolds. 

\begin{remark}[H\"older holonomies]\label{rem:pinchemadre}
Since the stable subbundle is one-dimensional, the local unstable holonomies $h^\u_{X,Y}$ and their inverses are both H\"older continuous with H\"older exponent $\theta$ arbitrarily close to $1$ (see~\cite[Theorem C]{PalVia:88}). 

Note that the Pinching hypothesis
\[
	\kappa_\s\mu_\w\le\lambda_\w
\] 
is merely~\eqref{eq:constants} with the factor $\kappa_\s>1$ included.  In particular it implies that for any $\theta\in(0,1)$ we have $\kappa_\s^\theta\mu_\w<\lambda_\w$. Hence, assuming also that $T$ is $C^2$, then by~\cite[Theorems A and A$'$]{PugShuWil:97} (see also the more detailed results in~\cite[Section 4]{PugShuWil:12}) we have that the local stable holonomies $h^\s_{X,Z}$ together with their inverses are both H\"older with  any 
H\"older exponent $\theta\in(0,1)$. 

The hypothesis that $T$ is $C^2$ has been relaxed to $C^{1+\varepsilon}$ assuming $\kappa_\s^\theta\mu_\w<\mu_\w^\theta\lambda_\w$ (see~\cite{Wil:13} for details). Note also precursors of  results of this type  in~\cite{Has:97} if $\theta>1$ and in~\cite{SchSig:92}.
\end{remark}

\begin{proof}[Proof of Theorems~\ref{t.anosov},~\ref{the:one-dimensional}, and~\ref{the:1}]
Given $X\in\Phi$, consider the product space 
\[
	A_X
	\eqdef (\Phi\cap\cW^\s_{\rm loc}(X,T))\times (\Phi\cap\cW^\u_{\rm loc}(X,T)).
\]	 
By Proposition~\ref{pro:localdimension-stable} we have $\dim_{\rm B}(\Phi\cap\cW^\s_{\rm loc}(X,T))=d^\s$. By Proposition~\ref{pro:localdimension-unstable} we have $\dim_{\rm B}(\Phi\cap\cW^\u_{\rm loc}(X,T))=d$. 
Then, by~\eqref{eq:product} we have $\dim_{\rm B} (A_X)=d^\s+d$.

It remains to show that  the direct product $A_X$ has the same dimension as $\Phi$.
For that we follow the arguments in~\cite{PalVia:88}. We consider $h_X\colon A_X\to\Phi$  given by $h_X(Y,Z)\eqdef [Y,Z]$ which is a homeomorphism of $A_X$ onto a neighborhood of $X$ in $\Phi$. We will show that $h_X$ and $h_X^{-1}$ both are H\"older continuous with H\"older exponent $\theta$ arbitrarily close to $1$. Hence, for $V_X=h_X(A_X)$ we will conclude $\dim_{\rm B}( V_X)\in[\theta,\theta^{-1}]\dim_{\rm B}( A_X)$ for every $\theta\in(0,1)$ and thus $\dim_{\rm B}( V_X)=d^\s+d$. 
Given $W_1,W_2\in A_X$, let $W_i=h_X(Y_i,Z_i)=[Y_i,Z_i]$, $i=1,2$.
Consider the auxiliary point $W=[Y_2,Z_1]$ and observe  that $W\in \cW^\u_{\rm loc}(W_1,T)$ and $W\in\cW^\s_{\rm loc}(W_2,T)$.  By H\"older continuity of the holonomies, there is some positive constant $C_\theta$ such that
\[\begin{split}
	d(W_1,W_2)
	&\le \varrho^\u(W_1,W) + \varrho^\s(W,W_2)
	\le C_\theta \varrho^\u(Y_1,Y_2)^\theta+C_\theta \varrho^\s(Z_1,Z_2)^\theta\\
	&\le 2C_\theta (\max\{ \varrho^\u(Y_1,Y_2),\varrho^\s(Z_1,Z_2)\})^\theta,
\end{split}\]
where $\varrho^\s,\varrho^\u$ denote the induced distances in the stable and unstable local manifolds, respectively.
Hence $h_X$ is $\theta$-H\"older. The proof that $h_X^{-1}$ is $\theta$-H\"older is analogous.

Finally, since $\{V_X\}_{X\in\Phi}$ is an open cover of the compact set $\Phi$ each having the same dimension $d^\s+d$, we can select a finite subcover and by stability of box dimension we obtain the claimed property $\dim_{\rm B}(\Phi)=d^\s+d$.
This finishes the proof.
\end{proof}

\bibliographystyle{amsplain}

\end{document}